\documentclass[11pt]{amsart}
\usepackage[utf8]{inputenc}

\usepackage{amssymb,amsthm,amsmath,amstext,amscd}
\usepackage{eucal}
\usepackage{epsfig}
\usepackage{mathrsfs}
\usepackage{bm}
\usepackage[utf8]{inputenc}
\usepackage{comment}
\usepackage[all]{xy}
\usepackage{dynkin-diagrams}
\usepackage[margin = 3cm]{geometry}
\usepackage{marginnote}

\theoremstyle{plain}
\newtheorem{theorem}{Theorem}[section]
\newtheorem*{conjecture*}{Conjecture}
\newtheorem{prop}[theorem]{Proposition}
\newtheorem{lemma}[theorem]{Lemma}
\newtheorem{coro}[theorem]{Corollary}
\theoremstyle{definition}
\newtheorem{remark}[theorem]{Remark}

\def\AA{{\mathbb{A}}}
\def\CC{{\mathbb{C}}}
\def\GG{{\mathbb{G}}}
\def\FF{{\mathbb{F}}}

\def\PP{{\mathbb{P}}}
\def\QQ{{\mathbb{Q}}}\def\ZZ{{\mathbb{Z}}}

\def\cD{{\mathcal{D}}}

\def\cO{{\mathcal{O}}}
\def\cE{{\mathcal{E}}}

\def\cI{{\mathcal{I}}}

\def\cN{{\mathcal{N}}}
\def\cR{{\mathcal{R}}}
\def\cS{{\mathcal{S}}}\def\cU{{\mathcal{U}}}
\def\cC{{\mathcal{C}}}\def\cQ{{\mathcal{Q}}}

\def\cW{{\mathcal{W}}}
\def\ra{{\rightarrow}}
\def\lra{{\longrightarrow}}
\def\ft{{\mathfrak t}}\def\fs{{\mathfrak s}}
\def\fc{{\mathfrak c}}\def\fe{{\mathfrak e}}
\def\fg{{\mathfrak g}}\def\fh{{\mathfrak h}}
\def\fso{\mathfrak{so}}\def\fz{\mathfrak{z}}\def\fl{\mathfrak{l}}

\def\fsl{\mathfrak{sl}}
\def\fgl{\mathfrak{gl}}

\def\CR4{{\mathrm{CR}_4}}
\def\S3{{\mathrm{S}_3}}

\DeclareMathOperator{\Spin}{Spin}

\def\lau#1{\textcolor{blue}{{ Lau:} #1 { }}}


\subjclass[2020]{}

\usepackage{hyperref}

\title{On linear sections of the spinor tenfold II}
\author[Y. Liu]{Yingqi Liu}
\address{AMSS, Chinese Academy of Sciences, 55 ZhongGuanCun East Road, Beijing, 100190, China and
University of Chinese Academy of Sciences, Beijing, China}
\email{liuyingqi@amss.ac.cn}
\author[L. Manivel]{Laurent Manivel}
\address{Toulouse Mathematics Institute, CNRS/Toulouse University,
118 route de Narbonne, 
F-31062 Toulouse Cedex~9, France}
\email{manivel@math.cnrs.fr}
\date{}

\begin{document}

\medskip 

\begin{abstract}
Following previous work by A. Kuznetsov, we study the Fano manifolds obtained as linear sections of the spinor tenfold in $\PP^{15}$. Up to codimension three there are finitely many such sections, up to projective equivalence. In codimension four there are three moduli, and this family is particularly interesting because of its relationship with Kummer surfaces on the one hand, and a grading of the exceptional Lie algebra $\fe_8$ on the other hand. 
We show how the two approaches are intertwined, and we prove that codimension four sections of the spinor tenfolds and Kummer surfaces have the very same GIT moduli space. The Lie theoretic viewpoint provides a wealth of additional information. In particular we locate and study the unique section admitting an action of $SL_2\times SL_2$; similarly to the Mukai-Umemura variety in the family of prime Fano threefold of genus $12$, it is a compactification of a 
quotient by a finite group. 
\end{abstract}
\maketitle

\bigskip

\section{Introduction}

This article is a sequel to the paper {\it On linear sections of the spinor tenfold I}, by A. Kuznetsov \cite{kuz_spin}. The title suggested that a second part would soon follow, but it didn't. 
Motivated by our recent series of works around geometric aspects of $\theta$-representations (see e.g. \cite{coble-quadric}), we 
decided to write this sequel, but with a stronger Lie-theoretic 
flavour, for reasons we will soon explain. 

\smallskip
The Lie group $Spin_{10}$ has two  half-spin representations $\Delta_+$ and $\Delta_-$ of dimension $16$, dual one to each other. 
The spinor tenfold is the closed orbit  inside the projectivisation $\PP(\Delta_+)$; 
it is a Fano variety of index $8$, with very nice geometric properties, similar to those of the Grassmannian $G(2,5)$. By results of Mukai \cite{mukai-curves}, any prime Fano threefold, any general polarized K3 surface, any general curve of genus $7$ can be obtained as a linear section of the spinor tenfold. 

\smallskip
As noticed by Kuznetsov (see also \cite{baifuman}), linear sections of small codimension also have nice properties. For example, up to codimension three they are only finitely many, up to projective equivalence: only $1$ in codimension one, $2$ in codimension two, $4$
in codimension three. In codimension two there is a general type 
and a special type of smooth sections, and it was one of the main observations in \cite{kuz_spin} that the special ones are parametrized by a quadratic section $\cR$ of the Grassmannian $G(2,\Delta_-)$, that he coined the  {\it spinor quadratic complex of lines}. We observe that any codimension three section defines a projective plane in $G(2,\Delta_-)$, on which the spinor complex cuts a conic; this conic degenerates in codimension one and this defines a quartic hypersurface $\cR_3$ in $G(3,\Delta_-)$, that we are legitimate to call the {\it spinor quadratic complex of planes}. Its most remarkable property is highlighted when we go to the next case: any codimension four linear section of  the spinor tenfold defines a $\PP^3$ in $G(3,\Delta_-)$, and its general intersection with $\cR_3$, by Proposition \ref{r3}, is a Kummer surface! This is not so surprising when one remembers F. Klein's construction of Kummer surfaces from quadratic complexes of lines, but we will see that the connection is remarkably tight. 

\smallskip
As a matter of fact, codimension four sections have three moduli, and 
their geometry is much richer than for sections of bigger dimensions. In order to understand their family, a marvelous tool will enter the game: 
Vinberg's theory of graded Lie algebras. Indeed, it turns out that the exceptional Lie algebra $\fe_8$ admits a cyclic grading with one component isomorphic to $\CC^4\otimes\Delta_-$ (which is clearly closely related to $G(4,\Delta_-)$). This has extremely strong consequences, starting with the possibility to get a Jordan type classification of orbits.
The effective classification was carried out by W. de Graaf \cite{degraaf}; the fact that the nilpotent cone contains no less than $145$ orbits is a clear indication of its complexity. As a by-product, one gets a classification of orbits in $G(4,\Delta_-)$, hence of codimension four sections of the spinor tenfold, from which we extracted a classification of smooth sections. The fact that some of them are defined by nilpotent elements of  $\CC^4\otimes\Delta_-$, or more generally, elements which are not semisimple, indicates that the moduli stack of sections is very far from being separated, something that should not come as a surprise. 

\smallskip
The situation is much more pleasant when one restricts to semisimple elements, which are parametrized by the projectivization of a so-called Cartan subspace $\fc\subset \CC^4\otimes\Delta_-$, of dimension four in this case. Cartan subspaces are all conjugate and we provide one explicitely. It admits an action of a complex reflexion group $W_\fc$, of order $46080$, and the quotient of $\PP(\fc)$ by $W_\fc$ is a three-dimensional weighted projective space that provides a reasonable GIT moduli space for codimension four sections of the spinor tenfold. The $60$ reflection hyperplanes in $\fc$ and their intersections define special sections; the singular ones are parametrized by $30$ lines (Proposition \ref{30lines}). Moreover,
outside the union of these lines we locate two orbits of special points 
corresponding to special sections with automorphism groups locally isomorphic to $GL_2$ and $SL_2\times SL_2$, respectively. This corresponds to two exceptional, isolated members in the family, for which we provide explicit models. In particular, by Proposition \ref{openorbit} the second one is a compactification $SL_2\times SL_2/\mu_{10}$, and appears very similar to the famous Mukai-Umemura prime Fano threefold of 
genus $12$, with its quasi-homogeneous $SL_2$-action and isolated in the family. 

\smallskip
The relationship with Kummer surfaces can be observed at the classical level: there exist 
striking connexions between the complex reflection group $W_\fc$ (which is  the group denoted $G_{31}$ in the Shephard-Todd classification \cite{st}), the Weyl group $W(D_5)$ of $Spin_{10}$, the automorphism group $W_{Kum}$ of the Kummer configuration of singular points on the surface, and another group $N$, of the same order, that plays a prominent role in the classical study of Kummers. Indeed, one classically starts with a generic point in $\PP^3$,    and generates sixteen points by applying a certain abelian group $F_0$; the group $N$ is the normalizer of this $F_0$ in $PGL_3$ and governs the moduli space of Kummer configurations in $\PP^3$ \cite{gd}.  We show that $N$ is nothing else than the quotient of $W_\fc$ by its center (Proposition \ref{F0}). We  also observe that the $16$ weights of the half-spin representations $\Delta_+$ and $\Delta_-$ allow to recover the $16_6$ Kummer configuration rather easily. This implies that $W(D_5)$ embeds, as a subgroup of index $6$, inside the automorphism group $W_{Kum}$; in fact this construction can be upgraded to $D_6$ and it appears that $W_{Kum}\simeq W(D_6)/\{\pm 1\}$ (Proposition \ref{wd6}).

\smallskip
These group-theoretical considerations are confirmed by natural geometric constructions, and we find it very remarkable that a substantial part of the classical theory follows. 
From a point of the Cartan subspace $\fc$, we can first construct a quadratic section of $G(2,4)$. In fact we get get a four dimensional linear system $\PP(U_5)$ of such sections, parametrized by quartics on $\fc$. Proposition \ref{CR4} yields:

\smallskip\noindent {\bf Proposition.}
{\it The resulting map $\Theta: \PP(\fc)\lra \PP(U_5)$ is a finite morphism of degree $16$, whose image is the Castelnuovo-Richmond, or Igusa quartic $\CR4$.}

\smallskip Moreover, quartics in $U_5$ are spanned by products of four equations of reflection hyperplanes; there are $15$ such products, splitting the $60$ reflection hyperplanes into $15$ tetrahedra (Proposition \ref{quartets}). Of course degree $16$ is not a coincidence: a general fiber
is a set of singular points of a Kummer surface. 

\smallskip Applying Klein's construction, we can deduce from our quadratic complex a Kummer surface in $\PP^3$. Rather astonishingly, we again get a 
four dimensional linear system  of such quartics: they are given 
in Hudson's canonical form! The coefficients are degree twelve polynomials
on $\fc$, forming a linear system $\PP (V_5)$. 

\medskip\noindent {\bf Proposition.}
{\it The resulting rational map $K: \PP(\fc)\dashrightarrow \PP(V_5)$ is  generically finite of degree $256$; its image is the Segre cubic $\S3$, projectively dual to  the Castelnuovo-Richmond quartic.}

\medskip Here $V_5$ is spanned by certain products
of twelve linear equations of reflection hyperplanes. There are  again $15$ such products, from which we recover a version of the classical combinatorics of "synthemes" and "totals" that play a prominent role in the classical theory \cite{hudson, Dolg_classical}. A "total" in this setting is a collection of five degree $12$ products that cover the whole $60$ hyperplanes.  
The complex reflexion group $W_\fc$ acts transitively on the six "totals" and the stabilizer of any given one is a central extension by $\FF_2$ of the spin cover of $W(D_5)$. The exceptional isomorphism $Spin_5\simeq SL_4$ shows that $W(D_5)$ admits a projective representation of dimension four, which is nothing else than the Cartan subspace $\fc$; this closes the circle of identifications. The conclusion is:

\medskip \noindent {\bf Theorem.} {\it The GIT moduli space of codimension four sections of the spinor tenfold is} 
$$G(4,\Delta_-)/\hspace{-1mm}/Spin_{10}\simeq \PP(\fc)/W_\fc \simeq w\PP(2,3,5,6)$$
{\it and is isomorphic with the moduli space of Kummer surfaces.} 

\medskip 
The paper starts with a preliminary section reminding some useful facts about the geometry of the spinor tenfold. Sections 3 and 4 are devoted to linear sections of codimension two and three, which were already studied in \cite{kuz_spin}; we use a few results from Lie theory to get one step further. In particular we describe explicit models for several types of sections, deduce precise descriptions of some of their Hilbert schemes and related birational models; studying such Hilbert schemes 
and birational maps was an important perspective in \cite{kuz_spin}.  Sections 5 goes in a different direction, as Jordan-Vinberg theory takes action, and makes direct connexions between codimension four sections of the spinor tenfold and Kummer surfaces, that we describe in some detail. The final section 6 enlarges the picture by describing the 
full classification and spotting special sections with surprisingly 
big automorphism groups. 


\smallskip 
We included three Appendices. The first one contains a description of the cohomology ring of the smooth codimension four sections of the spinor tenfold, including its quantum version.  The second one explains some aspects of Jordan-Vinberg theory and how orbits can be organized according to their stabilizer (which in our setting is exactly the automorphism group). The third one includes some combinatorial data for the reflection hyperplanes of the complex reflection group $W_\fc=G_{31}$; in the end, this recovers the classical arrangement of the Klein configuration.  

\medskip\noindent  {\it Acknowledgements.} We thank Alessandra Sarti, C\'edric  Bonnaf\'e, Igor Dolgachev, Xavier Roulleau for useful exchanges, and Sasha Kuznetsov for allowing us to use the title of this paper. Special thanks to Willem de Graaf for his interest in this project, for sharing the results of \cite{degraaf}, and for answering 
tirelessly our countless questions. 

L. Manivel is supported by the ANR project
FanoHK, grant ANR-20-CE40-0023.

\vfill
\pagebreak 

\section{The spinor tenfold}
In this section we gather some useful facts about the spinor tenfold and its linear sections, that will be used in the sequel. Most of them can be found in \cite{kuz_spin} and \cite{doublespinor}. 

The complex spin group $Spin_{10}$ has two half-spin representations $\Delta_+$ and $\Delta_-$, of dimension $16$, dual one of the other
\cite{chevalley}. Following \cite{kuz_spin}, 
we will denote by $X$ the closed $Spin_{10}$-orbit in $\PP(\Delta_+)$, and call it 
the spinor tenfold. This is indeed a ten dimensional projective manifold that can be equivariantly identified with the orthogonal Grassmannian $OG(5,10)_+$ (one component of the variety of maximal isotropic subspaces of the natural representation $V_{10}$); but the Pl\"ucker embedding is actually the square of the spinor embedding. 

A nonzero vector $\delta\in\Delta_+$ belonging to the cone over $X$ is called a {\it pure spinor}. The associated maximal isotropic space $U_\delta$ can be reconstructed pretty much in the same way as a subspace can be recovered from a point in a Grassmannian. Instead of exterior product one just has to use the Clifford action, which can be seen as an equivariant map from $V_{10} \otimes\Delta_+$ to $\Delta_-$. Then \cite{chevalley}
\begin{equation}\label{subspace}
U_\delta=\{v\in V_{10}, \;\; v.\delta=0\}.
\end{equation}

The spinor tenfold shares some important properties with the Grassmannian $G(2,5)$ in its Pl\"ucker embedding: for example, its complement is a unique orbit, and it is isomorphic with its projective dual variety. To be precise, its projective dual $X^\vee$ is the closed $Spin_{10}$-orbit in the dual projective space $\PP(\Delta_+^\vee)=\PP(\Delta_-)$; it is 
projectively isomorphic to $X$, but not canonically; it is naturally isomorphic with $OG(5,10)_-$, the other family of maximal isotropic subspaces of $V_{10}$. 
The conormal variety $\cI$ with its two projections
$$\xymatrix{ & \cI\ar[rd]^{\pi_+}\ar[ld]_{\pi_-} &\\ 
X & & X^\vee}$$
identifies with the orthogonal Grassmannian $OG(4,10)\simeq Spin_{10}/P_{4,5}$, 
which is homogeneous with
Picard rank two  (recall that a four-dimensional isotropic 
subspace is contained in exactly two maximal isotropic subspaces, one from each family); moreover the two projections are projective bundles. 

Linear spaces in $X$ can be read from the Dynkin diagram $D_5$. 
The rule is very simple: one extremal vertex  defines $X$, 
and one obtains a family of $\PP^k$'s in $X$ by isolating a 
subdiagram of type $A_k$, with the same extremal node. Such a family is parametrized by a flag variety of $Spin_{10}$, defined by the vertices at the outer boundary of the subdiagram. Morever, all the $\PP^k$'s in $X$ belong to one of these families. There are  
five possibilities, as follows: 
$$\dynkin[edge length=6mm] D{oot*o}\qquad\dynkin[edge length=6mm] D{oto*t} \qquad \dynkin[edge length=6mm] D{oto*o} \qquad \dynkin[edge length=6mm] D{too*t} \qquad \dynkin[edge length=6mm] D{ooo*t}$$
Applying \cite[Theorem 4.9]{LM-linear}, we deduce:

\begin{prop}\label{linear-spaces}

(i) Lines in $X$ are parametrized by $OG(3,10)$, planes by $OF(2,5^\vee,10)$.\\
(ii) There are two types of $\PP^3$'s in $X$: maximal ones parametrized by $OG(2,10)$, extendable ones parametrized by 
$OF(1,5^\vee,10)$.\\
(iii) $\PP^4$'s in $X$ are parametrized by $X^\vee$ and are all maximal.
\end{prop} 

Here we denoted by $OF(k,5^\vee,10)$ the flag variety parametrizing flags $(U_k\subset U_5\subset V_{10})$, where $U_k$ has dimension $k$ and  $[U_5]\in X^\vee$. 

\subsection*{Smoothness of linear sections}
We deduce an easy but useful Lemma for linear sections of $X$. We follow Kuznetsov's notations and let $X_K=X\cap \PP(K^\perp)$ for a subspace $K$ of $\Delta_-=\Delta_+^\vee$. We also let $\Delta_K:=K^\perp\subset \Delta_+$, so that $X_K\subset\PP(\Delta_K)$. 

\begin{lemma} \label{singular-section}
A linear section $X_K$ of codimension at most four of the spinor tenfold is smooth 
if and only if $K$ contains no  pure spinor.
\end{lemma}

\proof A point $p$ is singular in $X_K$ if and only if $K^\perp$ contains 
a hyperplane $H$ that is tangent at $p$ to the spinor variety. This means that $h=H^\perp\in K$ is a pure spinor. Conversely, if $h\in K$ is a pure spinor,
the hyperplane $H=h^\perp$ is tangent to the spinor variety along a $\PP^4$, 
which has to meet $X_K$. \qed

\medskip 
The proof also shows that the singular locus of $X_K$ can, in full generality, be described as 
\begin{equation}\label{singular-locus}
Sing(X_K)=\bigcup_{h=H^\perp\in \PP(K)\cap X^\vee}Sing(X_H)\cap  \PP(K^\perp),
\end{equation}
where $Sing(X_H)\simeq\PP^4$. 

\medskip
The normal bundle to $X$ in $\PP(\Delta_+)$ is a twist of the rank five tautological bundle \cite[Proposition 2.2]{doublespinor}, and an easy consequence is that we will still get 
finitely many $Spin_{10}$-orbits after blowing-up  $X$ in $\PP(\Delta_+)$. In fact, there is a rank $8$ spinor bundle $\cS$ on the invariant quadric $Q\subset\PP(V_{10})$ \cite{ott-spinor}, which is a subbundle of the trivial bundle with fiber $\Delta_+$, 
and  a diagram  
$$\xymatrix{ & \PP_Q(\cS)\ar[rd]^q\ar[ld]_p &\\ 
\PP(\Delta_+)\ar@{.>}[rr]^\gamma & & Q}$$
where $p$ is the blowup of the spinor tenfold. Since a half-spin representation of $\Spin_8$ is equivalent, through triality, to the vector representation, $\PP_Q(\cS)$
contains a quadric subbundle $\cQ$, whose total space is the exceptional
divisor of $p$. This divisor is isomorphic to the flag manifold 
$OF(1,5,10)$ with its two projections to the quadric and the spinor tenfold. In particular, given a spinor $\delta$ which is 
not pure (and nonzero), there exists a unique isotropic vector $v$ (up to scalar) such that $\delta$ belongs to $v\Delta_-$, or equivalently $v\delta=0$: this is $[v]=\gamma ([\delta])$. 

The map $\gamma$ is quadratic, and can be defined from the Clifford multiplication $V_{10}\otimes \Delta_{\pm}\lra\Delta_{\mp}$ and the duality between $\Delta_+$ and $\Delta_-$; combining them, we get equivariant morphisms 
\begin{equation}\label{gamma}
\gamma : S^2\Delta_\pm\lra V_{10}.
\end{equation} 

\subsection*{Automorphisms}
The following result was proved in \cite[Proposition 4]{dedieu-manivel}:

\begin{prop}\label{auto}
Let $X_K$ be a smooth Fano section of the spinor tenfold, $K\subset\Delta_-$. 
Then $$Aut(X)\simeq \mathrm{Stab}_{\mathrm{Spin}_{10}}(K).$$
\end{prop}

Using this, the generic stabilizer of a codimension four section $X_K$ was proved to be 
$(\ZZ/2\ZZ)^2$, with three commuting involutions corresponding to three nondegenerate planes 
$A,B,C$ in $V_{10}$. This data decomposes the spin representation into four pieces 
of the 
same dimension four, and $K$ must be spanned by four lines, one in each piece \cite[Theorem 22]{dedieu-manivel}.

\subsection*{Local completeness} A standard question in such a setting is whether the family of sections is locally complete. For a smooth linear section $X_K$ of $X$, the normal exact sequence $$0\lra TX_K\lra TX_{|X_K}\lra K^\vee\otimes \cO_{X_K}(1)\lra 0$$
shows that local completeness at $X_K$ is a consequence of the vanishing of $H^1(TX_{|X_K})$. Using the Koszul resolution, we see that this follows from the vanishing of $H^{i+1}(TX(-i))$ for each $i$ from $1$ to $k=dim(K)$. The following direct consequence of the Borel-Weil-Bott theorem shows that we are on the safe side:

\begin{lemma}
$TX(-i)$ is acyclic for any $i$ from $1$ to $6$.
\end{lemma}

We deduce local completeness up to codimension $6$. It also holds true in codimension $7$, although $TX(-7)$ is no longer acyclic: by Borel-Weil-Bott again, its non-zero cohomology group appears in degree $7$. 

\begin{prop} 
The family of smooth sections of the spinor tenfold is locally complete in any 
dimension bigger than two. 
\end{prop}

In codimension $8$ we get K3 surfaces, so local completeness definitely fails.

\section{Codimension two sections} 

In this section we mainly recall some results from \cite{kuz_spin}, which will be essential in the sequel. 
The fact that $\Delta_+$ and $\Delta_-$ have finitely many orbits can be thought of as a consequence 
of the existence of the exceptional Lie group $E_6$ \cite{KWE6}, and of the $\ZZ$-grading 
$$\fe_6=\fg_{-1}\oplus \fg_0\oplus \fg_1=\Delta_-\oplus (\fso_{10}\times\CC) \oplus \Delta_+$$
(see Appendix \ref{jordan}). 
For  $E_7$ there is a similar  $\ZZ$-grading 
$$\begin{array}{rcl} 
\fe_7 &=& \fg_{-2}\oplus \fg_{-1}\oplus \fg_0\oplus \fg_1\oplus \fg_2 \\
&=& V_{10}\oplus (\CC^2\otimes\Delta_-)\oplus (\fso_{10}\times\fgl_2) \oplus (\CC^2\otimes\Delta_+)\oplus V_{10}.
\end{array}$$
So $\CC^2\otimes \Delta_+$ and $\CC^2\otimes \Delta_-$ have the same finiteness property; their nine orbits under $GL_2\times Spin_{10}$ are discussed in \cite{KWE7}. In particular there is a codimension one orbit, which is a dense open subset of a  degree four hypersurface, and then a codimension four orbit. Note that since $\CC^2\otimes \Delta_-=Hom((\CC^2)^\vee, \Delta_-)$, we can deduce that the 
Grassmannian $G(2,\Delta_-)$ has exactly six $Spin_{10}$-orbits (we just throw away 
the three orbits for which the rank is not full). 

Moreover, the degree four hypersurface 
in $\CC^2\otimes \Delta_-$ induces a quadric hypersurface $\cR$ in $G(2,\Delta_-)$, and the next orbit has again codimension four. Using Lemma \ref{singular-section} and the 
explicit representative of the orbits given on 
\cite[page 40]{KWE7}, one can see that 
a codimension two section $X_K$ of $X$ is smooth if and only if $K$ belongs to the open orbit in $G(2,\Delta_-)$, or to the codimension one orbit (see also \cite[Corollary 6.17]{kuz_spin}). In the latter case we say
that $X_K$ is a {\it special} smooth section. 
From \cite{KWE7} we can easily extract explicit representatives:
$$\begin{array}{lclll}
\# & \;\mathrm{codim} & \mathrm{isotropy} & \mathrm{representative} & \mathrm{type}  \\
(1) & 0&G_2\times SL_2&\langle 1+e_{1235}, e_{45}+e_{1234}\rangle & \mathrm{general}\\
(2) & 1& (GL_1\times Spin_7).\GG_a^8 & \langle 1+e_{1235}, e_{35}+e_{1234}\rangle & \mathrm{special}
\end{array}$$
These representatives are given in terms of the usual model of the half-spin representations, obtained by choosing a decomposition $V_{10}=E\oplus F$ into the direct sum of two maximal isotropic subspaces. Then $\Delta_+=\wedge^+E$ and 
 $\Delta_-=\wedge^-E$, with the action of $V_{10}$ and $\fso_{10}\simeq
\wedge^2V_{10}$ defined by wedge-product by vectors from $E$ and contraction
by vectors from $F$. We choose a basis $e_1,\ldots ,e_5$ of $E$ and the dual basis $f_1,\ldots ,f_5$ of $F$. Finally, we denote $e_{ij}=e_i\wedge e_j$, etc. 

\medskip
Special sections can be characterized  by their geometry  \cite[Proposition 6.1]{kuz_spin}:

\begin{prop}\label{special}
A smooth codimension two section $X_K$ of $X$ is special if and only if either
\begin{itemize}
    \item $X_K$ contains a $\PP^4$, in which case it contains in fact a line of such $\PP^4$'s;
    \item the plane spanned by the conic $\gamma(K)$ is contained in $Q$. 
\end{itemize}
\end{prop}

Recall from Proposition \ref{linear-spaces} that four-planes in the spinor tenfold are parametrized by the dual spinor tenfold, and that's in this sense that we talk of a line of $\PP^4$'s. Concretely, if $\delta$
is a pure spinor in $\Delta_-$, we have recalled in (\ref{subspace}) that we can recover the 
corresponding maximal isotropic subspace $U_\delta$ of $V_{10}$ as the kernel of the 
map $V_{10}\lra\Delta_+$ defined by the Clifford product with $\delta$. 
The image of 
this map has also dimension five and defines the corresponding $\PP^4$ in the spinor
tenfold $X$. It is contained in $X_K$ if and only if $\langle v\delta, k\rangle =0$
for all $v\in V_{10}$, $k\in K$. But $\langle v\delta, k\rangle=\langle \delta, vk\rangle$,
so we need $\delta$ to be orthogonal to $V_{10}K\subset\Delta_+$. A straightforward 
computation with the representatives above shows that in the general case (1), 
$V_{10}K=\Delta_+$. In the special case (2), $V_{10}K=\langle e_{13}, 
e_{23}\rangle^\perp,$ and $\PP\langle e_{13}, 
e_{23}\rangle$ is a projective line contained in the spinor tenfold.
\medskip 

The appearance of the exceptional group $G_2$ in the generic stabilizer of $G(2,\Delta_-)$ comes as an interesting surprise, and is discussed around 
\cite[Definition 2.1.2]{pasquier}
\cite[Proposition 4.9]{baifuman}. In particular we have a 
specific model for the generic codimension two section of the spinor tenfold, as follows.  
Denote by $V_7$ the minimal representation of $G_2$, and by $A_2$ the natural 
representation of $SL_2$. Then $V_{10}=V_7\oplus S^2A_2$ has a natural quadratic form and one can check that $\Delta_\pm\simeq (V_7\oplus \CC)\otimes A_2$. In particular 
$\Delta_-$ contains a unique invariant plane $K\simeq A_2$, whose orthogonal 
in $\Delta_+$ is $\Delta_K=V_7\otimes A_2$. An easy consequence is that the
general codimension two section of the spinor tenfold 
$$X_K\subset \PP(V_7\otimes A_2)$$ 
contains $\QQ^5\times\PP^1$, and blowing-up this codimension two subvariety yields the total space of the projective bundle $\PP(N\otimes A_2)$ over the adjoint variety $G_2/P_2\subset G(2,V_7)$, where $N$ is the restriction of the tautological bundle on the Grassmannian. 

In the sequel we will give similar models for other types of sections.

\medskip
From \cite[Proposition 32]{sk} we also know that the complement of the open orbit in 
$\CC^2\otimes\Delta_\pm$ is a quartic hypersurface. This is another way to understand the {\it spinor quadratic line complex}  $\cR\subset G(2,\Delta_-)$ from \cite{kuz_spin}; this is the locus of special codimension two sections, and it 
is defined by the unique invariant line in $S_{22}\Delta_+=S_{22}\Delta_-^\vee$.
Geometrically, $\cR$ can be described by  a Kempf like resolution of singularities
$$\xymatrix{ & Gr_{OG(3,10)}(2,\cW)\ar[rd]\ar[ld] &\\ 
OG(3,10) & & \cR}$$
where $\cW$ is a rank eight spinor bundle on $OG(3,10)$ \cite[Lemma 6.28]{kuz_spin}.
Kuznetsov deduces that the singular locus of $\cR$ has codimension $7$ in $\cR$ 
\cite[Corollary 6.25]{kuz_spin}. More precisely, it can be described as the locus
of secant lines to the dual spinor tenfold $X^\vee$. 

By Proposition \ref{special},
there is a morphism from the codimension one orbit in $G(2,\Delta_+)$,
to the orthogonal Grassmannian $OG(3,10)$. 
Explicitly, for the representative given above the kernels are generated by vectors
of the form $a^2e_3-b^2f_4+abf_5$, and they span the isotropic three-plane $U=\langle e_3,f_4,f_5\rangle$. Conversely, it is easy to check that $K\subset\wedge^2U.\Delta_+$.

\begin{remark}
One can easily give an explicit formula for an equation $Q_\cR$ of $\cR$. 
Indeed recall from (\ref{gamma}) the equivariant morphism from $S^2\Delta_-$ to $V_{10}$. For any basis $(v_i)$ of $V_{10}$, with dual basis $(w_i)$ (with respect to the invariant quadratic form $q$), 
the polarisation of $\gamma$ is given by 
$$\gamma(\delta, \delta')= \sum\langle v_i\delta,\delta'\rangle w_i.$$
Taking the symmetric square of $\gamma$ we get an equivariant morphism from $S^2(S^2\Delta_-)$ to $S^2V_{10}$, then to $\CC$ by composing with $q$. Restricting to $S_{22}\Delta_-$, which we can describe inside 
$S^2(S^2\Delta_-)$ as the image of the morphism from $S^2(\wedge^2\Delta_-)$ defined by $$(\delta_1\wedge\delta_2)(\delta_3\wedge\delta_4)\mapsto 
(\delta_1\delta_3)(\delta_2\delta_4)-(\delta_2\delta_3)(\delta_1\delta_4),$$
we deduce the explicit quadratic equation
$$Q_\cR(\delta\wedge\delta')=\sum_i\Big(\langle v_i\delta,
\delta'\rangle 
\langle w_i\delta,\delta'\rangle -\langle v_i\delta,\delta\rangle 
\langle w_i\delta',\delta'\rangle\Big).$$
\end{remark}

\section{Codimension three sections} 

For codimension three sections of the spinor tenfold, we go one step  beyond the results of \cite{kuz_spin} by using the explicit classification provided by Lie theory. We obtain explicit models for the general and the most special smooth sections. We deduce nice birational models in terms of  
binary forms, and we describe some of their Hilbert schemes. 

\subsection{Classification of smooth sections}
Exactly as for the previous cases, the existence of $E_8$ ensures that $\CC^3\otimes \Delta_+$ and $\CC^3\otimes \Delta_-$ have finitely many orbits under the action of $GL_3\times Spin_{10}$. They were classified in \cite{orbites_SL3Spin10} and \cite{KWE8}, the conclusion being that there are $32$ orbits. Once we suppress those whose rank are not maximal, there remains $23$ orbits of $Spin_{10}$ inside the Grassmannians $G(3,\Delta_+)$ and $G(3,\Delta_-)$.

One way to distinguish at least some of these orbits is, for $K$ a three-dimensional subspace of $\Delta_-$, to look at the intersection of the plane $G(2,K)\simeq \PP^2$ with $\cR$. 
When $G(2,K)$  is not contained in $\cR$, 
this intersection is a conic $c_K$ given by a section of $S_{22}\cU^\vee$, where $\cU$ denotes the rank-three tautological bundle. Taking its determinant we get a section of $\det(\cU^\vee)^4$, 
that is a quartic hypersurface $\cR_3\subset Gr(3,\Delta_-)$.  

\smallskip Using Lemma \ref{singular-section} it is not difficult 
to locate, inside the classification  of \cite{KWE8} (or that of
 \cite{orbites_SL3Spin10}),
the orbits in $GL(3,\Delta_-)$ corresponding to smooth sections. 
There are four of them, and they can indeed be distinguished by the type of 
the conic $c_K$: smooth for (1), 
singular but reduced for (2), non reduced for (3). The latter orbit has codimension three as expected.
There is another orbit (4) in $G(3,\Delta_-)$ of codimension three, whose closure parametrizes
singular sections. And surprisingly there is an extra orbit (6) 
of smooth sections, for which 
the plane $G(2,K)$ is contained in $\cR$. These are called {\it very special} sections in 
\cite[Definition 7.8]{kuz_spin}.

$$\begin{array}{lclll}
\# & \;\mathrm{codim} & \mathrm{isotropy} & \mathrm{representative} & \mathrm{type\;of\;} c_K \\
(1) & 0& SL_2\times SL_2 & \langle 1+e_{2345}, e_{12}-e_{1345}, e_{23}+e_{1245}\rangle & \mathrm{smooth\;conic}\\
(2) & 1& GL_1^2.U_5& \langle 1+e_{2345}, e_{12}+e_{34}, e_{13}+e_{45}\rangle & \mathrm{singular\;conic}\\
(3) & 3&(GL_1\times SL_2).U_5& \langle 1+e_{2345}, e_{12}+e_{34}, e_{13}+e_{1245}\rangle & \mathrm{double\;line}\\
(6) & 5&(GL_1\times SL_2).\mathbb{G}_a^7& \langle 1+e_{2345}, e_{12}+e_{34}, e_{45}+e_{1234}\rangle & \mathrm{whole\;plane}
\\
(4) & 3&(GL_1^2\times SL_2).U_4& \langle 1+e_{2345}, e_{23}-e_{1345}, e_{12}\rangle & 
\end{array}$$
\smallskip

We included on the last line of the list the orbit whose closure parametrizes {\it 
singular} sections, to stress the surprising fact that there is an orbit of bigger codimension that parametrizes (very special) smooth sections. 

\smallskip
In the description of isotropy groups, $U_d$ denotes a $d$ dimensional unipotent group, and $\mathbb{G}_a$ is the additive group. So all these sections have non trivial automorphism group, all the bigger than they
are more special. That there is a factor $\mathbb{G}_a^7$ in the automorphism group
of the very special sections is in agreement 
with the observation made in \cite[section 4.2]{baifuman}, according to which 
those {\it very special} sections are equivariant compactifications of $\CC^7$, and are acted on quasi-transitively 
by their automorphism group \cite[Proposition 4.11]{baifuman}. This is not the case for the generic codimension three section, whose automorphism group 
$SL_2\times SL_2$ 
(up to a finite group) is just too small to act with an open orbit. 

\smallskip
The computation of $c_K$ goes as follows. Let us treat case (1). We start by computing $\gamma$ restricted to $\PP (K)$. If $p=s(1+e_{2345})+t(e_{12}-e_{1345})+u(e_{23}+e_{1245})$, we get 
$$\gamma(p)=[t^2e_1-(u^2+st)e_2-ute_3+s^2f_1+stf_2-suf_3].$$
A codimension two section defined by a line $\overline{p_1p_2}\subset\PP(K)$ is special when the image of the line by $\gamma$ is contained in $Q$, which happens 
exactly when $ \gamma(p_1)$ and $ \gamma(p_2)$ are orthogonal. But 
$$q(\gamma(p_1),\gamma(p_2))=(s_1t_2-t_1s_2)^2-(s_1u_2-u_1s_2)(t_1u_2-u_1t_2),$$
which is the equation of a smooth conic in the Pl\"ucker coordinates of the line 
$\overline{p_1p_2}$. 

The other cases are similar. For (2) we find $q(\gamma(p_1),\gamma(p_2))=(s_1u_2-u_1s_2)(t_1u_2-u_1t_2),$
showing that $c_K$ is a reduced singular conic. For (3) we get 
$q(\gamma(p_1),\gamma(p_2))=(s_1u_2-u_1s_2)^2$, showing that $c_K$ is a double line. 

\medskip
An explanation of the fact that we get the whole plane for the very special 
sections of type (6), is that $X_K$ itself contains a $\PP^4$, which is 
in fact unique. According to \cite[Proposition 7.7]{kuz_spin}, another characterization of very special sections $X_K$ is that  $\gamma(K)$ spans a 
$\PP^4$ contained in $Q$, and is in particular a projected Veronese surface. 
Of course this isotropic $\PP^4$ defines, in agreement with Proposition \ref{linear-spaces}, the $\PP^4$ in the spinor tenfold that is contained in $X_K$. 



\subsection{A model for the generic section}
We can describe the generic sections through the following model. Consider two-dimensional spaces $A$ and $B$, and the orthogonal direct sum
$$
V_{10}=S^{4}A \oplus \mathbb{C} \oplus A \otimes B,
$$
with its natural action of  $SL(A) \times SL(B)$. Note that 
$S^{4}A$ and  $A \otimes B$ both admit invariant quadratic forms, uniquely defined up to scalar. Then there are some coefficients to choose in order to define an invariant non-degenerate quadratic form
on $V_{10}$, but they will not matter.
In the sequel we identify $S^k{A}$ with its dual by choosing a volume form on $A$.

Being defined that way, $V_{10}$ splits into the orthogonal sum of $A \otimes B$, whose spin modules are just $A$ and $B$, and $S^{4}A \oplus \mathbb{C}$, whose spin modules, being the same as those $S^{4}A$, identify with $S^3A$ and its (isomorphic) dual. We deduce that the induced actions of $SL(A) \times SL(B)$ on the half-spin modules $\Delta_{\pm}$ decompose as:
$$
\Delta_{\pm}=S^2A \oplus S^4A \oplus S^3A \otimes B,
$$
Consider $K=S^2A\subset \Delta_-$, so that 
$$
\Delta_{K}=S^4A \oplus S^3A \otimes B.
$$

\begin{lemma}
 The associated section $X_K$ of the spinor tenfold is 
generic. 
\end{lemma}

\proof To prove that $X_K$ is smooth we need to check that $K$ contains no pure spinor, or equivalently that the quadratic map $\gamma: K\lra V_{10}$ does not vanish outside the origin; but this is clear since by equivariance it has to factorize through the map $K=S^2A\lra S^4A$ mapping 
a degree two polynomial to its square. Moreover, the only invariant conic in $\PP(K)=\PP(S^2A)$ is $v_2\PP(A)$; it must be $c_K$, which in particular is smooth. \qed 


\medskip
The equations of $X_K\subset \PP\Delta_K$ are quadratic equations with values in $S^{4}A \oplus \mathbb{C} \simeq \wedge^2(S^3A)$ and $A\otimes B$, respectively. Choose a basis $(b_1,b_2)$ of $B$. A point $p=[\gamma+c_1\otimes b_1+c_2\oplus b_2]$ belongs to $X_K$ if and only if 
\begin{equation}\label{XKequations}
c_1\wedge c_2=Q(\kappa), \qquad  c_1\lrcorner \kappa = c_2\lrcorner \kappa = 0.
\end{equation}
Here $Q: S^2(S^4A)\lra \wedge^2(S^3A)$ is some $SL(A)$-equivariant map. We have already noticed 
that $ \wedge^2(S^3A) \simeq S^{4}A \oplus \mathbb{C}$ is not irreducible; neither is $S^2(S^4A)$,
which contains an extra factor $S^8A$. In particular $Q$ is not uniquely determined, even up to scalar. At this point, all we can say is that it kills the extra factor $S^8A$, which implies that 
$X_K$ contains the rational quartic curve $\cC=v_4\PP(A)\subset\PP(S^4A)$.
(In order to see that $\cC$ parametrizes  maximal isotropic subspaces of $V_{10}$, observe that for any non-zero $a\in A$, $a^3A\oplus a\otimes B\subset S^4A\oplus A\otimes B$ is a four-dimensional isotropic subspace of $V_{10}$, hence contained in a uniquely defined maximal isotropic subspace from each of the two families.) 

The second equations in (\ref{XKequations}) mean that the two cubics $c_1, c_2$ are apolar to the quartic $\kappa$. For a general such pencil, there is a unique such quartic, up to scalar. We can compute it explicitly by writing $\kappa =\kappa_0x^4+4\kappa_1 x^3y+ 6\kappa_2x^2y^2+ 4\kappa_3 xy^3+ \kappa_4 y^4$, $c_1=\alpha_0x^3+3\alpha_1 x^2y+ 3\alpha_2xy^2+ \alpha_3 y^3$ and
 $c_2=\beta_0x^3+3\beta_1 x^2y+ 3\beta_2xy^2+ \beta_3 y^3$. Then $(\kappa_0,\ldots , \kappa_4)$ must belong to the kernel of the matrix 
 $$M(c_1,c_2)=\begin{pmatrix} 
 \alpha_0 & \alpha_1 & \alpha_2 & \alpha_3 & 0\\
 \beta_0 & \beta_1 & \beta_2 & \beta_3 & 0\\
 0 & \alpha_0 & \alpha_1 & \alpha_2 & \alpha_3 \\
 0 & \beta_0 & \beta_1 & \beta_2 & \beta_3 
 \end{pmatrix}.$$
This kernel, when one-dimensional, is generated by the quartic $\mathrm{ap}(c_1\wedge c_2)$ whose coefficients are the 
maximal  minors of this matrix; these coefficients are quadratic in the Pl\"ucker coordinates of the pencil $\langle c_1, c_2\rangle$. If $U$ is the tautological bundle on the Grassmannian $G(2, S^3A)$, we can interprete the matrix $M(c_1,c_2)$ as encoding the morphism $U\otimes A\lra S^4U$; its image has codimension bigger than one when there is a non-trivial kernel, which  means that there exists linear forms $\ell_1$ and $\ell_2$ such that $c_1\ell_1=c_2\ell_2$; so $c_1$, $c_2$ are multiple of the same binary form $q$, defining a plane $qA\subset S^3A$. This embeds $\PP(S^2A)$ as a Veronese surface $\Sigma$ inside $G(2,S^3A)$.

Conversely, given a general binary quartic $\kappa$, there is a 
plane $AP_\kappa$ of cubics apolar to $\kappa$. The Pl\"ucker coordinates of this plane are the $2\times 2$ minors of the matrix 
$$\begin{pmatrix}
\kappa_0 & \kappa_1 & \kappa_2 & \kappa_3 \\
\kappa_1 & \kappa_2 & \kappa_3 & \kappa_4 
\end{pmatrix},$$
that define $\mathrm{AP}(\kappa)\in\wedge^2(S^3A)$.
This implies that the rational map $\mathrm{AP}: \PP(S^4A)\dashrightarrow G(2,S^3A)$ is defined by the linear system of quadrics vanishing on the rational quartic curve $\cC=v_4\PP(A)$. In particular, this yields an identification $\wedge^2(S^3A)\simeq I_\cC(2)\simeq S^4A\oplus \CC$, whose inverse we denote by $Q_0$. 

By construction, the rational maps $\mathrm{AP}$ and $\mathrm{ap}$ are  birational equivalences between $G(2,S^3A)$ and $\PP(S^4A)$, inverse one to the other. 
The rational map $AP$ is defined by $I_\cC(2)$. 
It blows-up $\cC$ to the resultant hypersurface $\cR$ parametrizing pencils of cubics admitting a common root, and contracts the divisor $Sec(\cC)$, the secant variety of $\cC$, to the Veronese surface $\Sigma=Sing(\cR)$ parametrizing pencils of cubics with two common roots. The inverse rational map $ap$ is defined by $I_\Sigma(2)$. This is summarized in the following diagram, where we included the linear projection $\pi : G(2,S^3A)\lra\PP(S^4A)$ defined by the inclusion of 
$S^4A$ inside $\wedge^2(S^3A)$, which is the projection from a point outside the Grassmannian (hence a degree two finite morphism).

$$\xymatrix{  & \hspace{-2cm}Bl_\cC\PP(S^4A)\hspace{-2cm}\ar[d] & \quad = &Bl_\Sigma G(2,S^3A)\ar[d]  & \\
\cC\subset Sec(\cC)\subset  &\hspace{-5mm} \PP(S^4A)\ar@{.>}[rr]<2pt>^*{\mathrm{AP}} & &  G(2,S^3A)\ar[d]_\pi\vspace*{-5mm}\ar@{.>}[ll]<2pt>^*{\mathrm{ap}} &  \hspace{-1.1cm}\supset \cR \supset\Sigma \\
&&& \PP(S^4A) &  \hspace{-2.2cm}\supset \cD
}$$

The projection $\pi$ was used in \cite{newstead} in order to deduce invariants of pencils of binary cubics from the classical invariants $I$ and $J$ of binary quartics. A simple computation shows that the image of $\cC$ by $\pi\circ \mathrm{AP}$ is the discriminant hypersurface $\cD$ in 
$ \PP(S^4A)$, of equation $I^3-27J^2=0$; so the pull-back of $\cD$ by $\pi$ must be twice the 
resultant hypersurface $\cR$, which has degree three. Note that $\cR$ is a birational image of 
$\PP(A)\times G(2,S^2A)$ by the multiplication morphism, which is two-to-one along $\Sigma$;
in particular $\cR$ is not normal along $\Sigma$.

\smallskip
Let us come back to the inverse birationalities  $\mathrm{AP}$ and $\mathrm{ap}$. 
Algebraically, the maps $c_1\wedge c_2\mapsto \mathrm{ap}(c_1\wedge c_2)$ and $\kappa\mapsto \mathrm{AP}(\kappa)$ being inverse one to the other means that there must exist  cubic forms $C$ and $c$ such that
$$\mathrm{AP}(\mathrm{ap}(c_1\wedge c_2))=C(c_1\wedge c_2) c_1\wedge c_2, \qquad \mathrm{ap}(\mathrm{AP}(\kappa))=c(\kappa)\kappa. $$

Now consider equations (\ref{XKequations}) and suppose that $\mathrm{ap}(c_1\wedge c_2)$ is well-defined. Then $\kappa$ being apolar to $c_1$ and $c_2$ must be a multiple of it, 
say   $\kappa = t\mathrm{ap}(c_1\wedge c_2)$. Plugging this into the first equation we get 
$c_1\wedge c_2=t^2Q(\mathrm{ap}(c_1\wedge c_2))$. For this equation to admit solutions, we need that up to scalar, $Q=\mathrm{AP}$; and we may suppose there is indeed equality. We then get the equation 
$ (1-t^2C(c_1\wedge c_2)) c_1\wedge c_2=0.$  When $C(c_1\wedge c_2)\ne 0$ we get two solutions, 
showing that the projection from $X_K$ to $\PP(S^3A\otimes B)$ is generically finite of degree two. 
But when $C(c_1\wedge c_2)=0$ we get no solution, unless $c_1$ and $c_2$ are proportional. 
Denote by $F$ the sextic hypersurface in  $\PP(S^3A\otimes B)$ defined by the condition that $C(c_1\wedge c_2)\ne 0$. Since the generic point of $F$ is not in the image of the regular map 
$X_K/\cC\lra \PP(S^3A\otimes B)$, it has to be the image of the exceptional divisor $E$ in the blow-up $\tilde{X}_K$ of $X_K$ along $\cC$. Of course $E=\PP(N_{\cC/X_K})$ is  the projectivization of the normal bundle of $\cC$ in $X_K$.     
This normal bundle can be obtained by restricting the equations of $X_K$ in $\PP\Delta_K$ to first order at a point $[a^4]\in \cC$: for $[a^4+\delta a^4+\delta c_1\otimes b_1+ \delta c_2\oplus b_2]$ this yields 
$$0=Q(a^4,\delta a^4), \qquad  \delta c_1\lrcorner a^4 = \delta c_2\lrcorner a^4 = 0.$$  
The first equation simply means that $\delta a^4$ is tangent to $\cC$ at $a^4$; the other equations are verified if and only if $c_1$ and $c_2$ are divisible by $a$. So up to twist, the normal bundle at $a^4$ is $aS^2A\otimes B\subset S^3A\otimes B$. Therefore the image of $E$ in $\PP(S^3A\otimes B)$ must be the union of these linear spaces, a birational image of $\PP(A)\times \PP(S^2A\otimes B)$, whose equation is given by the resultant of two binary cubics. By Sylvester's formula this is the determinantal sextic in $\PP(S^3A\otimes B)$ defined by 
 $$\det \begin{pmatrix} 
 \alpha_0 & \alpha_1 & \alpha_2 & \alpha_3 & 0&0\\
 \beta_0 & \beta_1 & \beta_2 & \beta_3 & 0&0\\
 0 & \alpha_0 & \alpha_1 & \alpha_2 & \alpha_3&0 \\
 0 & \beta_0 & \beta_1 & \beta_2 & \beta_3 &0\\
 0&0 & \alpha_0 & \alpha_1 & \alpha_2 & \alpha_3 \\
 0&0 & \beta_0 & \beta_1 & \beta_2 & \beta_3 
 \end{pmatrix}=0,$$
 which is indeed a cubic polynomial in the Pl\"ucker coordinates.
In particular we deduce:

\begin{lemma} Given two binary cubics $c_1$ and $c_2$, $C(c_1\wedge c_2)$ is their resultant. Given a binary quartic $\kappa$, $c(\kappa)$ is its $J$-invariant. \end{lemma}

The diagram below summarizes the situation. 

Beware that the map $\eta$ from $\tilde{X}_K$ to $\PP(S^3A\otimes B)$ is only generically finite; it has positive dimensional fibers over the pre-image $\Sigma'$ of $\Sigma$, which is no longer smooth but a birational, singular image of $\PP(S^2A)\times\PP(A\otimes B)\simeq \PP^2\times\PP^3$. In fact the projection from $\PP(S^3A\otimes B)$ to $G(2,S^3A)$ is only defined outside the Segre product 
$\PP(S^3A)\times \PP(B)=Sing(\Sigma')$ (which is also embedded in $X_K$). Outside this locus, its fibers are 
complements of smooth quadric surfaces in $\PP^3$. 
The morphism $\eta$ factorizes through the double cover $\hat{\PP}$ of $\PP(S^3A\otimes B)$ branched over the sextic hypersurface defined by the resultant. The resulting birational morphism $\hat{\eta} : \tilde{X}_K\lra \hat{\PP}$ has non-trivial fibers over $\Sigma'$: conics outside the singular locus, and the union of three lines over the codimension one singular locus. This means that $\hat{\PP}$ is not normal; but
$\hat{\eta}$ has to lift to the normalization $\breve{\PP}$, and the lift $\breve{\eta}$ is an extremal contraction of $\tilde{X}_K$.
$$\xymatrix{
& & \tilde{X}_K\ar[ld]_{Bl_\cC}\ar[rd]^{2:1} & & \\
\cC\subset \hspace{-1.9cm}& X_K\ar@{.>}[rr]\ar@{.>}[dd]  & & \PP(S^3A\otimes B)\ar@{.>}[dd] & \hspace{-1.1cm}\supset\Sigma' \\
& & \tilde{\PP}\ar[ld]_{Bl_\cC}\ar[rd]^{Bl_\Sigma} & & \\
\cC\subset \hspace{-1.2cm} & \PP(S^4A)\ar@{.>}[rr] & & G(2,S^3A) & \hspace{-1.5cm}\supset\Sigma
}$$

\medskip\noindent
{\it Remark.} Note that $S^3A\otimes B$ is a Vinberg representation coming from the affine $G_2$; it admits two independent invariants of degree $2$ and $6$, so there is a pencil of invariant sextics. This is in agreement with Newstead's observation that these invariants are induced by the two fundamental invariants $I$ and $J$ for binary quartics.


\medskip
Consider an extendable $\PP^3$ in $X_K$. It is contained in  a unique $\PP^4$ of the spinor tenfold $X$. This $\PP^4$ defines a codimension two oversection $X_H\supset X_K$, which is obviously special. Hence a map from $F_3(X_K)$ to $c_K\simeq \PP(A)$, whose fibers are lines by Proposition \ref{special}. We can be more precise:

\begin{prop} The Hilbert scheme  $F_3^{ext}(X_K)\simeq \PP(A)\times \PP(B)$, the three-dimensional space defined by $[a]\in\PP(A)$ and  $[b]\in\PP(B)$ being  $\PP^3_{a,b}=\PP\langle a^4, aS^2A\otimes b\rangle.$
\end{prop}

\begin{coro}
Two extendable $\PP^3$'s can intersect only along $\cC$ or along $\PP(S^3A)\times \PP(B)$, the singular locus of $\Sigma'$. A generic point of $Sing(\Sigma')$ belongs to finitely many extendable $\PP^3$'s (three in general), but any point of $\cC$ 
is contained in a one-dimensional family of extendable $\PP^3$'s. In particular the curve $\cC$ 
is uniquely defined in $X_K$.
\end{coro}

Along similar lines, one can easily prove the following statement:

\begin{prop} The Hilbert scheme  $F_3^{next}(X_K)\simeq  \PP(B)$, the three-dimensional space defined by   $[b]\in\PP(B)$ being  $\PP^3_{b}=\PP(S^3A\otimes b).$
\end{prop} 


\subsection{A model for the special smooth section}
Now consider a two dimensional space $U$, and endow the direct sum
$$
V_{10}=S^4U \oplus S^4U^{\vee}
$$
with the non-degenerate quadratic form defined by the natural pairing between the two factors 
$E=S^4U$ and $F=S^4U^\vee$. Since $E$ (as well as $F$) is isotropic, the half-spin representations can be obtained as $\Delta_+=\wedge^+E$
and $\Delta_-=\wedge^-E$, the even and odd parts of the exterior algebra of $E$. Since $S^4U$ has dimension five, we have $\wedge^3(S^4U)\simeq \wedge^2(S^4U)\simeq S^6U\oplus S^2U$ and $\wedge^4(S^4U)\simeq S^4U$ as $SL(U)$-modules. We deduce that the  actions of $SL(U)$ on the half-spin modules $\Delta_{\pm}$ decompose as
$$
\Delta_{\pm}=\mathbb{C} \oplus (S^2U \oplus S^6U) \oplus S^4U. 
$$
Consider $K=S^2U \subset \Delta_{-}$, hence
$$
\Delta_{K}=\mathbb{C} \oplus S^6U \oplus S^4U.
$$

\begin{lemma}
The associated section $X_{K}$ is smooth and very special.
\end{lemma}

\begin{proof}
The quadratic map $\gamma:K \rightarrow V_{10}$ factorizes through the map $K=S^2 U \rightarrow S^4 U$ mapping a degree two polynomial to its square. Thus $\gamma$ does not vanish outside the origin, hence $X_{K}$ is smooth. Moreover the linear span of $\gamma(K)$ equals to $S^4U$ which is a maximal isotropic subspace by definition. By \cite[Prop 7.7.(3)]{kuz_spin}, this means that the smooth section $X_K$ is very special. 
\end{proof}

The equations of $X_K \subset \PP(\Delta_{K})$ are the following. Consider the  $SL(U)$-covariant   $$\pi: S^2(S^6U)\lra S^4U,$$ which is unique up to scalar. 
Then $[\delta_0, \delta_6, \delta_4]$ belongs to $X_K$ if and only if 
\begin{equation}\label{eq-veryspecial}
\delta_4\lrcorner \delta_6=0,  \qquad \pi(\delta_6)=\delta_0\delta_4
\end{equation}
(for a suitable normalization).
Note in particular that $\PP(S^4U)\simeq\PP^4$ is contained in $X_K$, confirming that this is a very special section \cite[Prop 7.7.(1)]{kuz_spin}. 

Equations (\ref{eq-veryspecial}) show that if $\delta_0\ne 0$, $\delta_4$ is uniquely determined by $(\delta_0,\delta_6)$. So  $X_{K}$ is the image of the birational map:
$$
\psi: \mathbb{P}(\mathbb{C} \oplus S^6U) \dashrightarrow X_{K} \subset \mathbb{P}(\Delta_{K})
$$
$$
[z:p] \dashrightarrow [z^2:zp:\pi(p)]
$$

Denote by $Z \subset \mathbb{P}(S^6U)$ the zero-locus of $\pi$, which is also the indeterminacy locus of $\psi$. In fact $Z$ is a del Pezzo threefold of degree $5$, as follows from the facts that $\wedge^2(S^4U)=S^6U\oplus S^2U$: indeed this implies that $Z$ is contained in the linear section of $\PP(S^6U)$ of $\PP(\wedge^2(S^4U))$, in which its equations coincide with the restrictions of the Pl\"ucker equations; 
hence  $$Z=G(2,S^4U)\cap \PP(S^6U).$$ 

\noindent {\it Remark.} According to Mukai \cite[Proposition 10]{mukai-fano3folds}, $Z$ is quasi-homogeneous under the action of $SL_2$, and can be described as the compactification of $SL_2/\Gamma$, 
for $\Gamma\simeq S_4$ embedded in $SO_3$ as the octahedral group. We can see $Z$ as the orbit closure in $\PP(S^6U)$ of the binary sectic $ab(a^4+b^4)$. That this binary sextic has symmetries given by the octahedral group was first observed by Bolza in 1887.

\medskip
Note that $\psi^{-1}$ is defined by the projection from $\PP(S^4U)$. Following \cite{kuz_spin}, there is a diagram
\[ \xymatrix{E_Z\ar[d]\subset\hspace*{-5mm}  & Bl_{Z}(\mathbb{P}(S^4U))\ar[d] & \simeq &  Bl_{\mathbb{P}(S^4U)}(X_{K})\ar[d] &  \hspace*{-5mm}\supset\;\; F_\PP\ar[d]
\\ Z\subset\hspace*{-5mm} & \mathbb{P}(\mathbb{C} \oplus S^6U)\ar@{.>}[rr]^\psi & & \PP(\Delta_K) &\hspace*{-1cm}\supset \;\;\mathbb{P}(S^4U) } \]
where $E_Z, F_\PP$ are the exceptional divisors. 
Then $E_Z$ is mapped birationally onto the singular hyperplane section $D$ of $X_{K}$ defined by $D=\{\delta_{0}=0\}$. The fibers of the projection from $E_Z$ to $Z$ map to  a covering family of $\PP^3$'s on $D$, parametrized by $Z$.

\begin{prop}
The Hilbert scheme of $\PP^3$'s in $X_K$ has three  components.
More precisely, 
$$F_3^{next}(X_K)\simeq Z, \qquad F_3^{ext}(X_K)\simeq \PP(S^4U)\cup F_1(Z),$$
where $F_1(Z)\simeq \PP(S^2U)$ denotes the Fano variety of lines in the del Pezzo threefold.  
\end{prop}

\begin{proof}
Any $P\simeq \PP^{3}$ in $X_{K}$, not contained in $\mathbb{P}(S^4U)$, must be projected by $\psi^{-1}$ to a linear subspace of $Z$, necessarily of dimension at most one since $F_{2}(Z)= \emptyset$. 

Suppose that $\psi^{-1}(P)$ is a point  $x=[\delta_{6}] \in Z$. A line connecting 
$x$ with $y=[\delta_4]\in\mathbb{P}(S^4U)$ is contained in $X_K$ when  $\delta_{6} \lrcorner \delta_{4}=0$. The union of such lines 
is exactly the $\PP^3$ obtained as the image in $X_K$ of the fiber $E_{x}$ of the $\PP^3$-bundle $E$ over $Z$; it has to coincide with $P$. 

Now suppose that $\psi^{-1}(P)$ is a line $L$ in $Z$.  By \cite{fn}, such lines are parametrized by $\PP(S^2U)$; more precisely, they are of the forms
$L=\mathbb{P}\langle  a^5b,ab^5 \rangle $ or $L=\mathbb{P} \langle a^5b,a^6 \rangle$ for some $a,b\in U$, the first case being generic. Then one calculates that each such line is contained in a unique $\mathbb{P}^{3}$ in $X_{K}$, obtained as the linear span of $L$ with a line $L'$ in $\mathbb{P}(S^4U)$: either $L'=\PP\langle a^4, b^4\rangle$ in the fist case, or  $L'=\PP\langle a^4, a^3b\rangle$ in the second case. 

If $P$ is extendable to some $\PP^4$, say $\widetilde{P}$, the latter must be contained in some $X_{L}$ for a plane $L$ inside $K$. Denote by $\Pi$ the span of $\gamma(L)$ in $S^4U$. From \cite[Prop 6.1]{kuz_spin} $\Pi$ has dimension three and there exists a unique $F \in X^{\vee}$ containing $\Pi$ such that $\widetilde{P}=\pi_{+}^{-1}(F)$. Then $F \cap S^4U=\Pi$, hence the intersection of $P$ and $\mathbb{P} S^4U$ is a line, namely the line given by the isotropic $4$-spaces in $S^4U$ that containing $\Pi$. 
\end{proof}
The automorphism group $Aut^{o}(X_{K})=(\mathbb{G}_{m} \times SL_{2}) \cdot \mathbb{G}_{a}^{7}$ can be realized as follows. The action of $\mathbb{G}_{m} \cdot \mathbb{G}_{a}^{7}$ on $X_{K}$ is inherited from the associated action on $\mathbb{P}^{7}$ through the birational map $\psi$. This is a special example of Euler-symmetric variety as studied in \cite{Euler_sym}.
More precisely, consider the obvious action of the vector group $S^6U$ on $\mathbb{P}(\mathbb{C} \oplus S^6U)$ defined by
$a.[z:p]=[z:p+za]$, for any $a \in S^6U$.
Then $\psi$ induces an $S^6U$-action on $\mathbb{P}(\Delta_{K})$ by:
$$
a.[\delta_{0}:\delta_{6}:\delta_{4}]=[\delta_{0}:\delta_{6}+\delta_{0}a:\delta_{4}+\delta_{0}\pi(a,a)+2\pi(\delta_{6},a)].
$$
It restricts to an action of $S^6U$ on $X_{K}$, which realizes $X_{K}$ as an equivariant compactification of $S^6U$. Similarly the $\mathbb{G}_{m}$-action on $\mathbb{P}(\mathbb{C} \oplus S^6U)$ defined by 
$z'.[z:p]=[z:z'p]$ induces a $\mathbb{G}_{m}$-action on $X_{K}$ by:
$$
z'.[\delta_{0}:\delta_{6},\delta_{4}]=[\delta_{0}:z\delta_{6}:z^2\delta_{4}].
$$
Note that both actions are trivial on $Z$.  This induces the  exact sequence:
$$
1 \rightarrow \mathbb{G}_{m}\cdot \mathbb{G}_{a}^{7} \rightarrow Aut^{o}(X_{K}) \rightarrow Aut^{o}(Z) \simeq SL(U) \rightarrow 1.
$$

\subsection{A brief reminder on Kummer surfaces}
Before going any further, we need a short reminder about Kummer surfaces. Recall that they are 
nodal quartic surfaces in $\PP^3$ with the maximal number of nodes, namely $16$. 
It is a classical fact that the equation of any  Kummer surface in $\PP^3$ can be expressed, in suitable coordinates, as 
\begin{equation}\label{kummer}
A(x^2y^2+z^2t^2)+B(x^2z^2+y^2t^2)
+C(x^2t^2+y^2z^2)+2Dxyzt+E(x^4+y^4+z^4+t^4)=0
\end{equation}
for some coefficients $(A,B,C,D,E)$ satisfying the cubic condition \cite[§53]{hudson}
\begin{equation}\label{segre}
4E^3-(A^2+B^2+C^2-D^2)E+ABC=0.
\end{equation}
This is {\it Hudson's canonical form}.
In such coordinates, the nodes of the surface form an orbit of an order sixteen subgroup $F_0\simeq \ZZ_2^4\subset PGL_3$. 

\medskip
Alternatively, a Kummer surface can be obtain a the image of a principally polarized abelian surface 
$(\mathcal{A},\Theta)$ by the linear system $|2\Theta|\simeq\PP^3$. This linear system is invariant under multiplication by $-1$ on $\mathcal{A}$, and its image $\mathcal{K}$ is the quotient of $\mathcal{A}$ by this involution. In particular, the sixteen nodes of $\mathcal{K}$ are the images of the two-torsion points of 
$\mathcal{A}$. 
Moreover, this abelian surface is the Jacobian  of a genus two curve $\mathcal{C}$; the coefficients of Hudson's canonical form can be expressed in terms of the theta constants of the curve. 
The space of quartics of type (\ref{kummer}) 
is exactly the space of quartic polynomials that are invariant under the action of the finite Heisenberg group generated by the transformations 
$$\begin{pmatrix} 1&0&0&0 \\ 0&1&0&0 \\ 0&0&-1&0 \\ 0&0&0&-1 \end{pmatrix},\quad \begin{pmatrix} 1&0&0&0 \\ 0&-1&0&0 \\ 0&0&1&0 \\ 0&0&0&-1 \end{pmatrix},\quad \begin{pmatrix} 0&1&0&0 \\ 1&0&0&0 \\ 0&0&0&1 \\ 0&0&1&0 \end{pmatrix},\quad \begin{pmatrix} 0&0&1&0 \\ 0&0&0&1 \\ 1&0&0&0 \\ 0&1&0&0 \end{pmatrix}.$$


\medskip
Equation (\ref{segre}) defines the Segre cubic $\S3$ in $\PP^4$, which has the maximal number of nodes, namely $10$. More symmetrically, it can be described in $\PP^5$ by the equations
$$z_0+\cdots +z_5=z_0^3+\cdots +z^3_5=0,$$
showing that its automorphism group is the symmetric group $S_6$. 
The projective dual of the Segre cubic is the 
Igusa (or Castelnuovo-Richmond) quartic $\CR4$ in $\PP^4$. This quartic has $15$ singular lines meeting along $15$ points, with $3$ points on each line and conversely. It is isomorphic to the Satake compactification of the moduli space  of principally polarized abelian surfaces with a full level two structure. 
An amazing property of the Igusa quartic is that any tangent section is a Kummer surface! More information on this moduli space can be found in \cite{hunt, hulek-sankaran, hulek-heisenberg}. 
The relationship with the Heisenberg group and its Schr\"odinger representation is discussed in \cite{dolg-config}. 

\medskip
Klein discovered a beautiful relationship between Kummer surfaces and quadratic line complexes, that is, quadric sections  $\Gamma=G(2,4)\cap Q$ of the Grassmannian parametrizing lines in $\PP^3$. The point is that $G(2,4)$ admits two families of projective planes, parametrized by $\PP^3$ and its dual. Restricting $Q$ to the planes in one such family yields a family of conics, that become singular above the discriminant surface. This is in fact a quartic 
surface in $\PP^3$, which itself acquires singularities when the conic becomes a double line; generically, this happens at exactly $16$ points and one gets a Kummer surface \cite[Chapter 6]{GH}. In fact the two families of planes yield a Kummer $\mathcal{K}\subset\PP^3$ and its projective dual $\mathcal{K}^\vee\subset\check{\PP}^3$, which are projectively isomorphic one to the other. Under the polar morphism, each singular point is blown-out to a conic which is a double section of the dual Kummer, and contains six singular points of the dual Kummer. The associated genus two curve is obtained as the double cover of any of these conics, branched at the six singular points. Finally, the set of singular points of the singular conics parametrized by $K$ is the intersection of  $G(2,4)\cap Q$ with another quadric, hence a K3 surface. The $16$ singular points and the $16$ double conics form a $16_6$ configuration, the {\it Kummer configuration}, that will be discussed in more details later on. 

\subsection{The spinor Kummer complex}
The spinor quadratic complex $\cR\subset G(2,\Delta)$ yields  a whole family of quadratic complexes of lines, just by restriction to $G(2,K)$ for $K\subset \Delta_-$ four-dimensional. Restricting 
this quadric $\cR_K$ to the planes in the Grassmannian, one obtains two conic bundles over $\PP(K)$
and $\PP(K^\vee)=G(3,K)$, whose discriminant hypersurfaces are, in general,
a Kummer surface and its isomorphic dual . By construction, we get:

 \begin{prop}\label{r3}
 The quartic hypersurface $\cR_3\subset G(3,\Delta)$ has the property 
 that 
 its generic intersection with $G(3,K)\simeq \PP^3$, for any generic 
 $K\subset \Delta$ of dimension four, is a Kummer surface. 
 \end{prop}
 
Recall that the singularities of the  Kummer surface can be resolved to get a  K3 surface,  obtained by looking at the singular points of the singular conics in the conic bundle \cite{GH}. 
Over $[U_1]\in\PP(V_4)$ the induced conic in 
 $\PP(V_4/U_1)$ is singular at the point $[U_2]$ if and only if 
 $$Q(\wedge^2U_2,U_1\wedge V_4)=0.$$
 But $U_1\wedge U_2=\wedge^2U_2$, so that if we suppose that $U_2$
 belongs to $\Gamma$, this gives two conditions on $U_1\subset U_2$. More precisely,  we get over $\PP(U_2)$ a section of $(U_1\wedge V_4/\wedge^2U_2)^\vee \simeq (U_1\otimes  (V_4/U_2))^\vee$, hence two sections of $\cO(1)$. These two sections have a common zero if and 
 only if they are proportional, which means that the morphism from $U_2\otimes \wedge^2U_2$ to $(V_4/U_2)^\vee$ defined on $\Gamma$ by sending $u\otimes v\wedge w$ to the linear form $Q(v\wedge w, u\wedge\bullet)$, must have rank one. Its determinant is a morphism from $\det(U_2)^3$ to $\det(U_2)\otimes \det(V_4^\vee)$, hence a section of $\cO_\Gamma(2)$ that vanishes (in general) along a smooth K3 surface of degree eight. 
 
 \medskip
 When the quadric $\cQ$ is the restriction of $\cR$ to $G(3,K)$, this same analysis yields a locus $\Sigma_\cR\subset Fl(2,4,\Delta)$, dominating $\cR\subset G(2,\Delta_-)$ and defined inside the pull-back of $\cR$ by a section 
 of $\det(U_2^\vee)^2\otimes \det(K^\vee)$. We can state:
 
 \begin{prop}
The morphism  $\Sigma_\cR\lra G(4,\Delta_-)$ is generically a K3 bundle, whose fiber over $K$ is a smooth model of the Kummer surface $\cR_3\cap G(3,K)\simeq\PP^3$.
 \end{prop}
 
 One can also easily recover the abelian surface whose Kummer surface is $\cR_3\cap G(4,K)$. 
 Consider the flag manifold $Fl(1,3,K)$ parametrizing flags $U_1\subset U_3\subset K$. Over this 
 manifold, the bundle $\cE$ with fiber $U_1\wedge U_3$ is a rank two subbundle of the trivial bundle with fiber $\wedge^2K$, and inherits from the spinor quadratic complex a quadratic form, that is, a section of $S^2\cE^\vee$. As observed in \cite[9.3]{GSW}, the zero locus of this section is in general an abelian surface $A$, and the projections to $\PP(K)$ and $G(3,K)\simeq \PP(K^\vee)$ are degree two covers of the Kummer surface and its (isomorphic) dual. 
 
 Interestingly, $A$ is in general the singular locus of the hypersurface of $Fl(1,3,K)$
 defined by the condition that the quadratic form defined on $\cE$ degenerates. This hypersurface is given by a section of $\det(\cE)^2\simeq \det(U_1)^{-2}\otimes\det(U_3)^{-2}$; this means that it is a quadric section of  $Fl(1,3,K)$, with respect to its minimal ample line bundle. 
 It is natural to coin this quadric section as the {\it Coble quadric}, since its singular locus is an abelian surface, exactly as for the classical Coble cubic in $\PP^8$. 
 
 This discussion globalizes as follows: 
 
 \begin{prop}\label{Coble}
 The spinor quadratic complex defines a hypersurface $\cR_{2,2}$ of $Fl(1,3,\Delta_-)$, 
 whose intersection with  $Fl(1,3,K)$, for the generic $K\in G(4,\Delta_-)$, is a Coble 
 quadric section,  singular exactly along the abelian surface that covers the Kummer surface  
 $\cR_3\cap G(3,K)$. 
 \end{prop}
 
\section{Codimension four sections: Kummer surfaces}

In this section we study the codimension four sections of the spinor tenfold with the help of Vinberg's theory of graded Lie algebras and we make various connexions wit hKummer surfaces. As we already mentionned in the introduction, it turns out that the exceptional Lie algebra $\fe_8$ admits a cyclic grading with one component isomorphic to $\CC^4\otimes\Delta_-$. This has extremely strong consequences, starting with the possibility to get a Jordan type classification of orbits, whose general principles are explained in Appendix \ref{jordan}. The effective classification was carried out by W. de Graaf \cite{degraaf}. 

The main ingredient for this section is the Cartan subspace $\fc$, that parametrizes semisimple elements. Cartan subspaces are all conjugate and we provide one explicitely, with the associated complex reflection group $W_\fc$ that permutes the $60$ reflection hyperplanes. Using the spinor quadratic complex we get a quadratic complex of lines, and then a Kummer surface. 
We will discuss how this is related to, and allows to recover a substantial part of, the classical theory of Kummer surfaces. Our main reference for the latter are \cite{hudson, Dolg_classical, gd}.

\subsection{A grading of $\fe_8$}
Since $E_9$ does not exist as an algebraic group, one cannot expect that $\CC^4\otimes\Delta_-$ or $G(4,\Delta_-)$ have finitely many orbits -- and
they don't, as a simple dimension count confirms. But the diagram $E_9$ is the affine Dynkin diagram $E_8^{(1)}$,
which indicates that  $\CC^4\otimes\Delta_+$ and  $\CC^4\otimes\Delta_-$ are connected to Vinberg's theory 
of graded Lie algebras. Indeed, there is a $\ZZ_4$-grading 
$$\fe_8=\fsl_4\times\fso_{10}\oplus (\CC^4\otimes\Delta_+)\oplus (\wedge^2\CC^4\otimes V_{10}) \oplus (\wedge^3\CC^4\otimes\Delta_-),$$
and this implies that the $SL_4\times Spin_{10}$-module structure of $\CC^4\otimes\Delta_+$
has very nice properties, similar to those of the usual Jordan theory of adjoint actions on semisimple Lie algebras \cite{vinberg}. In particular, the GIT-quotient 
$$\CC^4\otimes\Delta_+/\hspace{-1mm}/SL_4\times Spin_{10}\simeq \fc/W_\fc\simeq \AA^4,$$
where $\fc$ is a Cartan subspace acted on by its normalizer $W_\fc$.
The latter is not a Weyl group as in the usual Jordan theory, but a finite complex reflection group. More precisely, we can read from \cite[Case 11 of Table page 492]{vinberg} that $W_\fc$ is the reflection group  $G_{31}$ in the Sheppard-Todd classification \cite{st}; 
its order  is $46080$ and its 
invariants form a polynomial algebra on four generators of degrees $8, 12, 20, 24$.

\medskip\noindent {\it Remark.} Write the grading of $\fe_8$ as $\fg_0\oplus\fg_1\oplus\fg_2\oplus\fg_3$, indices being taken modulo four.
The Lie bracket restricts to a map $\wedge^2\fg_1\lra\fg_2$, where $\fg_1=\CC^4\otimes\Delta_+$ and $\fg_2=\wedge^2\CC^4\otimes V_{10}$. Being $\fsl_4\times\fso_{10}$-equivariant, this morphism has to factorize through
$$\wedge^2(\CC^4\otimes\Delta_+)\lra \wedge^2\CC^4\otimes S^2\Delta_+
\stackrel{Id\otimes\gamma}{\lra} \wedge^2\CC^4\otimes V_{10}, $$
showing that the map $\gamma$ from (\ref{gamma}) plays an essential role in the algebra structure of $\fe_8$. 

Note also that $\fg_3=\fg_{-1}$ is dual to $\fg_1$, since  $\wedge^3\CC^4$ is dual to $\CC^4$ and  $\Delta_-$ is dual to $\Delta_+$. 

\subsection{From $\mathrm{Spin}_{10}$ to the Kummer configuration}
The relationship between $Spin_{10}$ and Kummer surfaces that we observed in the previous section may look like a coincidence, since it mainly relied on the existence of the 
spinor quadratic complex. To give more evidence that the connection is deeper than
it may seem at first sight, we make the following construction. Recall that $\Delta_+$ and $\Delta_-$
are minuscule representations, in the sense that their weights have multiplicity one and 
are acted on transitively by the Weyl group. These weights, in terms of the usual basis $\epsilon_1,\ldots, \epsilon_5$, are of the form $\frac{1}{2}
(\pm\epsilon_1\pm\epsilon_2\pm\epsilon_3\pm\epsilon_4\pm\epsilon_5)$ with an even (resp. odd) number of minus signs for the weights $P_+$ (resp. $P_-$) of  $\Delta_+$ (resp. $\Delta_-$). It is convenient to  represent such weights simply by a sequence of $+$ and $-$. 

Consider a quadratic form on the weight lattice such that 
$(\epsilon_i,\epsilon_j)=\delta_{ij}$. Then, define on $P_+\times P_-$ an incidence 
relation by deciding that two weights $\mu$ and $\nu$ are incident when $(\mu,\nu)=\frac{3}{4}$ mod $2$. If we represent the weights by a collection of five $\pm$ signs, two weights are incident when they are opposite or have a unique difference. So clearly each of the $16$ weights of $P_+$ is incident to $6$ weights
in $P_-$, and conversely. We call this configuration the {\it spinorial $16_6$ configuration}. 

\smallskip
Following \cite{gd}, we call a set of six weights incident to a given one a {\it special plane}. This refers to the famous example of $16_6$ configuration provided by  the configuration of singular points and tropes on a Kummer surface, 
the tropes being the sixteen projective planes that cut the Kummer surface along a double conic (mapping through the Gauss map to one of the singular points of the dual, isomorphic Kummer surface). We call this the Kummer configuration. 
A classical, easy way to represent the Kummer configuration is to place $16$ points on a square $4\times 4$-grid. Then each special plane correspond to a point ($\bullet$) in the grid, being given by the union of the other points ($\circ$) 
in  the same line or same column:

$$\begin{array}{|c|c|c|c|} \hline 
 & & \circ & \\ \hline 
 & & \circ & \\ \hline 
 \circ &\circ &\bullet &\circ  \\ \hline  & & \circ & \\
 \hline \end{array}$$
 
\begin{prop}\label{16_6}
The $16_6$ spinorial configuration coincides with the Kummer configuration.
\end{prop}

\proof 
Abstract $16_6$ configurations were classified in \cite[Theorem 1.20]{gd}: there exist only three non-degenerate such configurations, non-degeneracy meaning that 
any two points belong to exactly two special planes, and conversely, any two special planes share  exactly two points. 

\begin{lemma} The spinorial configuration is non-degenerate. \end{lemma}

\noindent {\it Proof of the Lemma.} Consider a pair a weights $\omega_0,\omega_1$ in $P_+$. Since the Weyl group acts transitively on $P_+$, we may suppose that $\omega_0=+++++$. Its stabilizer in the Weyl  group is $S_5$, and up to the action of this stabilizer we may suppose that $\omega_1=+++--$ or $\omega_1=+----$. In the first case, $\omega_0$ and $\omega_1$ are incident to  $\omega_2=++++-$ and  $\omega_3=+++-+$, and only to those weights. In the second case, they are incident 
to $\omega_2=-++++$ and $\omega_3=-----$ and only to those weights. By symmetry between $P_+$ and $P_-$, this proves the claim. \qed 

\medskip 
Among the three non-degenerate $16_6$ abstract configurations, the Kummer configuration was  characterized in \cite[Theorem 1.42]{gd} as the only one that can be realized in $\PP^3$ by a collection of $16$ points and planes. Since we don't have any natural 
representation of the spinorial configuration in $\PP^3$, we rather give a 
direct isomorphism. Below we provide two $4\times 4$ grids $G_+$ and $G_-$ 
containing the weights of $P_+$ and $P_-$, respectively. We leave to the reader the explicit verification that for each weight in $P_+$, corresponding to a point $\omega_+$ in $G_+$, hence a point $\omega_-$ in $G_-$, the six other weights in the same row or column as $\omega_-$ in $G_-$ are exactly the weights incident to $\omega_+$ in the spinorial configuration - and conversely. 

$$ G_+\qquad\begin{array}{|c|c|c|c|} \hline 
 --+++ & ----+ & ---+- & --+-- \\ \hline 
 +++++& ++--+& ++-+- & +++-- \\ \hline 
 +---- &+-++- &+-+-+ & +--++  \\ \hline 
 -+---& -+++-& -++-+ & -+-++ \\
 \hline \end{array}$$

$$ G_-\qquad\begin{array}{|c|c|c|c|} \hline 
 ----- & --++- & --+-+ & ---++ \\ \hline 
 ++---& ++++-& +++-+ & ++-++ \\ \hline 
 +-+++ &+---+ &+--+- & +-+--  \\ \hline 
 -++++& -+--+& -+-+- & -++-- \\
 \hline \end{array}$$
 
 This proves the claim. \qed

\begin{coro}
The Weyl group of $D_5$ acts on the Kummer configuration.
\end{coro}

The automorphism group $W_{Kum}$ of the Kummer configuration has been well studied by the classics. It has order $11520$ \cite[Ex. 30 p. 42]{charmichael}, which is only $6$ times the order of $W(D_5)$, equal to $1920$. 

\begin{prop} \label{wd6}
$W_{Kum}\simeq W(D_6)/\{\pm 1\}$ is the quotient of the Weyl group of $D_6$ by its center. 
\end{prop} 

\proof The two half-spin representations of $Spin_{12}$ have dimension $32$, and $W(D_6)$ acts transitively on their weights, that we can identify with a length six sequence of $\pm$, with an even or odd number of minus signs, respectively. 
We define an incidence relation between pairs of opposite weights of $\Delta_{12}^+$ and pairs  of opposite weights of $\Delta_{12}^-$ by deciding that $\pm w$ is incident to $\pm w^*$ when $\langle w,w^*\rangle \ne 0$. This identifies with the
spinorial configuration constructed on weights of  $\Delta_{10}^+$ by embedding them into the sets of weights of either  $\Delta_{12}^+$ or  $\Delta_{12}^-$, just by adding a sign, which yields maps $u\mapsto w=\pm u_+$ and $u\mapsto w^*=\pm u_-$. By definition, $u$ and $u'$ are not incident when they have only two different signs, which means that $\pm u_+$ and $\pm u'_-$ have three different signs, hence are not incident as well. So we get the same configuration. 

We deduce a morphism from $W(D_6)$ to $W_{Kum}$, with kernel $\{\pm 1\}$. 
Since $W(D_6)$ has order $2^5 6!$, the morphism is surjective, and we are done.
\qed

\medskip\noindent {\it Remark.} Automorphisms of the Kummer configuration are usually described in terms of $Sp(4,\FF_2)$. As observed by I. Dolgachev, 
$W(D_6)$ is a non-split extension of the latter by $\FF_2^4$.   

\medskip\noindent {\it Remark.} When we constructed the Kummer configuration from the weights of the two half-spin representations of $Spin_{10}$, we observed that we ended up with some extra data: namely, a bijection between "points" and "planes", such that each 
point is incident to the associated plane. There should be only six such bijections, acted on transitively by $W_{Kum}$, with stabilizers isomorphic with $W(D_5)$.  

\medskip
Since $W(D_6)$ is a group of signed permutations, if contains both $S_6$ and $\FF^5_2$  as subgroups, and the normal subgroup $\FF_2^5$ contains the center. We deduce that 
$$W_{Kum}\simeq \FF_2^4\rtimes S_6.$$

The abstract automorphism group $W_{Kum}$ has a very close relative which plays a very important role in the theory of Kummer surfaces, defined as follows: choose a 
copy  of  $\FF_2^4$ inside $PGL_4$ (they are all conjugate), and let $N$ be its normalizer in $PGL_4$. This group is described in  \cite[Theorem 1.64]{gd}, where it is denoted by $N$; it has the same order $11520$ as $W_{Kum}$, and there is an exact sequence 
$$1\lra \FF_2^4\lra W_{Kum}\lra  S_6\lra 1,$$
which is non-split: $S_6$ is not a subgroup of $N$, and $N$ is not a semi-direct product. According to \cite[Claim 1.68]{gd}, linear automorphisms of the Kummer configuration in $\PP^3$ can be realized in $N$, but they form only a proper 
subgroup of $W_{Kum}$. 



\subsection{The Cartan subspace and the little Weyl group}\label{cartan}
Not surprisingly, we can also use the relationship with the combinatorics of the Kummer configuration to describe an explicit Cartan subspace $\fc$ of $\CC^4\otimes \Delta_+$. 
Choose a basis $a_1,\ldots, a_4$ of $\CC^4$. Since $\fc$ has dimension four, it must be generated by some $c_i=\sum_p a_p\otimes\delta_{ip}$, $1\le i\le $. The Lie brackets  
$$[c_i,c_j]=\sum_{p,q}a_p\wedge a_q\otimes \gamma(\delta_{ip}\delta_{jq})$$
vanish if $\gamma(\delta_{ip}\delta_{jq})=\gamma(\delta_{iq}\delta_{jp})$ for all $p,q$, 
where $\gamma : S^2\Delta\ra V_{10}$ is the equivariant projection. Suppose that  $\delta_{ip}$ and $\delta_{jq}$ are pure spinors, parametrizing isotropic
five-spaces $U_{ip}$ and $U_{jq}$; then  $\gamma(\delta_{iq}\delta_{jp})=v_{iq,jp}$ is a  generator  of
$U_{ip}\cap U_{jq}$ if this is a line, and  $\gamma(\delta_{iq}\delta_{jp})=0$ if its dimension is bigger. So we need to find $16$ maximal isotropic spaces $U_{ip}$ such that  $U_{ip}\cap U_{jq}=U_{iq}\cap U_{jp}$
when these intersections are lines, or they both have bigger dimension. 
It is of course tempting to choose the $16$ isotropic subspaces corresponding to the
weights of $\Delta_+$, but it is not obvious that we can number them in such a way 
that our conditions are satisfied. 

In order to solve this combinatorial problem, we can just open \cite[page 45]{hudson} 
and 
see that each time we choose two weights in $P_+$ and $P_-$ that are not incident,
there is a natural way to arrange the sixteen weights of $P_+$ inside a $4\times 4$ 
array, exactly as we wish to. For example, one of these $4\times 4$ arrays, written in terms of pure spinors $\delta_{ip}$, is 
\begin{equation}\label{square}
\begin{array}{cccc}
 e_{53} & e_{1245} & e_{42} & e_{31} \\
 e_{52} & e_{1345} &  e_{34} &  e_{12} \\
 e_{1234} &  1 & e_{1235} &  e_{54} \\
 e_{14} &  e_{23} &  e_{51} &  e_{2345}  
\end{array}
\end{equation}
We can see that for any position $(i,p)$ in this array, there is a unique other position $(j,q)$, with $j\ne i$ and $q\ne p$, such that $U_{ip}\cap U_{jq}$ is a line. Moreover this line is always 
equal to $U_{iq}\cap U_{jp}$. This implies that we can normalize the spinors $\delta_{ip}$ so that the relations $\gamma(\delta_{ip}\delta_{jq})=\gamma(\delta_{iq}\delta_{jp})$ hold, since 
each $\delta_{ip}$ is involved in a single such relation. And in fact we have 
normalized them correctly in the array above. In other words, we get a abelian subspace $\fc = \langle p_1,p_2,p_3,p_4\rangle$ in $\CC^{4} \otimes \Delta_{+}$ by letting 
$$\begin{array}{ccc}
 p_1 & = & a_1\otimes e_{53} +a_2\otimes  e_{1245} +a_3\otimes  e_{42} +a_4\otimes  e_{31}, \\
 p_2 & =& a_1\otimes e_{52} +a_2\otimes  e_{1345} +a_3\otimes   e_{34} +a_4\otimes   e_{12}, \\
 p_3 &=& a_1\otimes e_{1234} +a_2\otimes   1 +a_3\otimes  e_{1235} +a_4\otimes   e_{54}, \\
 p_4 &=& a_1\otimes e_{14} +a_2\otimes   e_{23} +a_3\otimes   e_{51} +a_4\otimes   e_{2345}.  
\end{array}$$\par 

\begin{prop} \cite{degraaf} $\fc$ is a Cartan subspace of $\CC^{4} \otimes \Delta_{+}$.\end{prop}

\noindent {\it Remark.} Grids completely similar to (\ref{square}) already appear in the classics, see \cite[page 45]{hudson}, with a different interpretation. We find it rather amazing that the natural way to fill the
$4\times 4$ square is by the sixteen weights of the half-spin representation. In this language, 
one key property of  (\ref{square}) is that the weights sum to zero on each line and each column. 

\smallskip\noindent {\it Remark.} From the general theory, we know that there is a unique way to complete the Cartan subspace $\fc\subset \fg_1=\CC^4\otimes \Delta_+$ by another Cartan subspace $\fc'\subset \fg_{-1}=(\CC^4)^\vee\otimes\Delta_-$  so that the direct sum 
$\fh=\fc\oplus\fc'$ be a Cartan subalgebra of $\fc'$ \cite{vinberg}. In our setting, a basis of $\fc'$ can be constructed that is (not surprisingly) completely
similar to the basis we introduced for $\fc$, namely
$$\begin{array}{ccc}
 p'_1 & = & a^\vee_1\otimes e_{124} +a^\vee_2\otimes  e_{3} +a^\vee_3\otimes  e_{153} +a^\vee_4\otimes  e_{254}, \\
 p'_2 & =& a^\vee_1\otimes e_{143} -a^\vee_2\otimes  e_{2} +a^\vee_3\otimes   e_{152} +a^\vee_4\otimes   e_{354}, \\
 p'_3 &=& a^\vee_1\otimes e_{5} +a^\vee_2\otimes e_{12345} -a^\vee_3\otimes  e_{4} +a^\vee_4\otimes   e_{123}, \\
 p'_4 &=& a^\vee_1\otimes e_{253} +a^\vee_2\otimes   e_{154} +a^\vee_3\otimes   e_{243} +a^\vee_4\otimes   e_{1}.  
\end{array}$$

\medskip
W. de Graaf also identified the $60$ reflection hyperplanes; their equations are listed in Appendix \ref{combinatorics}. They form a unique orbit under the action of the little Weyl group $W_\fc$, which is generated by the following transformations of $\fc$, given by their matrices in the basis $p_1, p_2, p_3, p_4$:
$$s_1=\begin{pmatrix} -1&0&0&0 \\ 0&1&0&0 \\ 0&0&1&0 \\ 0&0&0&1 \end{pmatrix}, \quad s_2=\begin{pmatrix} 0&-1&0&0 \\ -1&0&0&0 \\ 0&0&1&0 \\ 0&0&0&1 \end{pmatrix}, \quad s_3=\begin{pmatrix} 0&-i&0&0 \\ i&0&0&0 \\ 0&0&1&0 \\ 0&0&0&1 \end{pmatrix},$$ 
$$s_4=\frac{1}{2}\begin{pmatrix} 1&-1&-1&-1 \\ -1&1&-1&-1 \\ -1&1&1&-1 \\ -1&-1&-1&1 \end{pmatrix}, \quad s_5=\frac{1}{2}\begin{pmatrix} 0&0&-1-i&-1+i \\ 0&1&0&0 \\ -1+i&0&1&i \\ -1-i&0&-i&1 \end{pmatrix}.$$
The first three matrices generate a group that acts on $p_1,p_2$ by either changing their signs, 
multiplying both of them by $i$, or swaping them. In particular the 
following matrices belong to $W_\fc$: 
$$g_1=\begin{pmatrix} 1&0&0 &0 \\ 0&1&0&0 \\ 0&0&-1&0 \\ 0&0&0&-1 \end{pmatrix}, \quad
g_2=\begin{pmatrix} 0&i&0 &0 \\ -i&0&0&0 \\ 0&0&1&0 \\ 0&0&0&-1 \end{pmatrix}, \quad
g_3=\begin{pmatrix} 1&0&0 &0 \\ 0&-1&0&0 \\ 0&0&0&i \\ 0&0&-i&0 \end{pmatrix}, $$
$$g_4=\begin{pmatrix} 0&-1&0 &0 \\ -1&0&0&0 \\ 0&0&0&-1 \\ 0&0&-1&0 \end{pmatrix}, \qquad
g_5=\frac{1}{2}\begin{pmatrix} 0&0&1+i&1-i \\ 0&0&1-i&1+i \\ 1-i&1+i&0&0 \\ 1+i&1-i&0&0 \end{pmatrix}.$$
These are five involutions that commute pairwise, except the pairs $(g_1,g_5)$, $(g_2,g_3)$, $(g_2,g_4)$, $(g_3,g_4)$, $(g_3,g_5)$ that only anticommute. Moreover $g_1g_2g_3g_4=iId$. We deduce:

\begin{prop} The group $F$ generated by $g_1,\ldots ,g_5$ contains the center   $Z(W_\fc)\simeq \FF_4$, and fits into an exact sequence 
\begin{equation}\label{Heisenberg}
1\lra Z(W_\fc)\simeq \FF_4\lra F\lra F_0\simeq\FF_2^4\lra 1.
\end{equation}
In particular $F$ is a group of Heisenberg type, of order $64$.
\end{prop}


\medskip

\noindent {\it Remark}. The subgroup of $SL_4\times Spin_{10}$ that acts trivially on $\fc$ can be shown to be another Heisenberg group, isomorphic to $F$ \cite{gopel}. This implies that the automorphism group of a generic codimension four section of the spinor tenfold is isomorphic to $\FF_2^4$, and not to $\FF_2^2$ as erroneously claimed in \cite{dedieu-manivel}. (The mistake is in the dimension counts at the end of the proof of Theorem 22.)

\medskip
In the next subsections, we will show that the combinatorial relationship between 
the half-spin representation $\Delta_+$ and the Kummer configuration can be 
upgraded to a geometric relationship between codimension four sections of the spinor tenfold on the one hand, and Kummer surfaces on the other hand, with quadratic complexes of lines as intermediaries.

\subsection{Quadratic complexes of lines from $\CC^4\otimes\Delta_+$}

 \bigskip

Let $a=a_1p_1+a_2p_2+a_3p_3+a_4p_4$ be a general element of the Cartan subspace $\fc$. 
We can see it as a map from $\CC^4$ to  $\Delta_+$, and deduce a morphism
$$\gamma_a: S^2\CC^4\lra V_{10}.$$

\begin{lemma}
$\gamma_a$ is given by the following expression:
$$\begin{array}{rcl}
\gamma_a(x)& = & -((a_1^2+a_2^2)x_2x_4+2a_3a_4x_1x_3)e_1+((a_1^2+a_2^2)x_1x_3+2a_3a_4x_2x_4)f_1\\
& & -(a_2a_3x_1^2-a_1a_4x_2^2-a_1a_3x_3^2+a_2a_4x_4^2)e_2+(a_1a_4x_1^2-a_2a_3x_2^2-a_2a_4x_3^2+a_1a_3x_4^2)f_2 \\
& & -(a_1a_3x_1^2-a_2a_4x_2^2+a_2a_3x_3^2-a_1a_4x_4^2)e_3-(a_2a_4x_1^2-a_1a_3x_2^2+a_1a_4x_3^2-a_2a_3x_4^2)f_3\\
& & -((a_3^2+a_4^2)x_1x_4+2a_1a_2x_2x_3)e_4 -((a_3^2+a_4^2)x_2x_3+2a_1a_2x_1x_4)f_4\\
& & +((a_1^2-a_2^2)x_1x_2-(a_3^2-a_4^2)x_3x_4)e_5+((a_1^2-a_2^2)x_3x_4+(a_3^2-a_4^2)x_1x_2)f_5.
\end{array}$$
\end{lemma}

\medskip
Considered as a quadratic map from $\PP^3$ to $\PP (V_{10})$, 
the image of $\gamma_a$ is contained in the quadric $Q$. For two distinct points $x,y$ in $\PP^3$, recall that the line they generate is mapped by $\gamma_a$ to a conic, whose span is contained in $Q$ whenever the corresponding codimension two section of the spinor tenfold is special. Since  $\gamma_a(x)$ and $\gamma_a(y)$ are isotropic, this happens exactly when 
$$Q(\gamma_a(x),\gamma_a(y))=0.$$
This is a condition on the line $\overline{xy}$, so we should be able to express it in terms of the Pl\"ucker coordinates of the line. Indeed it depends quadratically on $x$ and $y$, it is symmetric in these arguments and vanishes when they coincide; this means that we can see it as a element of the kernel of the  product map from $S^2(S^2\CC^4)$ to $S^4\CC^4$. This kernel is 
$S_{22}\CC^4$, that is, the space of quadratic sections on the Grassmannian $G(2,4)$. Let $\pi_{ij}$ denote the Pl\"ucker coordinates on $G(2,4)$, for $1\le i<j\le 4$. 

\begin{prop}\label{quadratic-section}
The quadratic section of $G(2,4)$ defined by $a$ is 
$$\begin{array}{rcl}
\Theta(a) &= & (a_1^2-a_2^2)(a_3^2-a_4^2)(\pi_{12}^2-\pi_{34}^2) \\
 & & -2(a_1^2+a_2^2)a_3a_4(\pi_{13}^2+\pi_{24}^2) 
 +\frac{1}{3}(2a_1^4+2a_2^4-a_3^4-a_4^4+6a_3^2a_4^2)(\pi_{12}\pi_{34}-\pi_{14}\pi_{23}) \\
& &  +2a_1a_2(a_3^2+a_4^2)(\pi_{14}^2+\pi_{23}^2) 
+\frac{1}{3}(2a_3^4+2a_4^4-a_1^4-a_2^4+6a_1^2a_2^2)(\pi_{12}\pi_{34}+\pi_{13}\pi_{24}).
\end{array}$$
\end{prop}


When $a$ varies, the span of $\Theta(a)$ is a five dimensional subspace $U_5$ of $S_{22}\CC^4$, that we can identify with a subspace of $S^4\fc$; it is an irreducible component of the action of $W_\fc$. An easy computation yields:

\begin{prop}\label{Finv}
$U_5$ is the space of $F$-invariant polynomials in $S^4\fc^\vee$.
\end{prop}

Note that $U_5$ contains the polynomial $a_1a_2(a_3^2+a_4^2)$,  a product of four linear forms that are equations of four reflection hyperplanes. The action of $W_\fc$ yields other products of the same type, each one defining a tetrahedron in $\PP(\fc)$. Computing them, we get the following result.

\begin{prop}\label{quartets}
Up to scalars, $U_5$ contains $15$ products of $4$ reflection hyperplanes and is generated by them. They split the $60$ reflection hyperplanes of $W_\fc$ into $15$ tetrahedra. 
\end{prop}

Here is another direct connection with the reflection hyperplanes.

\begin{prop}
The quadratic section $\Theta(a)$ of $Gr(2,4)$ is singular if and only if $a$ belongs to one of the $60$ reflection hyperplanes.
\end{prop}

\begin{proof}
Denote by $P_{\lambda,\mu}=\lambda \Theta(a) +\mu (\pi_{12}\pi_{34}+\pi_{14}\pi_{23}-\pi_{13}\pi_{24})$ the pencil of hyperquadrics in $\mathbb{P}^{5}$ generated by $\Theta(a)$ and the equation of $Gr(2,4)$. The pencil is $W_{\fc}$-invariant by definition.
The base locus of $P_{\lambda,\mu}$ is singular if and only if the discriminant of the binary form $det(P_{\lambda,\mu})$ vanishes \cite[Prop 2.1]{reid}, a condition that  defines a $W_{\fc}$-invariant hypersurface
in $\mathbb{P}\fc$.
Note that the determinant of $P_{\lambda,\mu}$ can be written as
$$
det(P_{\lambda,\mu})=d_{12,34}(\lambda,\mu) \cdot d_{14,23}(\lambda,\mu) \cdot d_{13,24}(\lambda,\mu),
$$
where $d_{12,34},d_{13,24}$ and $d_{14,23}$ are three quadratic forms in $\lambda,\mu$. Thus the discriminant hypersurface is defined by the polynomial
$$d=disc(d_{12,34})\cdot disc(d_{13,24})\cdot disc(d_{14,23})\cdot R(d_{12,34},d_{13,24})\cdot R(d_{13,24},d_{14,23})\cdot R(d_{14,23},d_{12,34}),$$
where the resultants $R(d_{12,34},d_{13,24}),R(d_{13,24},d_{14,23}),R(d_{14,23},d_{12,34})$ have  degree $32$. This implies that 
$det(P_{\lambda,\mu})$ has degree $120$. Now we compute  that  $$disc(d_{12,34})disc(d_{13,24})disc(d_{14,23})= a_{1}^2a_{2}^2a_{3}^2a_{4}^2(a_{1}^4-a_{2}^4)^2(a_{3}^2-a_{4}^4)^2$$ 
is the square of a product of equations of $12$ reflection hyperplanes. 
So by invariance, since $W_\fc$ acts transitively on the $60$ reflection hyperplanes, we deduce that $
det(P_{\lambda,\mu})$ must be the square of the product of their equations, up to 
non-zero scalar. Hence the claim.
\end{proof}

\noindent {\it Remark.} A computation with GAP confirms that $Sym^4\fc$ splits into 
two irreducible components of dimension $5$ and $30$, respectively (we thank W. de Graaf for this observation.)

\medskip
Now consider $\Theta$ as a rational map from $\PP(\fc)$ to $\PP(U_5)$, defined by 
quartic polynomials. 

\begin{prop}\label{CR4}
 $\Theta: \mathbb{P}(\fc) \rightarrow \PP(U_5)\simeq\PP^{4}$  is a $W_{\fc}$-equivariant morphism of degree $16$ onto its image, which is projectively isomorphic with the Castelnuovo-Richmond quartic $CR_{4}$. 
 The preimages of the $15$ singular lines of the quartic is the union of  $30$ lines. 
\end{prop}

Recall that the Castelnuovo-Richmond quartic \cite[9.54]{Dolg_classical} is also called the Igusa quartic, and is a compactification of the moduli space of principally polarized abelian surfaces with full level two structure. 

\begin{proof}
The first part of the statement follows from Proposition \ref{Finv} and \cite[Prop 10.2.7]{Dolg_classical}. The last part is a direct computation. 
\end{proof}

\medskip\noindent {\it Remark.}
In the basis of Proposition \ref{quadratic-section},  $$t_0(\pi_{12}^2-\pi_{34}^2) +t_1(\pi_{13}^2+\pi_{24}^2) 
 +t_2(\pi_{12}\pi_{34}-\pi_{14}\pi_{23})+t_3(\pi_{14}^2+\pi_{23}^2)+t_4(\pi_{12}\pi_{34}+\pi_{13}\pi_{24}) $$
 belongs to the Castelnuovo-Richmond quartic if an only if 
$$t_0^4+t_1^4+t_3^4+2t_0^2t_1^2+2t_0^2t_3^2-2t_1^2t_3^2-(2t_0^2+t_1^2-2t_3^2)t_2^2-(5t_0^2+t_1^2+t_3^2)t_2t_4-(2t_0^2-2t_1^2+t_3^2)t_4^2=0.$$
We also note that the $30$ lines, more precisely the $15$ pairs of lines in $\PP(\fc)$ mapping to the singular locus of $CR_4$, will appear in the next section as forming the discriminant locus for codimension four sections of the spinor tenfold. 

\medskip
By \cite[p.532]{Dolg_classical}, the quotient of $\CR4$ by its automorphism group $S_6$ is isomorphic with the weighted projective space $w\PP(2,3,5,6)$ (and can also be interpreted as the moduli space of singular cubic surfaces!). This gives a
geometric interpretation of our moduli space.

\begin{coro}
$\Theta$ induces an isomorphism:
$\mathbb{P}\fc/W_{c} \simeq CR_{4}/Aut(CR_{4}) \simeq w\PP(2,3,5,6)$.
\end{coro}

Of course $w\PP(2,3,5,6)\simeq w\PP(8,12,20,24)$, so we recover the exponents of $W_\fc$ up to a  factor $4$, whose main effect is to kill the center $Z(W_\fc)\simeq \FF_4$. 

\smallskip 
Note that instead of the Cartan subspace $\fc$ of $\fg_1\subset \fe_8$, we could have used the dual Cartan subspace $\fc'$ of $\fg_{-1}$, for which we also provided a special basis. We get a  map $\Theta'$ similar to $\Theta$, sending $a'\in \fc'$ to a quadratic section of $G(2,(\CC^4)^\vee)$ which is canonically isomorphic to $G(2,\CC^4)$. Computing $\Theta'$ we deduce the following statement:

\begin{prop} 
For $a=a_1p_1+a_2p_2+a_3p_3+a_4p_4\in\fc$ and $a'=a_1p'_1+a_2p'_2+a_3p'_3+a_4p'_4\in \fc'$,   the quadratic sections
 $\Theta(a)=\Theta'(a')$ coincide. 
\end{prop}

\subsection{Coble quadrics from  $\CC^4\otimes\Delta_+$}
Applying Proposition \ref{Coble}, we can associate to a generic $a\in\fc$ a quadratic section $\cC(a)$ of $Fl:=Fl(1,3,\CC^4)$, whose singular locus will in general be an abelian surface. 
The easiest way to compute an equation of $\cC(a)$ is to deduce it from the quadratic complex $\Theta(a)\in S_{22}\CC^4$. Indeed, we observe that there exists an $SL_4$-equivariant morphism, unique up to non-zero scalar,
$$\Gamma : S^2(S_{22}\CC^4) \lra S_{422}\CC^4\subset S^2\CC^4\otimes S^2(\CC^4)^\vee.$$
Suppose $\Theta$ is in the form of Proposition \ref{quadratic-section}, that is 
$$\Theta = 
P(\pi_{12}^2-\pi_{34}^2)+Q(\pi_{13}^2+\pi_{24}^2)+R(\pi_{14}^2+\pi_{23}^2)+2S\pi_{12}\pi_{34}  + 2T \pi_{14}\pi_{23}-2U\pi_{13}\pi_{24},$$
with $S+T+U=0$. Then a computation yields that its image by $\Gamma$ is of the following form, where $x,y,z,t$ are our coordinates on $\CC^4$, and $\xi, \eta, \zeta, \theta$ their duals: 
$$\begin{array}{rcl}
\cC &=& 2PQ(x^2\otimes \theta^2 +y^2\otimes \zeta^2 -z^2\otimes \eta^2 -t^2\otimes \xi^2)+ \\
& &   +2PR(x^2\otimes \zeta^2 +y^2\otimes \theta^2 -z^2\otimes \xi^2 -t^2\otimes \eta^2)+\\
 & & +2QR(x^2\otimes \eta^2 +y^2\otimes \xi^2 +z^2\otimes \theta^2 +t^2\otimes \zeta^2)+\\
 & & +4(S^2+Q^2-R^2)(xy\otimes \xi\eta+zt\otimes \zeta\theta)+4(T^2+P^2-Q^2)(xt\otimes \xi\theta+yz\otimes\eta\zeta)+\\
 & & +4(U^2+R^2-P^2)(xz\otimes \xi\zeta+yt\otimes \eta\theta)-\\
  & & -(S^2+T^2+R^2)(x^2\otimes \xi^2 +y^2\otimes \zeta^2 +z^2\otimes \eta^2 +t^2\otimes \theta^2)+\\
 & & +4(UT-PS)(xy\otimes\zeta\theta+zt\otimes \xi\eta)+4(TS-QU)(xz\otimes \eta\theta+yt\otimes \xi\zeta) +\\
 & & +4(SU-RT)(xt\otimes\eta\zeta+yz\otimes\xi\theta).
 \end{array}$$ 
 So the map associating to $a\in\fc$ the Coble quadric $\cC(a)$ is the composition of $\Theta$
 with the rational map $\Psi : \PP^4\dashrightarrow \PP^8$ mapping $[P,Q,R,S,T,U]$  to
 \begin{equation}\label{Gamma}
 [QR, PR, PQ, S^2+Q^2-R^2,  T^2+P^2-Q^2, U^2+R^2-P^2, UT-PS, TS-QU, SU-RT].
 \end{equation}
 Note that $\PP^8=\PP(U_9)$ where $U_9$ is a representation of $W_\fc$ spanned by degree eight polynomials on $\fc$. It contains $PQ$ which is a product of $8$ equations of reflection hyperplanes, belonging to two disjoint tetrahedra. We deduce that $U_9$ is again a Macdonald representation. More precisely, applying $W_\fc$ to $PQ$ we deduce: 
 
 \begin{prop} 
 $U_9$ is generated by a $W_\fc$-orbit of $45$ such products of $8$ reflection hyperplanes. 
 \end{prop}
 
Using Proposition \ref{CR4}, we deduce that   $\Gamma$ maps the Castelnuovo-Richmond quartic 
to a family of Coble 
quadrics whose singular loci  are precisely the abelian surfaces parametrized by $CR_4$, over which we would therefore be able to construct a universal family. 

\subsection{Kummer surfaces from $\CC^4\otimes\Delta_+$}

From the quadric section $\Theta(a)$ of $G(2,4)$ we can derive explicitly 
the associated Kummer surface. By construction, We just need to consider a 
hyperplane in $\PP^3$, 
say orthogonal to $[x,y,z,t]$ and restrict to the corresponding plane in $G(2,4)$. 
We thus get a conic and  compute its discriminant. 
For a quadric section of the form of $\Theta(a)$, the equation of the Kummer turns out to be much simpler than one could expect. In order to simplify the computation, we add a multiple of the Plücker equation to put in the following form:
$$\begin{array}{rcl}
\Theta(a) &= & (a_1^2-a_2^2)(a_3^2-a_4^2)(\pi_{12}^2-\pi_{34}^2)
-2(a_1^2+a_2^2)a_3a_4(\pi_{13}^2+\pi_{24}^2) +2a_1a_2(a_3^2+a_4^2)(\pi_{14}^2+\pi_{23}^2) \\
& &  
((a_3^2+a_4^2)^2+4a_1^2a_2^2)\pi_{12}\pi_{34}+((a_3^2-a_4^2)^2-(a_1^2-a_2^2)^2)\pi_{14}\pi_{23}.
\end{array}$$

\begin{lemma}\label{Kummer_from_quadric}
Consider a quadratic section of $G(2,4)$ of the form 
$$\Theta = 
P(\pi_{12}^2-\pi_{34}^2)+Q(\pi_{13}^2+\pi_{24}^2)+R(\pi_{14}^2+\pi_{23}^2)+2S\pi_{12}\pi_{34}  + 2T \pi_{14}\pi_{23}.$$
The equation of the associated Kummer quartic surface in $\PP^3$ is 
$$\begin{array}{rcl}
K & = &-PQR(x^4+y^4-z^4-t^4)
-P(Q^2+R^2-T^2)(x^2y^2-z^2t^2)  \\
& & +Q(R^2-P^2-(S-T)^2)(x^2z^2+y^2t^2) +R(Q^2-P^2-S^2)(x^2t^2+y^2z^2)  \\
 & &
+2(S(R^2-Q^2)+T(P^2+Q^2)+ST(S-T))xyzt. 
\end{array}$$
\end{lemma}

Beware that in this Lemma, $\Theta$ does not belong to $S_{22}\CC^4$ in general; this simplifies a little bit the computation below. But of course we can always add a suitable multiple of the Pl\"ucker quadric in order to get into   $S_{22}\CC^4$; this does not affect the corresponding quadratic section of $G(2,4)$, hence neither, a fortiori, the associated Kummer surface.

\proof When we restrict to the hyperplane orthogonal to $[x,y,z,1]$, the Pl\"ucker coordinates verify the relations 
$$\pi_{14}=y\pi_{12}+z\pi_{13}, \quad \pi_{24}=-x\pi_{12}+z\pi_{23}, \quad 
\pi_{34}=-x\pi_{13}-y\pi_{23}.$$
Substituting, we get a quadratic form in $\pi_{12}, \pi_{13}, \pi_{23}$ that is singular when 
$$K=\det \begin{pmatrix} P+x^2Q+y^2R & yzR-xS & xzQ-y(T-S) \\
yzR-xS & Q-x^2P+z^2R & xyP-zT \\
xzQ-y(T-S) & xyP-zT & R-y^2P+z^2Q \end{pmatrix} = 0.$$
Expanding the determinant yields the result. \qed

\medskip
Suppose we write the equation of the Kummer surface as 
$$K= A(x^4+y^4-z^4-t^4)+B(x^2y^2-z^2t^2)+C(x^2z^2+y^2t^2)+D(x^2t^2+y^2z^2)+2Exyzt.$$
This is almost Hudson's canonical form (up to some signs that we could get rid of by multiplying $z,t$ by some eighth root of unity). In order that this is the equation of a Kummer surface, the 
coefficients $A,B,C,D,E$ are subject to the cubic relation
$$4A^3-(B^2-C^2-D^2+E^2)A+BCD=0,$$
which is our equation for the Segre cubic primal $\S3$ \cite[Theorem 10.3.18]{Dolg_classical}. 

\smallskip
Applying Lemma \ref{Kummer_from_quadric} to  $\Theta(a)$, we deduce the equation $K(a)$ of the  corresponding Kummer surface. 

\begin{prop}\label{Kummer-equation} The equation of the Kummer surface defined by $a\in\fc$ is of the form 
\begin{equation}\label{kummerCartan}
\begin{array}{rcl}
K(a)& = & A(a)(x^4+y^4-z^4-t^4)+B(a)(x^2y^2-z^2t^2)+C(a)(x^2z^2+y^2t^2)\\ 
 & & \qquad +D(a)(x^2t^2+y^2z^2)+2E(a)xyzt.
\end{array}
\end{equation}
where $A(a), B(a), C(a), D(a), E(a)$ are degree $12$ polynomials  in $a$, explicited in Appendix \ref{combinatorics}.
\end{prop}

As for Proposition \ref{CR4}, we can interprete $K$ as a degree twelve equivariant map from $\fc$ to $S^4(\CC^4)^\vee$, or equivalently, to a linear equivariant morphism from $S^{12}\fc$ to $S^4(\CC^4)^\vee$. The latter morphism has rank $5$, and yields an isomorphism between a five dimensional subspace $V_5$ of $S^{12}\fc^\vee$, and a five dimensional subspace of $S^4(\CC^4)$, that we denote in the same way. 
This representation of $W_\fc$ contains in particular the polynomial 
$$A(a)= 4a_1a_2a_3a_4(a_1^4-a_2^4)(a_3^4-a_4^4),$$
which is a product of linear equations of $12$ among the $60$ reflection hyperplanes of $W_\fc$. The action of the Weyl group yields many more similar products inside $U_5$. W. de Graaf checked the following statement:


\begin{theorem}\label{blocks}
Up to scalars, $V_5$ contains $15$ products of $12$ reflection hyperplanes and is generated by these. Moreover, there are exactly six ways to choose five such products so as to cover the $60$ reflection hyperplanes of $W_\fc$.
\end{theorem}

\noindent {\it Remark.} To construct representations by making a finite group act on polynomials is a classical idea, that goes back at least to the construct of representation of symmetric groups as Specht modules. This was generalized to Weyl groups by Macdonald \cite{macdonald}, and partly to complex reflection groups in 
\cite{morris-mwamba}. Actually Proposition \ref{quartets} already provided another 
instance of such a Macdonald representation $U_5$ of $W_\fc$, also of dimension five.

\medskip\noindent {\it Definition.}\label{def-blocks-pentads} We call {\it block} each of the $15$ groups of  $12$ reflection hyperplanes mentionned in Theorem \ref{blocks}. They are denoted $D_1,\ldots , D_{15}$ and can be found  in Appendix  \ref{combinatorics}. We call  {\it pentad} each splitting of the $60$ reflection hyperplanes into $5$ 
blocks. The six pentads are denoted $P_1,\ldots , P_6$ and can be found in the same Appendix.

\medskip
Since $W_\fc$ permutes the reflection hyperplanes, it has to permute the six pentads.
Each of the $15$ blocks  is the intersection of exactly two pentads, so the stabilizer of the $6$ configurations is also the stabilizer of the $15$ blocks. 
This observation allows to deduce:



\begin{prop}\label{F0}
The subgroup of $W_\fc$ acting trivially on $\PP(U_5)$ is $F$, and 
$\Theta$  coincides with the quotient of $\PP(\fc)$ by $F$. This  induces the identification 
 $$W_{\fc}/F \simeq Aut(\CR4)\simeq Aut(\S3) \simeq S_{6}.$$
 Moreover the quotient of $W_\fc$ by its center is isomorphic to $N$, and this yields an exact sequence
$$1\lra F_0\lra N\lra S_6\lra 1.$$
\end{prop}

\proof The first claim is an explicit computation. It implies that $W_\fc/Z(W_\fc)$ normalizes $F_0$, and therefore coincides with $N$ since this is exactly how it is defined in \cite{gd}. Moreover 
 $F_0$ is also the kernel of the map from $N$ to the permutation group of the six configurations, and since $\# N =16\times 6!$, this map must be surjective. Moreover the projective action of $N$  clearly preserves the Segre cubic, so we get an embedding of $S_6$ into $Aut(\S3)\simeq
 Aut(\CR4)\simeq S_6$ \cite[Lemma 5]{mukai-igusa}, hence an isomorphism. \qed 



\medskip The combinatorics of blocks and pentads  recovers the classical combinatorics of $S_6$. Indeed, since each block belongs to exactly two pentads,   
we can identify the $15$  blocks with the $15$ pairs of integers, from $(12)$ to $(56)$. Recall that $S_6$ admits an outer automorphism which can be defined in terms of synthemes, which are triples of disjoint pairs, and pentads (or totals) of disjoint synthemes (see e.g. \cite{syntheme}): there are exactly six totals and the induced action of $S_6$ yields an outer automorphism. In our setting, totals correspond to what we call {\it quintets}, which  are 
obtained by splitting the fifteen blocks into five triples, in such a way that three blocks in a same triple never belong to the same pentad. The six quintets are listed in Appendix \ref{combinatorics}. 

In turns out that on these six quintets, the action of 
the generators $s_1,\ldots ,s_5$ of $W_\fc$ given in section \ref{cartan} is particularly simple, since they act by simple transpositions. 

\begin{lemma}\label{totals}
The action of $s_i\in W_\fc$ on the six quintets is a permutation 
$s_i^T$, with 
$$s^T_1=(12), \quad s^T_2=(56), \quad s^T_3=(34),\quad s^T_4=(23), \quad s^T_5=(46).$$
\end{lemma}

The induced action on pentads is more complicated but can be recovered as follows: two quintets intersect along a unique triple of blocks, that split the six pentads into three pairs. This maps a simple transposition to a product of three disjoint transpositions. For example, quintets $4$ and $6$ share the triple $2-10-11$, 
which splits the pentads into $P_3-P_4$, $P_2-P_5$, $P_1-P_6$; so that the action of $s_4$ on the pentads must be given by $s_5^P=(34)(25)(16)$. \medskip

\begin{lemma} 
The action of $s_i\in W_\fc$ on the six pentads is a permutation 
$s_i^P$, with 
$$s^P_1=(12)(34)(56), \quad s^P_2=(12)(36)(45), \quad s^P_3=(12)(35)(46),$$
$$ s^P_4=(14)(25)(36), \quad s^P_5=(16)(25)(34).$$
\end{lemma} 

The consequence of these two Lemmas  is that the actions of $W_\fc$ on $U_5$ and $V_5$ factorize through $S_6$ and are exchanged by its outer automorphism. The stabilizers are different; we describe the stabilizer of a pentad as follows.  

\begin{prop}
The stabilizer in $W_{\fc}$ of a given pentad is a central extension by $\FF_2$ of the spin cover $\widetilde{W}(D_5)$ of the Weyl group $W(D_5)$.
\end{prop}

\noindent {\it Remark.} Spin covers of Weyl groups are closely related with their 
projective representations, a very classical object of study since it goes back to Schur for the symmetric groups; see \cite[(2.4.4)]{ciubotaru} for the general definition, and for more references. In our case, note that because of the 
exceptional isomorphism $Spin_5\simeq SL_4$, the spin cover $\widetilde{W}(D_5)$ admits an irreducible representation of dimension $4$, contrary to the Weyl group $W(D_5)$
itself, whose non trivial irreducible representations have dimension $5$ and more. In particular $W(D_5)$ cannot be a subgroup of $W_\fc$. 

\proof Let $s_6 = s_2s_3, s_7 = s_4s_5$ and $s_8 = s_1s_6$. Using the previous Proposition, we check that the elements $s_8, s_6s_7s_8s_7s_6, s_8s_5s_6s_5$ and 
$s_3s_7s_8s_7s_8s_7$ of $W_\fc$ 
fix the first pentad. Then a computation with GAP shows that the subgroup of $W_\fc$ generated by these elements has index $6$, so that it must be the full stabilizer of the pentad. Of course it contains the center $\FF_4$, and GAP also confirms that the quotient of $S$ by the order two subgroup $\FF_2\subset\FF_4$
 is isomorphic with  $\tilde{W}(D_5)$. \qed 
 
\medskip Let us come back to the statement of Theorem \ref{blocks}. The second
part implies that we can find six basis of $U_5$ in which the equation of the Kummer surface associated to $a\in\fc$ has a specially simple form:

\begin{theorem}\label{parametrization}
For each pentad $P_i = D_{i_1}-D_{i_2}-D_{i_3}-D_{i_4}-D_{i_5}$, the rational map
$$[a_1,\ldots , a_4] \mapsto \left[\prod_{j\in D_{i_1}}\ell_j,\prod_{j\in D_{i_2}}\ell_j,\prod_{j\in D_{i_3}}\ell_j,\prod_{j\in D_{i_4}}\ell_j,\prod_{j\in D_{i_5}}\ell_j \right]$$
maps $\PP^3$ to the Segre cubic $\S3\subset\PP^4$.
\end{theorem}

\noindent {\it Remark.} This statement is reminiscent of a construction of the Segre cubic as the moduli space of six ordered points on $\PP^1$, or equivalently the GIT quotient $(\PP^1)^6/\hspace{-1mm}/ PGL_2$ with respect to the "democratic" polarization. Given six points $p_i=[x_i,y_i]$, let $\pi_{jk}=x_jy_k-x_ky_j$. For any total $T=(S_1,\ldots ,S_5)$ of five synthems, the rational map 
$$(p_1,\ldots ,p_6) \mapsto \left[\prod_{(jk)\in S_1}\pi_{jk}, \prod_{(jk)\in S_2}\pi_{jk}, \prod_{(jk)\in S_3}\pi_{jk}, \prod_{(jk)\in S_4}\pi_{jk}, \prod_{(jk)\in S_5}\pi_{jk}\right],$$
is $PGL_2$-invariant and its image is the Segre cubic. Note that this identifies $U_5$ with the Specht module $[222]$ of $S_6$. This was first observed by Joubert in 1867, see \cite{syntheme} or \cite[Theorem 9.4.10]{Dolg_classical} 
for a modern treatment.

\medskip
In other words, for each pentad $P_i$, the exists five quartics $S_i^1,\ldots , S_i^5$ such that any Kummer quartic in $U_5$ can be written as 
$$K= \Big(\prod_{j\in D_{i_1}}\ell_j\Big) S_i^1+\Big(\prod_{j\in D_{i_2}}\ell_j\Big) S_i^2+\Big(\prod_{j\in D_{i_3}}\ell_j\Big) S_i^3+\Big(\prod_{j\in D_{i_4}}\ell_j\Big) S_i^4+\Big(\prod_{j\in D_{i_5}}\ell_j\Big) S_i^5.$$

These six quintuples of quartics are expressed in Appendix \ref{combinatorics} in terms of 
$$K_1=x^4+y^4-z^4-t^4, K_2=x^2y^2-z^2t^2, K_3=x^2z^2+y^2t^2, K_4=x^2t^2+y^2z^2, K_5=2xyzt.$$

\begin{prop}
The quartics $S_i^j$ are projectively isomorphic to the same Kummer surface  $K_0$.
Moreover, the abelian surface whose quotient is  $K_0$ 
is the Jacobian of the Bolza curve; it decomposes as a product of two copies of the elliptic curve   $\CC/\ZZ+i\sqrt{2}\ZZ$. \end{prop}

\noindent {\it Remark.} A similar statement holds for the quadric sections of $G(2,4)$, with $K_0$ replaced by a quadric section $Q_0$ which is the union of two planes.

\medskip
Recall that the Bolza curve is the genus two curve with largest automorphism group, isomorphic with $GL_3(\FF_2)\rtimes \ZZ_2$. Presumably, its symmetries are induced by the action of $W_\fc$.

\proof Multiplying $z,t$ by fourth roots of unity shows that $S_1^1$ is projectively isomorphic to the surface  $K_0'$ of equation 
$$x^4+y^4+z^4+t^4-4xyzt=0,$$ 
while $S_1^8, S_1^{11}, S_1^{13}, S_1^{14}$ are projectively isomorphic to  $K_0''$ of equation
$$x^2y^2+z^2t^2+x^2z^2+y^2t^2+x^2t^2+y^2z^2+2xyzt=0.$$
Now, that $K'_0$ and $K''_0$ are projectively isomorphic follows from the identity 
$$X^4+Y^4+Z^4+T^4-4XYZT=  32(x^2y^2+z^2t^2+x^2z^2+y^2t^2+x^2t^2+y^2z^2+2xyzt)$$
if  $X=x+y+z+t$, $Y=x-y-z+t$, $Z=x-y+z-t$, $T=x+y-z-t$.

\smallskip
Since we get the other quartics $S_i^j$ through the linear action of $W_\fc$, 
they are all projectively isomorphic to  $K'_0$ and $K''_0$. 
That the corresponding abelian surface is the Jacobian of the Bolza curve was checked by X. Roulleau with the help of Magma. \qed 

\medskip\noindent {\it Remark.}
The Kummer quartic $K_0$ is of Segre type, in the sense that its equation can be expressed as a quadratic form $Q$ in the squares of the variables, after some change of coordinates. When this happens for a smooth quadric $(Q=0)$, by \cite[Proposition 15]{catanese} the associated abelian  surface $A$ is isogenous to a product of elliptic curves: more precisely, $A=(E_1\times E_2)/H$, with $H\simeq\ZZ_2^2$. Here, starting from the equation of $K'_0$ and letting $x=X+Y$, $y=X-Y$, 
$z=Z+T$, $t=Z-T$ we get in terms of $a=X^2, b=Y^2, c=Z^2, d=T^2$ the equation
$$Q(a,b,c,d)=2(a+b)^2+2(c+d)^2-(a-b+c-d)^2=0,$$
so that the quadric surface $Q=0$ is a cone and Catanese's result does not apply.

\medskip
Next we discuss the rational map $K$ defined by Proposition \ref{Kummer-equation}, sending a point $a$ of $\fc$ to the Kummer surface of equation $K(a)$. We can interprete  Lemma \ref{Kummer_from_quadric} as defining cubic rational map from $\mathbb{P}(U_5)$ to $S_3$; we denote it by $\widetilde{K}$. Its composition with $\Theta$ recovers $K$. By Proposition \ref{CR4} the map factors through the Igusa quartic. In a summary we get:
$$K: \mathbb{P} (\fc) \xrightarrow{\Theta} CR_{4} \stackrel{\widetilde{K}_{|_{CR_{4}}}}{\dashrightarrow} S_3. $$
According to Macaulay the rightmost map from $\CR4$ to $S_3$ has degree $16$. 
It is defined by a linear system of cubics whose base locus is the union of $15$ lines, and it has to coincide with the degree $16$ rational map from $\CR4$ 
to $\S3$ that appears in \cite[Theorem 8.10]{Igusa16}.

\medskip
An amazing property of $\CR4$ is that its intersection with its tangent space at a generic point is again a Kummer surface! Choose a basis $Q_0,\ldots ,Q_4$ of quadrics in  $U_5$ and write $\Theta(a)=t_0(a)Q_0+\cdots +t_4(a)Q_4,$ where $t_0,\ldots , t_4$ are quartic polynomials on $\fc$. The equation of $\CR4$ is 
$(F(t_0,\ldots ,t_4)=0)$ for some quartic polynomial $F$ (see the Remark after Proposition \ref{CR4}. Its tangent hyperplane at $[\Theta(a)]$ has equation $\sum_i\partial F/\partial t_i(\Theta(a))t_i=0$. So its inverse image by $\Theta$ is a quartic surface $K_\fc(a)$ in $\PP(\fc)$, of equation 

\begin{equation}\label{kfc}
\sum_i\partial F/\partial t_i(\Theta(a))Q_i(x)=0, \qquad x\in\fc.
\end{equation}

\begin{lemma} 
 $K_\fc(a)$ is a Kummer quartic surface in  $\PP(\fc)$, singular along the $16$ points of the fiber  $\Theta^{-1}(\Theta(a))$. 
\end{lemma} 

\proof Since $K_\fc(a)$ only depends on $\Theta(a)$, it is enough to prove that
it is singular at $a$. But this is straightforward: 
$$\partial K_\fc/\partial x_j(a)=\sum_i\partial F/\partial t_i(\Theta(a))\partial Q_i/\partial x_j(a)=\partial (F\circ\Theta)/\partial x_j(a)=0$$
since $F\circ\Theta$ is identically zero. \qed 

\medskip\noindent {\it Remark.} Formula (\ref{kfc}) shows that the Kummer surfaces 
$K_\fc(a)$ are parametrized by the projective dual variety to $\CR4$, so we recover the 
classical fact that this dual variety is nothing else than the Segre cubic $\S3$ \cite[Proposition 9.4.18]{Dolg_classical}. 
\medskip

Let us summarize some results we have obtained so far. 

The morphism 
$\Theta : \PP(\fc)\lra\CR4\subset \PP(U_5)$ mapping a  codimension four section of the spinor tenfold to the associated quadratic complex of lines is a degree $16$ morphism.  
The Castelnuovo-Richmond quartic $\CR4$ is the Satake compactification of the moduli space of principally polarized abelian surfaces with full level $2$ structure; 
we recover the abelian surface as the singular locus of a Coble quadric in a five-dimensional flag manifold, given by a morphism $\cC : \PP(\fc)\lra \PP(U_9)$. 

The quotient of $\CR4$ by $Sp_4(\FF_2)\simeq S_6$ is the moduli space of Kummer surfaces and can be identified with the GIT-moduli space 
$\PP(\fc)/W_\fc$ of codimension four sections of the spinor tenfold. 
Moreover, if we take the dual Cartan subspace $\fc'$ in the story, we get a map
$$\PP(\fc)\times_{\CR4}\PP(\fc') \stackrel{(\Theta,\Theta')}{\lra}\CR4$$
whose general fiber is a Kummer configuration in $\PP(\fc)\times\PP(\fc')\simeq \PP^3\times \check{\PP}^3.$
The rational map $K:~\PP(c)\dashrightarrow\S3\subset \PP(V_5)$ mapping a codimension four section of the spinor tenfold to the associated Kummer surface has degree $256$. The Segre cubic and its quotient by $S_6$ are the moduli spaces of $6$ ordered and unordered points in $\PP^1$, respectively; they are birational to $\CR4$ and its quotient, and provide distinct compactifications of the moduli space. A natural question would be to compare them with the moduli space of K-polystable varieties in the family.  

The three representations $U_5, U_9, V_5$ of $W_\fc$ are induced by the representations of its quotient $S_6$ constructed as the Specht modules $[5,1], [4,2], [3,3]$. They parametrize 
polynomials of respective degrees $4, 8, 12$, in perfect agreement with the fact that 
$[4,2]\subset Sym^2[5,1]$ and $[3,3]\subset Sym^3[5,1]$.


\section{Codimension four sections: classification}

In this section we use the full strength of Vinberg's theory, that allows to classify, in similar terms as for the usual Jordan theory of endomorphisms, orbits in a graded piece of a semisimple Lie algebra. Elements of this theory are reminded in Appendix \ref{jordan}. To summarize it in a few words: semisimple elements are parametrized by the Cartan subspace, up to the action of the little Weyl group;
their types are determined by their position with respect to the reflection hyperplanes; for each semisimple part there are finitely many possibilities, up to conjugation, for the nilpotent part. 

For $\CC^4\otimes\Delta_-$, the full classification was devised in \cite{degraaf}. We deduce a classification of codimension four sections of the spinor tenfold by considering only tensors of maximal rank. And we easily extract 
a classification of smooth sections by using Lemma \ref{singular-section}. 
This simplifies the situation a lot; for example, among the $145$ nilpotent orbits, only $3$ give smooth sections. 

We will start by classifying the semisimple sections, and characterizing the discriminant locus that parametrizes the singular ones. Then we will classify 
arbitrary sections according to the type of their automorphism groups. Finally we will study in some detail a few sections with many automorphisms. 

\subsection{Flats and the discriminant}
According to \cite{st}, the hyperplane arrangement in $\PP(c)$ formed by the $60$ reflection hyperplanes  of $W_\fc$ is nothing else than the classical Klein arrangement. It stratifies $\PP(\fc)$ according to the intersections of hyperplanes, yielding special lines and points at which pass an exceptional number of such 
hyperplanes; such non-empty intersections in a hyperplane arrangement are commonly called {\it flats}. Two points on a same flat, and not belonging to any smaller flat, have the same stabilizer. 

Up to conjugacy we get nine types, including the origin and the complement of the reflection hyperplanes, the reflection hyperplanes, plus three types of projective lines and three types of points, each forming an orbit of $W_\fc$ (acting on planes, lines or points).  We present them in the following table, with a basis of one member of each orbit, the number of translates and the number of hyperplanes in which they are contained, that we call  valency  (see e.g. \cite{st, pokora}); we also include the type of the generic stabilizer, which was 
identified in \cite{degraaf} (whose numbering we follow); by $T_k$ we mean a $k$-dimensional torus.

\begin{equation}\label{flats}
  \begin{array}{ccccc} 
  \mathrm{No} & \mathrm{Basis} & \mathrm{Valency} & \mathrm{Type} & \mathrm{Number} \\
  1 & p_1,p_2,p_3,p_4 &0&1 &\\
  2 & p_1,p_2,p_3&1&T_1 & 60 \\
  3 &p_2,p_3&3& T_2 & 360 \\
  4 &p_2-p_3,p_3-p_4&2& A_1& 320 \\
  5 &p_1+p_2+p_3&15&A_1\times T_1& 960 \\
  6 &p_1,p_2 &6& A_1\times T_3 & 30 \\
  7 &p_2+p_3&6& A_1\times A_1& 480 \\
  8 &p_3+p_4&4& A_2\times A_1\times A_1\times T_1& 60 \\
  9 &0&&D_5\times A_3&
  \end{array}
  \end{equation}
In the sequel we denote by $F_k$ the union of flats of type $k$. 
  
\bigskip
Let us denote by $\Delta_G\subset G(4,\Delta_-)$ the discriminant locus parametrizing singular codimension four sections of the spinor tenfold, and by $\Delta_P\subset \PP(\CC^4\otimes\Delta_-)$ the corresponding locus. 

\begin{prop} \label{30lines}
Both $\Delta_G$ and $\Delta_P$ are irreducible, but 
$$\Delta_P\cap \PP(\fc)=\bigcup_{\ell\in F_6}\ell$$
is the union of $30$ projective lines in $\PP(\fc)\simeq\PP^3$.
\end{prop}


\medskip\noindent {\it Remark.} The $30$ lines  are permuted by $W_\fc$ and must therefore constitute class 6 in Table (\ref{flats}). 
Each line from $F_6$ contains six of the $60$ points of the smallest orbit $F_8$ given by class 8 in Table (\ref{flats}). 
Conversely, every such point is contained in three
lines from $F_6$; hence an interesting $(60_3,30_6)$ configuration (see \cite{dolg-config} for a  survey on configurations in algebraic geometry).

Moreover, the $30$ lines in $F_6$ can be paired in such a way that the three tetrahedra in each block of Theorem \ref{blocks} have two vertices on each. This gives $12$ points that should be dual to the $12$ hyperplanes in the block. Note also that according to \cite[Remark 1.66]{gd}, one can recover $N\subset PGL_4$ as the subgroup 
permuting a collection of $15$ pairs of lines. 

\proof 
Note that from our expression of the Cartan subspace $\fc$, every $p \in \fc$ is of rank four.
Consider the set $Z$ parametrizing pairs $([s],[p])$ in $\PP^3\times\PP(\fc)$ such that 
$p(s)\in \Delta_+$ is a pure spinor. According to Lemma \ref{singular-section}, the projection of 
$Z$ in $\PP(\fc)$ is exactly the intersection of the latter with $\Delta_P$.  The ten quadrics that cut out the spinor tenfold 
yield ten biquadratic equations in $\PP^3\times\PP(\fc)$. According to Macaulay, the scheme defined by these equations has pure dimension one and degree $240$. Indeed, we check that it contains $[s_i]\times\ell_0$ for $\ell_0$ the line $[* * 0 0]$, and also $d\times p_0$ for $p_0=[1000]$ and 
$d=[**00]$ or $[00**]$. From the $W_\fc$-invariance we deduce that the projection of $Z$ in $\PP(\fc)$ contains the $30$ lines for $F_6$. Moreover, for each such line $\ell$ there are four "horizontal lines" in $Z$ of the form $[s]\times\ell$. Since $p_0$ is the intersection point of two lines in $F_6$, we also deduce that $Z$ contains two "vertical lines" of the form  $d\times p$ for each point $p\in F_8$. Hence a total of $120$ horizontal and $120$ vertical lines, whose union must be the whole $Z$. \qed  

\medskip
Let us describe the singular locus and the automorphism group of $X_{K}$ when $K$ is associated to some $p \in \Delta_{P} \cap \mathbb{P}(\fc)$. A computation yields the following 

\begin{lemma}
Let $p\in \PP(\fc) \cap \Delta_{P}$, and let $K$ be its image in $\Delta_-$. \par 
(1) Suppose that $p$ belongs 
to a unique line from $F_6$. Then $Aut(X_{K})$ contains a subgroup isomorphic to $SL_{2} \times T_{3} \times  \ZZ_{2} \times \ZZ_{2}$. The singular locus of $X_K$ is the disjoint union of four lines, on which  $Aut(X_{K})$ acts transitively.\par 
(2) Suppose that $p$ belongs to $S_{0}$. Then $Aut(X_{K})$ contains a subgroup which is isomorphic to $SL_{2} \times SL_{3} \times \ZZ_{2}$.
The singular locus of $X_{K}$ is the disjoint union of two copies of $\mathbb{P}^{1} \times \mathbb{P}^{2}$,on  which  $Aut(X_{K})$  acts transitively.
\end{lemma}

\begin{coro}
$\Delta_G$ and $\Delta_P$ have codimension two.  
\end{coro}

\proof A dimension count shows that this codimension is $c=m+1$, where $m$ denotes the dimension of the singular locus of a generic singular section. By the previous Lemma, $m\le 1$ and $c\le 2$. 
But $c=1$ would imply that $\Delta_P$ is a hypersurface, whose intersection with $\PP(\fc)$ would also be a hypersurface, contradicting Proposition \ref{30lines}.\qed 

\medskip
Now suppose that $p$ belongs to the line $\ell$, and is generic on this line. A computation yields:

\begin{lemma}
The centralizer of $p$ has type $A_1\times T_3$. The set $N_p$ of nilpotent elements commuting 
with $p$ has dimension $6$ and two components $N^6_3$ and $N^6_4$ of maximal dimension, whose general elements have rank three or four, respectively. 
\end{lemma}

Since $A_1\times T_3$ has dimension $6$, the conjugates $p$ of elements in $\ell$ cover some subvariety of $\PP(\CC^4\otimes\Delta_-)$ of dimension $1+60-6=55$. But $\Delta_P$ must admit a dense 
subset of elements of Jordan form $p+n$. Since its dimension is $61$, for $p$ generic we must 
have a dimension $6$ family of nilpotent elements $n$ such that $[p+n]$ belongs to $\Delta_P$, 
an this family must therefore be a union of components of maximal dimension of $N_p$. A computation yields:

\begin{lemma} 
For $n$ generic in $N^6_4$, the section $X_K$ defined by $p+n$ is smooth. For $n$ generic in $N^6_3$, the section $X_K$ is singular along a single line.
\end{lemma}

We deduce the following description of the discriminant locus:

\begin{prop}
$\Delta_P$ is the closure of the set of elements of Jordan form $p+n$, where the semisimple part 
$p$ belongs to a conjugate of some line $\ell\subset\PP(\fc)$ from $F_6$, and the nilpotent part 
$n$ has rank three. 
\end{prop}


\subsection{Classification of smooth sections}
  As we already recalled,   Jordan-Vinberg's  theory states that  
  every element $x\in \CC^4\otimes \Delta_-$ admits a Jordan decomposition $x=p+n$. The semisimple part $p$ admits a $G_0$-conjugate in $\fc$, unique up to 
  $W_\fc$-conjugacy. So we may suppose that $p$ belongs to $\fc$. Let $L\subset \fc$ be a flat of minimal dimension containing $p$; then the stabilizer $Z(p)$ of $p$ 
  in $G_0$ does only depend on $L$, and its action on the nilpotent cone in $\fz(p)$ has finitely many orbits. 
  Choosing a representative $n$ of each of these orbits, we get representatives $x=p+n$ of elements whose semisimple part is conjugate to $p$. The result can be found in \cite{degraaf}. 
  
  This applies typically to the nilpotent cone in $\CC^4\otimes \Delta_+$, corresponding to $p=0$. The nilpotent cone of a semisimple Lie algebra is always irreducible, but this is not necessarily the case in the graded case.

\begin{prop}
The nilpotent cone in $\CC^4\otimes \Delta_+$ has $145$ orbits of $G_0$ (including the trivial one). It is the union of two irreducible components, both  of
dimension $60$. 

The generic element of one of these components defines a singular section of the spinor tenfold. The complement of this component in the other one is the union of only three 
$G_0$-orbits, each of which define smooth sections of the spinor tenfold. Representatives of these three orbits are:
$$\begin{array}{rcl}
n_1 &=& a_1\otimes (e_{12}+e_{1345})+a_2\otimes (e_{13}+e_{24})+a_3\otimes (e_{15}+e_{2345})+a_1\otimes (1+e_{1245}), \\
n_2 &=& a_1\otimes (e_{12}+e_{1345})+a_2\otimes (e_{13}+e_{24})+a_3\otimes (e_{15}+e_{23})+a_1\otimes (1+e_{1245}), \\
n_3 &=& a_1\otimes (e_{12}+e_{1345})+a_2\otimes (e_{13}+e_{24})+a_3\otimes (e_{23}+e_{45})+a_1\otimes (1+e_{1245}).
\end{array}$$
\end{prop}

\proof The first part of the statement is due to de Graaf \cite{degraaf}. 
 According to his computations, 
generic elements of the two irreducible components $\cN_1$ and $\cN_2$ of the nilpotent cone are 
$$\begin{array}{rcl}
n_0 &=& a_1\otimes (e_{12}+e_{1345})+a_2\otimes 
(e_{34}-e_{1234})+a_3\otimes (e_{2345}-e_{15})-a_4\otimes (e_{23}+e_{45}), \\
n_1 &=& a_1\otimes (e_{12}+e_{1345})+a_2\otimes (e_{13}+e_{24})+a_3\otimes (e_{15}+e_{2345})+a_1\otimes (1+e_{1245}).
\end{array}$$
The coefficient of $a_2$ in $n_0$ is a pure spinor, so the associated section is singular by Lemma \ref{singular-section}. On the other hand, an explicit computation shows that $n_1$ yields a smooth section. 

By semicontinuity, each nilpotent element in $\cN_1$ defines a singular section. 
The closure diagram of nilpotent orbits determined in \cite{degraaf} shows that 
the complement of $\cN_2$ in $\cN_1$ is the union of three orbits of dimension 
$60$, $59$ and $58$. Using their representatives provided by loc. cit., and reproduced above, another
explicit computations shows that they all define smooth sections of the spinor tenfold. \qed  

\medskip 
In particular, we see that even if a semisimple element $p$ defines a singular section of the spinor tenfold, it may very well happen that adding a general enough nilpotent element $n\in \fz(p)$ we get a smooth element $x=p+n$. From the explicit classifications of \cite{degraaf} we can describe all the orbits of nilpotent elements, for $p$ in each semisimple class, that finally give a smooth element $x=p+n$. A case-by-case verification yields the following list, where we only indicated (on the last column), for each type of semisimple part (a generic element in the flat with basis from the second column), the number of possible nilpotent parts, up to conjugation.  
\begin{equation}
  \begin{array}{cccc} 
  \mathrm{No} & \mathrm{Basis} & \mathrm{Nilpotent\,orbits} & \mathrm{Smooth\,ones}  \\
  1 & p_1,p_2,p_3,p_4 &1&1 \\
  2 & p_1,p_2,p_3&2&2 \\
  3 &p_2,p_3&4&4 \\
  4 &p_2-p_3,p_3-p_4&3 &3 \\
  5 &p_1+p_2+p_3&6&6 \\
  6 &p_1,p_2 &14&1 \\
  7 &p_2+p_3&5&5 \\
  8 &p_3+p_4&36&3 \\
  9 &0&145&3
  \end{array}
  \end{equation}
  

\medskip 

\subsection{Automorphism groups}

If $x\in \CC^4\otimes\Delta_-$ is a rank four element, with image $K\subset\Delta_-$, then the stabilizer of $[x]$ in $G_0=SL_4\times Spin_{10}$ maps surjectively,  with finite kernel,  to the stabilizer of $K$ in $Spin_{10}$. Recall that by Proposition \ref{auto}, the latter is exactly the automorphism group of $X_K$. 

\begin{lemma}
If $x$ is not nilpotent, then $Stab^{o}_{G_{0}}(x)=Stab^{o}_{G_{0}}([x])$. If $x$ is nilpotent, then $Stab^{o}_{G_{0}}([x])/Stab^{o}_{G_{0}}(x)=\mathbb{G}_{m}.$
\end{lemma}

\begin{proof}
If $x$ is nilpotent, then by the Jacobson-Morozov theorem, there exists a one-parameter subgroup in $G_{0}$ which acts on $x$ by scalar multiplications. This implies that
$$Stab^{o}_{G_{0}}([x])/Stab^{o}_{G_{0}}(x)=\mathbb{G}_{m}.$$
If $x$ is not nilpotent, let  $x=p+n$ be its Jordan decomposition in $\fg_{1}$, with $p \not =0$. To prove our claim, by the uniqueness of the decomposition, we can assume that $x=p$ is semi-simple. Then it follows from the fact that $\fc_{\fg}(p)=N_{\fg}([p])$ for a semisimple element $p$ in a semisimple Lie algebra $\fg$.
\end{proof}

\medskip 
From the classifications of \cite{degraaf} we can then deduce the Table below, listing all possible non-trivial automorphism groups (at least up to finite groups) of smooth sections of the spinor tenfold. 
Stratifications of such representations as $V=\mathbb{C}^{4} \otimes \Delta_{-}$ by stabilizer groups are discussed in Appendix \ref{jordan}.

For each type of automorphism group, the corresponding stratum is the union of one or several Jordan classes, as indicated on the corresponding row. For example, $F_{4,n}$ on the row of $\mathbb{G}_m$ means that there is one Jordan class giving rise to such a stabilizer, for which the semisimple part has type 4  in Table (\ref{flats}), and the nilpotent part is non-trivial; the symbol $3$ just below indicates that this Jordan class has codimension three.  

We would like to know, among other things, if smooth sections with $\GG_m$-action  are parametri\-zed by an irreducible subset of $G(4,\Delta_-)$; but we don't know which classes are contained in the closure of the codimension two Jordan class of $F_2$. What we know for sure is that for smooth sections with $\GG_a$-action, there are either 
three or four components; and exactly three for smooth sections with $\GG_m\GG_a$-action. 

%

  $$\begin{array}{cccccc}
	Aut^{o}(X_{K}) & \mathrm{Class} &&&&\\ &&&&& \\
	\mathbb{G}_{m} & F_{2}  & F_{3,n} & F_{3,n} & F_{4,n}   & F_{5,n} \\
	            & 2       & 3 & 3               & 3        & 4\\
	 \mathbb{G}_{a} & F_{5,n} & F_{7,n}  & F_{8,n} &  F_{9,n} &   \\
	             & 4      & 4       & 4  & 5  &  \\
	    \mathbb{G}_{m}^{ 2} & F_{3} &&&&  \\
	                & 4        &&&&   \\
\mathbb{G}_{m}\mathbb{G}_{a} & F_{5,n}& F_{7,n}& F_{8,n} &&\\ 
	                          &5& 5&5 &&\\
	\mathbb{G}_{a}^{2} & F_{9,n} & & & & \\
	&6&&&& \\
	         \mathbb{G}_{m}\mathbb{G}_{a}^{2}  & F_{7,n} & & & & \\
	                         &6  & &&&\\
	   SL_{2} & F_{4} &  F_{5,n}&&&\\
	                           &5	  & 6&&&\\
	                           GL_{2} &  F_{5} &&&&\\
	                           &7&&&&\\
	    SL_{2} \times SL_{2} & F_{7}&&&&\\
	                    & 9&&&&
\end{array}$$

\subsection{Special types} 

 Classes 5 and 7 in Table (\ref{flats}) correspond to special orbits of points in $\PP(\fc)$ that do not belong to the discriminant locus. Moreover they are the only possibilities to get an automorphism group of type $GL_2$ or $SL_2\times SL_2$. We deduce:

\begin{prop}
Up to projective isomorphism, there exists:
\begin{enumerate}
    \item a unique codimension four smooth section $X_K^{GL_2}$ of the spinor tenfold with automorphism group of type $A_1\times T_1\simeq GL_2$; 
     \item a unique codimension four smooth section $X_K^{SL_2\times SL_2}$ of the spinor tenfold with automorphism group of type
     $A_1\times A_1\simeq SL_2\times SL_2$. 
\end{enumerate}
\end{prop}

\subsubsection{The $GL_2$-variety} 
We describe $X_{K}^{GL_{2}}$ by the following model. Consider a two dimension space $U$. We denote by $S^{p}U(k)$ the irreducible module of $SL_{2} \times \mathbb{G}_{m}$, where the multiplicative group $\mathbb{G}_{m}$ acts with weight $k$. Choosing a nondegenerate $SL_{2}$-invariant quadratic form on $S^{2}U$, and consider the direct sum:
$$
V_{10}=(S^{2}U(2) \oplus S^{2}U(-2)) \oplus S^{2}U(0) \oplus \mathbb{C}(0).
$$
We can endow $V_{10}$ with invariant nondegenerate quadratic forms by using our
quadratic form on $S^2U$ to define a duality between $S^2U(2)$ and $S^2U(-2)$, 
and completing by an orthogonal direct sum. 
This induces an embedding of $\mathfrak{g}\mathfrak{l}_{2}$ inside $\fso_{10}$, hence an action on the half spin space $\Delta_-$. We also consider the four-dimensional module  $\mathbb{C}^{4}=U(1) \oplus U(-1)$.

\begin{prop}
The induced action of $\mathfrak{g}\mathfrak{l}_{2}$ on $\mathbb{C}^{4} \otimes \Delta_-$ has a two-dimensional space of invariants. Any generic such invariant defines the same variety $X_{K}^{GL_{2}}\subset \PP(\Delta_K)$, with 
$$
\Delta_{K}=U(3) \oplus S^{3}U(-1) \oplus S^{3}U(1) \oplus U(-3).
$$
\end{prop}

\subsubsection{The $SL_2\times SL_2$-variety}\label{sl2sl2}
In order to describe $X_K^{SL_2\times SL_2}$, we  use the following model. Consider two-dimensional spaces $A$ and $B$; then $S^2A$, $S^2B$ and $A\otimes B$ admit non degenerate quadratic forms (defined up to scalar). Consider the orthogonal direct sum 
$$V_{10}=S^2A\oplus S^2B\oplus A\otimes B.$$
This induces an embedding of $\fsl(A)\times\fsl(B)$ inside $\fso_{10}$, hence an action on the half-spin representation $\Delta_-$. We also consider its action on
$\CC^4=A\oplus B$. 

\begin{prop} The induced action of $\fsl(A)\times\fsl(B)$ on  $\CC^4\otimes
\Delta_{-}$ has a two-dimensional space of invariants. Any generic invariant 
defines the same variety $X_K^{SL_2\times SL_2}\subset\PP(\Delta_K)$, with
$$\Delta_{K}=A \otimes S^{2}B \oplus B \otimes S^{2}A.$$ 
\end{prop}

\begin{proof}
First choose an equivariant isomorphism  $V_{6}:=S^{2}A \oplus S^{2}B \simeq \wedge^{2}(A \otimes B)$,
and fix an invariant quadratic form on $V_6$ by choosing an isomorphism $\wedge^{4}(A \otimes B) \simeq \mathbb{C}$. Also, fix an invariant quadratic form on $V_{4}=A \otimes B$ by choosing
isomorphisms $\wedge^{2}A \cong \wedge^{2}B \cong \mathbb{C}$. The half-spin representations  of $\fso(V_{6})$ and $\fso(V_{4})$ are given as follows:
$$
\Delta^{6}_{+}=A \otimes B,\,\quad \Delta^{6}_{-}=(A \otimes B)^{\vee} \simeq A \otimes B, \,\quad
\Delta^{4}_{+}=A,\,\quad \Delta^{4}_{-}=B.
$$
By the restriction rules for spinor representations, 
$$\Delta_+\simeq \Delta^{6}_{+}\otimes \Delta^{4}_{+}\oplus \Delta^{6}_{-}\otimes \Delta^{4}_{-}, \qquad \Delta_-\simeq \Delta^{6}_{+}\otimes \Delta^{4}_{-}\oplus \Delta^{6}_{-}\otimes \Delta^{4}_{+},$$
we deduce that the  induced
$\fsl(A)\times\fsl(B)$-action on $\Delta_{\pm}$ is given by the following decomposition formula:
\begin{equation}\label{decomp}
\Delta_{\pm }\simeq 
A \otimes B \otimes A \oplus A \otimes B \otimes B\simeq
A\otimes (S^2B\oplus\wedge^2B)\oplus (S^2A\oplus\wedge^2A)\otimes B.
\end{equation}
One immediately deduces that the action of $\fsl(A)\times\fsl(B)$ on  $\CC^4\otimes
\Delta_{-}$ has a two-dimensional space of invariants, namely the sum of two copies of $\wedge^{2}A \otimes \wedge^{2}B$.  Choose a basis $a_{1},a_{2}$ of $A$ and a basis $b_{1}, b_{2}$ of $B$ such that $a_{1} \wedge a_{2}$ and $b_{1} \wedge b_{2}$ are sent to $1$ via the fixed homorphisms $\wedge^{2}A \simeq \wedge^{2}B \simeq \mathbb{C}$. Then any invariant is of the form:
$$
\delta_{\alpha,\beta}=\alpha\left(b_{1}\otimes (a_{1}b_{2}a_{2}-a_{2}b_{2}a_{1})-b_{2}\otimes (a_{1}b_{1}a_{2}-a_{2}b_{1}a_{1})\right) + $$
$$ \hspace*{5cm} +\beta \left(a_{1}\otimes (a_{2}b_{1}b_{2}-a_{2}b_{2}b_{1})+a_{2}\otimes (a_{1}b_{2}b_{1}-a_{1}b_{1}b_{2})\right).
$$
For $\alpha \beta \not=0$, the associated 4-space that $\delta_{\alpha,\beta}$ gives in $\Delta_{-}$ is
$$
K=\langle a_{1}b_{2}a_{2}-a_{2}b_{2}a_{1},  a_{2}b_{1}a_{1}-a_{1}b_{1}a_{2}, a_{2}b_{1}b_{2}-a_{2}b_{2}b_{1}, a_{1}b_{2}b_{1}-a_{2}b_{2}b_{1}\rangle,
$$
and does not depend on the choice of $\alpha$ and $\beta$. 

To prove that the corresponding section $X_K$ is smooth, we need to check that $K$ contains no pure spinor. The quadratic equations on $\Delta_{-}$ defining pure spnots are given by  $V_{10}$. One checks directly that those equations given by the subspace $\wedge^{2}(A \otimes B) \subseteq V_{10}$ have already no non-trivial zero on $K$.  So $X_K$ is a smooth section, and its automorphism group contains $SL(A) \times SL(B)$, or a finite quotient. 
\end{proof}


\medskip
The equations defining $X_{K}$ in $\PP(\Delta_K)$ are given by an equivariant morphism:
\begin{equation}\label{quad_K}
S^2(A \otimes S^{2}B \oplus B\otimes S^{2} A ) \rightarrow V_{10}=S^{2}A \oplus S^{2}B \oplus A \otimes B,
\end{equation}
that we describe  explicitly as follows. 
By homogeneity, the equations given by $A \otimes B$ correspond to the projection:
$$\begin{array}{rcl} 
A\otimes S^{2}B \otimes B \otimes S^{2}A & \rightarrow & A \otimes B, \\
a'\otimes b^2\otimes b'\otimes a^2 & \mapsto & (a\wedge a')(b\wedge b')\, a\otimes b.
\end{array}$$
 The projection to $S^{2}A \oplus S^{2}B$ is given by two equivariant morphisms:
$$
S^{2}(A \otimes S^{2}B) \cong S^{2}A \otimes S^{2}(S^{2}B) \oplus \wedge^{2}A \otimes \wedge^{2}(S^{2}B) \xrightarrow{(c_{1},c_{2})} S^{2}A\oplus S^{2}B,
$$
$$
S^{2}(B \otimes S^{2}A) \cong S^{2}B \otimes S^{2}(S^{2}A) \oplus \wedge^{2}B \otimes \wedge^{2}(S^{2}A) \xrightarrow{(c_{1}',c_{2}')} S^{2}B\oplus S^{2}A.
$$
where $c_{1},c_{2},c_{1}',c_{2}'$ are nonzero constants. Indeed we get two parameters for each projection, corresponding (for the first projection) 
to the discriminant map $D: S^2(S^2B)\lra \CC$ and the Jacobian map $Jac: \wedge^2(S^2B)\lra S^2B$, only defined up to non-zero scalars. 

\smallskip
Recall our basis $a_1,a_2$ of $A$ and $b_1,b_2$ of $B$. If we decompose $\theta$ in $\Delta_K$ as  
$$\theta=a_{1} \otimes \beta_{1} + a_{2} \otimes \beta_{2} +b_{1} \otimes \alpha_{1} + b_{2} \otimes \alpha_{2},$$ 
this means that the (quadratic) projection to $S^{2}A$ is equal to:
$$
c_{1}(D(\beta_1)a_{1}^{2} +D(\beta_2)a_{2}^{2} +2D(\beta_{1},\beta_{2}) a_{1}a_{2})+c_{2}'Jac(\alpha_{1}, \alpha_{2}),
$$
and similarly its projection to $S^{2}B$ is:
$$
c_{1}'(D(\alpha_{1})b_{1}^{2}+D(\alpha_{2}) b_{2}^{2}+2D (\alpha_{1},\alpha_{2}) b_{1}b_{2})+c_{2}Jac(\beta_{1},\beta_{2}).
$$
We normalize the discriminant and Jacobian in such a way that 
$$
Jac(a_{1}^{2}, a_{2}^{2})=2a_{1}a_{2}, \quad  Jac(a_{1}^{2}, a_{1}a_{2})=a_{1}^{2}, \quad Jac(a_{1}a_{2}, a_{2}^{2})=a_{2}^{2};
$$
$$
D(a_{1}^{2})=D(a_{2}^{2})=0,  \quad D(a_1^2, a_2^2)=1, \qquad D(a_1a_2)=-\frac{1}{2},
$$
and similarly for $B$. 

\medskip
Now we can classify orbits of $SL_2\times SL_2$ in $X_K^{SL_2\times SL_2}$
as follows. Consider $\theta=\theta_1+\theta_2$, with $\theta_1\in S^2A\otimes B$ 
and $\theta_2\in A\otimes S^2B$. Suppose $\theta_1$ has rank two as a morphism from 
$B^\vee$ to $S^2A$, so that it defines a pencil $P_1$ of binary forms. 
Call this pencil {\it degenerate} if it is tangent to the conic $C$ of double points in $\PP(S^2A)$, {\it regular} otherwise. This defines two orbits.

\begin{lemma}
These two orbits admit the following representatives:
$$\begin{array}{rcl}
\theta & = & a_1\otimes b_1^{2}+a_2\otimes b_2^{2}+b_1\otimes a_2^2+b_2\otimes a_1^2, \\
\theta' &=&  a_1\otimes b_1^{2}-2a_2\otimes b_1b_2\otimes a_1-2b_{1}\otimes a_{1}a_{2}-b_2\otimes a_2^2.
\end{array}$$
\end{lemma}

\proof 
If $P_1$ is regular, it meets the conic $C$ in two points $[a_1^2]$ and $[a_2^2]$ for some $a_1, a_2$ defining a basis of $A$. Then there is a basis $b_1, b_2$ of $B$ such that 
$$\theta_1 = a_1^2\otimes b_1+a_2^2\otimes b_2, \qquad 
\theta_2 = a_1\otimes q_1+a_2\otimes q_2$$
for some $q_1, q_2\in S^2B$. Let us discuss the conditions (\ref{quad_K}) for $[\theta]$ to belong to $X_K^{SL_2\times SL_2}$. Mixed quadratic equations (coming from $A\otimes B$) imply that there exists 
scalars $z_1, z_2$ such that $q_1=z_1b_2^2$, $q_2=z_2b_1^2$. Unmixed equations 
(coming from $S^2A\oplus S^2B$) then reduce to the identities 
$c_1-c'_2z_1z_2=c_2+c'_1z_1z_2=0.$

If $P_1$ is degenerate, it meets the conic $C$ at a unique point $[a_1^2]$ and 
contains another point of the form $[a_1a_2]$, so as above there is a basis $b_1, b_2$ of $B$ such that 
$$\theta_1 = a_1^2\otimes b_1+2a_1a_2\otimes b_2, \qquad 
\theta_2 = a_1\otimes q_1+a_2\otimes q_2$$
for some $q_1, q_2\in S^2B$. The mixed equations for $[\theta']$  allow to reduce to
$\theta_2=a_1\otimes (ub_2^2+vb_1b_2)+wa_2\otimes b_2^2$, and then the unmixed equations yield the conditions $4c_2-c'_1v^2=c_1-2c'_2vw=v+2w=0$. 
For these equations to admit non-trivial solutions (which have to exist), we need
that $c_{1}c_{1}'+c_{2}c_{2}' =0$. Then we can normalize $\theta$ and $\theta'$ 
as claimed. \qed

\begin{prop} \label{openorbit}
$SL_2\times SL_2$ acts on $X_K^{SL_2\times SL_2}$ with an open orbit, a codimension one orbit, two codimension two orbits, a codimension three orbit and 
two closed orbits isomorphic to $\PP^1\times v_2(\PP^1)$.
\end{prop}

\proof 
It is a straightforward calculation that 
$$
Stab_{SL_{2} \times SL_{2}}([\theta])=
\biggl\{
\biggl(\begin{pmatrix}
	\xi^{3} & 0 \\
	0 & \xi^{-3}
\end{pmatrix}
	 , \begin{pmatrix}
	 	\xi & 0 \\
	 	0 & \xi^{-1}
	 \end{pmatrix} \biggl), \quad \xi^{10}=1 \biggl\}.
$$ 

So there is an open orbit in $X_K^{SL_2\times SL_2}$, and this open orbit is affine since the generic stabilizer is finite. So its complement must have codimension one. This must be the orbit of $\theta'$ since the other cases yield orbits of smaller dimension. 

Indeed, if we are not in the orbits of $\theta$ and $\theta'$, with the previous notations both 
$\theta_1=p\otimes b$ and $\theta_2=a\otimes q$ must have rank at most one, where $p\in S^2A$ and $q\in S^2B$. For $a, b\ne 0$, the quadratic equations yield the conditions 
that $p$ and $q$ are squares, with either $p$ a multiple of $b^2$ or $q$ a multiple of $a^2$. This gives two orbits of codimension two, and an orbit of codimension three when both conditions are fulfilled. Finally for $a=0$ or $b=0$ we get the two closed orbits. \qed 

\begin{coro} 
$X_K^{SL_2\times SL_2}$ compactifies $SL_2\times SL_2/\mu_{10}$.
\end{coro}

\proof This follows from the computation of the generic stabilizer in the previous proof.\qed 

\medskip
Note that the action of $SL_2\times SL_2$ on $\Delta_K$ is multiplicity free; so a fortiori the action of a maximal torus on $X_K^{SL_2\times SL_2}$ has finitely many fixed points, and we can deduce from the Bialynicki-Birula decomposition that is contains a copy of the affine space $\AA^6$. In fact, many other sections $X_K$ admit  $\mathbb{C}^{*}$-actions, and we get a large family of compactifications that we briefly  discuss, together with  their boundary divisors.

The generic section with $\CC^*$-action comes from Class 2, which is represented by a generic combination of $p_1, p_2, p_3$. The $\CC^*$-action is then given by a one-dimensional subtorus of the canonical maximal torus of $Spin_{10}$, defined by $t_4=t_5=1$ and $t_2=t_3=t_1^{-1}$. Under the action of this subtorus, he spin representation  splits into isotypic components of dimensions $2,6,6,2$; the two extremal ones give two projective lines contained in the 
spinor tenfold; the two others give two copies of $\PP^1\times\PP^2$. We get the fixed point set of the $\CC^*$-action on $X_K$ by cutting with 
$\PP(K^\perp)$: the two projective lines remain, and we additionally get two codimension two sections of $\PP^1\times\PP^2$, which are normal rational cubics, hence also copies of $\PP^1$. As a corollary, $X_K$ contains a $\PP^{1}$-family of of copies of $\AA^{6}$, as follows.

\begin{prop}
Let $L\simeq\PP^1$ be an extremal fixed component of the $\mathbb{C}^{*}$-action on $X_{K}$. For any point $x \in L$, the sixfold $X_{K}$ is a compactification of $\AA^{6}$, obtained by restricting to $L \backslash \{x\}$ the Bialynicki-Birula cell associated to $L$.  The boundary divisor is singular along a projective plane.
\end{prop}

\begin{proof}
Since $L$ is an extremal component, its associated Bialynicki-Birula cell is isomorphic to the total space of the normal bundle $N_{L/X_{K}}$. Thus the restriction to $L \ \backslash \{x\}$ is a copy $\AA^{6}$. The last claim follows from a direct calculation applying Lemma \ref{singular-locus}.
\end{proof}

For the two special types $X_{GL_{2}}$ and $X_{SL_{2}\times SL_{2}}$, a direct application of Lemma \ref{singular-locus} yields:

\begin{prop}
Let $X_{K}=X_{K}^{GL_{2}}$ or $X_{K}^{SL_{2} \times SL_{2}}$. Then the action of a maximal torus has finitely 
many fixed points, and $X_K$ is therefore a compactification of $\AA^6$. In both cases, the boundary divisor is again singular along a projective plane.
\end{prop}
\section{Appendix A. Cohomology}\label{cohomology}

In this Appendix we compute the cohomology of codimension four smooth sections of the spinor tenfold, including the quantum one. Lefschetz's hyperplane theorem allows to check  that a smooth codimension four section of the spinor tenfold
has the same Betti numbers as a six-dimensional quadric, and the cohomology is a free $\ZZ$-module in both cases. But the algebra structures are different; multiplication  by the hyperplane class
produces different graphs. Since they are all diffeomorphic we can use the special section from Proposition \ref{sl2sl2}. The point is that this sixfold admits a torus action with finitely many fixed points, allowing to use the powerful methods of equivariant cohomology \cite{gkm}. 

\subsection{Equivariant cohomology} 
Since the action of $SL_2\times SL_2$ on $\Delta_K=A \otimes S^{2}B \oplus B \otimes S^{2}A$ is multiplicity free, a maximal torus $T$ acts on  $X_K^{SL_2\times SL_2}$ with finitely
many fixed points. We can then use classical localization techniques to compute its
$T$-equivariant cohomology. The computations are straightforward, and we merely state 
the results. 

\smallskip
Let us denote by $a_1, a_2$ and $b_1, b_2$ some basis of $A$ and $B$ respectively, 
diagonalizing the action of $T$. The eight fixed points of $T$ in 
$X_K^{SL_2\times SL_2}$ are the
$[a_i^2\otimes b_j]$ and $[a_i\otimes b_j^2]$, for $1\le i,j\le 2$. Recall the GKM graph
has these fixed points for vertices, and there is an arrow between two fixed points when there are joined by a $T$-invariant curve in the variety. Inside $X_K^{SL_2\times SL_2}$ there are eight $T$-invariant lines and four $T$-invariant cubic curves, and we get the  GKM graph below, where we abbreviate $[a_i^2\otimes b_j]$ by $a_i^2b_j$ for simplicity. 

Each arrow is decorated by a character of the torus encoding its action on the corresponding invariant curve. Each $T$-equivariant cohomology class is then given by 
specifying, for each vertex of the GKM graph, a polynomial on $\ft=Lie(T)$. 

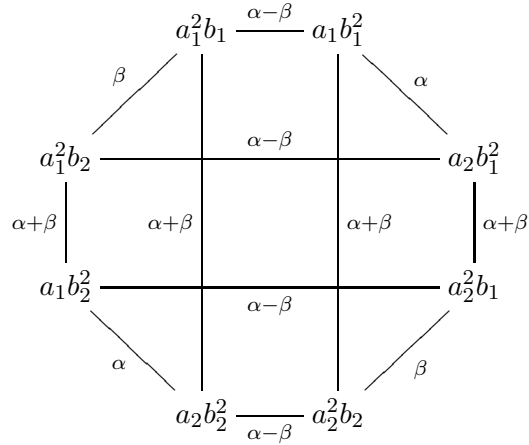
$$\xymatrix@R30pt@C25pt{
 & a_1^2b_1 \ar@{-}[r]^{\alpha-\beta}\ar@{-}[ddd]_{\alpha+\beta} & a_1b_1^2\ar@{-}[ddd]^{\alpha+\beta}\ar@{-}[rd]^{\alpha} & \\
  a_1^2b_2\ar@{-}[d]_{\alpha+\beta} \ar@{-}[ru]^{\beta} \ar@{-}[rrr]^{\alpha-\beta}&&& a_2b_1^2\ar@{-}[d]^{\alpha+\beta} \\
  a_1b_2^2\ar@{-}[rd]_{\alpha} \ar@{-}[rrr]_{\alpha-\beta}&&& a_2^2b_1 \\
 & a_2b_2^2 \ar@{-}[r]_{\alpha-\beta} & a_2^2b_2\ar@{-}[ru]_{\beta} &
}$$

\medskip
To get something more geometric, we can choose a generic one-dimensional subtorus
$S$ of $T$. It has the same fixed points and we therefore get a decomposition of  $X_K^{SL_2\times SL_2}$ into Bialynicki-Birula cells. The cohomology classes of their closures  provide a basis for the $T$-equivariant cohomology. 

Concretely, we first need to compute the weights of the $\ft$-action on the tangent bundle 
at each fixed point. The affine tangent space to $X_K^{SL_2\times SL_2}$
at $[a_1\otimes b_1^2]$ is $\langle a_1\otimes b_1^2, a_1\otimes b_1b_2, a_2\otimes b_1^2, a_1^2\otimes b_1, a_1a_2\otimes b_1, a_2^2\otimes b_1, a_1^2\otimes b_2\rangle $, so the weights of this action at this point are 
$$2\alpha, 2\beta, \beta-\alpha, \alpha+\beta, 3\alpha+\beta, 3\beta-\alpha.$$
(And we can deduce the corresponding weights for the other fixed points simply by permutations.) All these weights will be positive if we choose $S$ on which $\beta>\hspace{-1mm}>\alpha>\hspace{-1mm}>0$, meaning that the point $[a_1\otimes b_1^2]$ is attractive and the corresponding Bilanicki-Birula cell is open in  $X_K^{SL_2\times SL_2}$. 
The dimensions of the cells corresponding to each fixed point are indicated in the following graph:\small
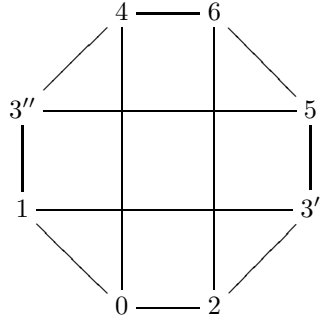
$$\xymatrix{
 & 4 \ar@{-}[r]\ar@{-}[ddd] & 6\ar@{-}[ddd]\ar@{-}[rd] & \\
  3''\ar@{-}[d] \ar@{-}[ru] \ar@{-}[rrr]&&& 5\ar@{-}[d] \\
 1 \ar@{-}[rd] \ar@{-}[rrr]&&& 3' \\
 & 0 \ar@{-}[r]  & 2\ar@{-}[ru] &
}$$\normalsize

\medskip
Poincar\'e duality is realized by the central symmetry of this picture.
Let $C_i$ denote the closure of the codimension $i$ cell (in codimension three 
there are two cells, whose closures we denote by $C_3'$ and $C_3''$ in accordance
with the previous picture).  In low dimension we have 
the following simple description:

\begin{prop}
$C_5$ is a line and $C_4$ is a plane. Moreover $C'_3=Sing(C_1)$ is a rank three quadric, singular along $C_5$. 
\end{prop}

We will denote by $\Sigma_i$ the $T$-equivariant cohomology class of  $C_i$ of the codimension $i$ cell; in codimension three there are two cells, we denote the corresponding cohomology classes by $\Sigma_3'$ and $\Sigma_3''$. By localizing to fixed points, we can represent each $\Sigma_i$ by a collection of homogeneous polynomials of degree $i$ in $\alpha, \beta$, one for each vertex of the GKM graph. These collections are constrained
by some combinatorial and geometric rules that allow to determine them easily. 
For example the fundamental class $\Sigma_0$ is obtained by putting $1$ on each vertex, while the hyperplane class $\Sigma_1$ is represented by 
$$
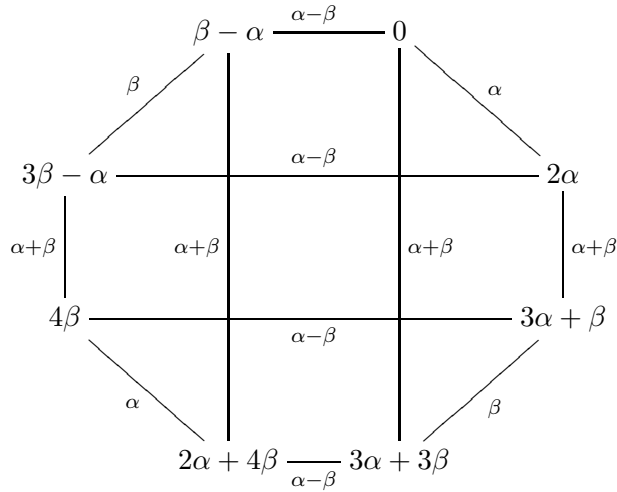
\xymatrix@R38pt@C20pt{
 & \beta-\alpha \ar@{-}[r]^{\alpha-\beta}\ar@{-}[ddd]_{\alpha+\beta} & 0 \ar@{-}[ddd]^{\alpha+\beta}\ar@{-}[rd]^{\alpha} & \\
  3\beta-\alpha \ar@{-}[d]_{\alpha+\beta} \ar@{-}[ru]^{\beta} \ar@{-}[rrr]^{\alpha-\beta}&&& 2\alpha\ar@{-}[d]^{\alpha+\beta} \\
  4\beta\ar@{-}[rd]_{\alpha} \ar@{-}[rrr]_{\alpha-\beta}&&& 3\alpha+\beta \\
 & 2\alpha+4\beta \ar@{-}[r]_{\alpha-\beta} & 3\alpha+3\beta\ar@{-}[ru]_{\beta} &
}$$


\medskip
Observe that there is a zero on the vertex corresponding to the fixed point
not belonging to $C_1$. Moreover the difference of two polynomials corresponding to vertices that are joined by an arrow is divisible by the label of this arrow. 
These two properties are valid for each  $\Sigma_i$. An additional property is that if $\Sigma_i$ corresponds to a fixed point $p_i$, its polynomial 
representative at $p_i$ is just the product of the weights of the normal bundle to the cell at this point, that is, the weights of the tangent bundle that become
negative when restricted to $S$. 

\smallskip
It is then straightforward to compute each $\Sigma_i$. Once this is done, the multiplication of cohomology classes is just polynomial multiplication on the vertices of the GKM graph, and it is a straightforward exercise to get formulas such as
$$\begin{array}{rcl}
\Sigma_1^2 & = & 2\alpha\Sigma_1+\Sigma_2, \\
\Sigma_2\Sigma_1 & = & (\beta-\alpha)\Sigma_2+3\Sigma_3'+2\Sigma_3'',\\
\Sigma_3'\Sigma_1 & = & (3\beta-\alpha)\Sigma_3'+2\Sigma_4, \\
\Sigma_3''\Sigma_1 & = & (3\alpha+\beta)\Sigma_3'+3\Sigma_4, \\
\Sigma_4\Sigma_1 & = & 3(\alpha+\beta)\Sigma_4+\Sigma_5,\\
\Sigma_5\Sigma_1 & = & 4\beta\Sigma_5+\Sigma_6, \\
\Sigma_2^2 & = & (\alpha-\beta)(3\alpha-\beta)\Sigma_2+12(\beta-\alpha)\Sigma_3'+4\beta\Sigma_3''+12\Sigma_4,\\
\Sigma_3'\Sigma_2 & = & 3(\alpha-\beta)(\alpha-3\beta)\Sigma_3'+12\beta\Sigma_4+2\Sigma_5, \\
\Sigma_3''\Sigma_2 & = & (\alpha+\beta)(3\alpha+\beta)\Sigma_3''+12(\alpha+\beta)\Sigma_4+3\Sigma_5,\\
\Sigma_3'\Sigma_3'' & = & 3(\alpha+\beta)(\alpha+3\beta)\Sigma_4+(\alpha+3\beta)\Sigma_5+\Sigma_6.
\end{array}$$

\subsection{Ordinary cohomology, Chow ring}
From the $T$-equivariant cohomology one readily deduces the usual cohomology (or Chow ring) of $X_K^{SL_2\times SL_2}$. Indeed this is a free $\ZZ$-module, with a basis given by the classes $\sigma_i$ of the Byalinicki-Birula cell closures.
Their products are obtained by specializing the previous formula to $\alpha=\beta=0$ and replacing $\Sigma_i$ by $\sigma_i$. We get the complete multiplication table:

\medskip
$$\begin{array}{|r|cccccc|}
\hline 
 & \sigma_1  & \sigma_2  & \sigma_3'  & \sigma_3''  & \sigma_4  & \sigma_5  \\ 
 \hline
 \sigma_1 & \sigma_2  & 3\sigma_3'+2 \sigma_3''  & 2\sigma_4  & 3\sigma_4 & 
 \sigma_5 & \sigma_6 \\ 
 \sigma_2  & 3\sigma_3'+2 \sigma_3'' & 12\sigma_4 & 2\sigma_5 & 3\sigma_5 & \sigma_6 & 0 \\
 \sigma_3'  &2\sigma_4 &2\sigma_5& 0 & \sigma_6 & 0&0\\
 \sigma_3''  & 3\sigma_4&3\sigma_5& \sigma_6 &0&0&0\\
 \sigma_4  & \sigma_5& \sigma_6 &0&0&0&0 \\
 \sigma_5& \sigma_6 &0&0&0&0&0 \\
 \hline
 \end{array}$$
\medskip

Over $\QQ$ the Chow ring is generated e.g. by $\sigma_1$ and $\sigma'_3$, subject to the relations
\begin{equation}\label{relChow}
\sigma_1^4=6\sigma_1\sigma'_3, \qquad (\sigma'_3)^2= 0. 
\end{equation}

\subsection{Quantum cohomology}
Once we know the usual cohomology, since the index is high it is not difficult to upgrade it to quantum cohomology. A first useful information can be deduced from  the fact that lines in $X_K$ through a generic point are parametrized by a 
quintic del Pezzo; this implies in quantum cohomology that $\sigma_1^4=12\sigma_4+5q$.
Since $\sigma_1^3=3\sigma'_3+2\sigma''_3$ we deduce  that 
$$\sigma'_3\sigma_1=2\sigma_4+q, \qquad \sigma''_3\sigma_1=3\sigma_4+q,$$
and this directly gives that the quantum version of the degree four relation is 
$\sigma_1^4=6\sigma_1\sigma'_3-q$. To get the other one we can first use the symmetries of the Gromov-Witten invariants, and the associativity of the quantum product. Symmetry ensures that there is an integer $N$ such that $\sigma_4\sigma_1=\sigma_5+qN\sigma_1$ and 
$\sigma_5\sigma_1=\sigma_6+qN\sigma_1^2$. Let $A, B, C$ be defined by the identities
$$(\sigma'_3)^2=qA\sigma_1^2, \qquad \sigma'_3\sigma''_3=\sigma_6+qB\sigma_1^2, \qquad 
(\sigma''_3)^2=qC\sigma_1^2.$$
Using associativity we deduce that there exists $M$ such that $N=3M+1$, $A=2M+1$, $B=6M+1$ and $C=9M+2$. We claim that $M=0$, $N=1$. In other words:

\begin{lemma} $I_1(\sigma_1,\sigma_4,\sigma_5)=1$. \end{lemma}

\proof $\sigma_5$ is Poincaré dual to the class of a line, and $\sigma_4$ to the class of a plane. So we have to check that given a general line $\ell$ and a general plane $P$ in $X_K$, there is a unique line in $X_K$ connecting $\ell$ to $P$. 
Since $X_K$ is a linear section of $X$, we need to prove the same statement in $X$ itself. 

Now, the  line $\ell$ in $X$ is defined as the family of $[U_5]\in X$ such that $U_5$ contains a fixed isotropic subspace $A_3$ of $V_{10}$. Similarly, 
the plane $P$ is defined  by an isotropic flag $B_2\subset B_5$, where $B_5$ belongs to the other family $X^\vee$ of maximal isotropic subspaces of  $V_{10}$: a point 
 $[V_5]\in X$ belongs to $P$ when $B_2\subset V_5$ and $B_5\cap V_5$ is four-dimensional. We need to prove that there is a unique such pair $(U_5,V_5)$ such that the line joining $[U_5]$ and $[V_5]$ is contained in $X$, which means that 
 $U_5\cap V_5$ is three-dimensional. 
 
 Suppose this is the case and let $W_3=U_5\cap V_5$, $V_4=B_5\cap V_5$. Recall that $V_4$ determines $V_5$. 
 Since $U_5$ meets $B_5$ in codimension one, $W_2=V_5\cap U_5\cap B_5$ has dimension (at least) two. Since $A_3\subset U_5$, taking orthogonals we get $U_5\subset A_3^\perp$, hence $W_2\subset B_5\cap A_3^\perp$ which has dimension two in general; so $W_2=B_5\cap A_3^\perp$ is uniquely determined. Since $W_2\subset U_5$ by definition, 
 and $W_2\cap A_3\subset B_5\cap A_3=0$, we recover $U_5=W_2+A_3$. Also $W_2\subset V_4$ by definition, and since $W_2\cap B_2\subset A_3^\perp\cap B_2=0$, we also recover $V_4=W_2+B_2$, hence also $V_5$. This proves the claim. \qed 
 
 \medskip We finally deduce the following statement.

\begin{prop} The small quantum cohomology of $X_K$ is 
$$QH^*(X_K)=\QQ[\sigma_1,\sigma'_3,q]/\langle \sigma_1^4-6\sigma_1\sigma'_3+q, (\sigma'_3)^2-q\sigma_1^2\rangle.$$
In particular it is generically semisimple. 
\end{prop} 

The matrix of quantum multiplication by $\sigma_1$ at $q=1$ is the following:

$$\hat{\sigma}_1=\begin{pmatrix}
0 & 0& 0& 1& 1& 0& 0& 0\\
1 & 0& 0& 0& 0& 1& 0& 0\\
0 & 1& 0& 0& 0& 0& 1& 0\\
0 & 0& 3& 0& 0& 0& 0& 1\\
0 & 0& 2& 0& 0& 0& 0& 1\\
0 & 0& 0& 2& 3& 0& 0& 0\\
0 & 0& 0& 0& 0& 1& 0& 0\\
0 & 0& 0& 0& 0& 0& 1& 0
 \end{pmatrix}.$$
 
The  eigenvalues of this matrix  all have  multiplicity one, being $(\sqrt{2}\pm 1)\times \zeta$, for $\zeta^4=1$. Since $X_4$ has index four, this is in perfect agreement with Conjecture $\mathcal{O}$ \cite[Conjecture 3.1.2]{ggi}, which predicts that the eigenvalues of maximal modulus all have multiplicity one, and can be deduced one from another by multiplication by a root of unity or order equal to the index. 

\section{Appendix B. Elements of Jordan-Vinberg theory}\label{jordan}

\subsection{Graded Lie algebras} 
Consider a semisimple Lie algebra $\fg=\oplus_{k\in I}\fg_k$, graded over $I=\ZZ$ or $I=\ZZ_m$ for some integer $m> 1$. In particular $\fg_0$ is a subalgebra and each $\fg_k$ is a $\fg_0$-module. Let $G_0$ be the connected subgroup of $G=Aut(\fg)$
with Lie algebra $\fg_0$; its action on $\fg$ preserves each $\fg_k$. 

For $I=\ZZ$ and $k\ne 0$, $\fg_k$ is contained in the nilpotent cone of $\fg$, and one can deduce from the finiteness of nilpotent orbits in $\fg$ that $G_0$ acts on $\fg_k$ with finitely many orbits. Typical $\ZZ$-gradings are obtained by choosing a vertex of the Dynkin diagram of $\fg$. For example, starting from $\fg=\fe_8$, the picture 
$$\dynkin[edge length=6mm] E{ooooo*oo} \qquad  \leadsto \qquad 
\dynkin[edge length=6mm] D{oooo*} \times \dynkin[edge length=6mm] A{*o}$$
indicates that $\fg_1=\Delta_+\otimes \CC^3$. We deduce that 
$Spin_{10}\times GL_3$ has finitely many orbits on this representation, hence that $Spin_{10}$ has finitely many orbits on
the Grassmannians $G(3,\Delta_+)$ or $G(3,\Delta_-)$. 

Typical $\ZZ_m$-gradings are obtained by playing the same game with affine Dynkin diagrams, which are labelled by the coefficients of the highest root; if we choose a vertex, $m$ is the corresponding label. For example, starting from the affine Dynkin diagram $E_8^{(1)}$, the picture 
$$\dynkin[edge length=6mm, extended] E{ooooo*oo} \qquad  \leadsto \qquad 
\dynkin[edge length=6mm] D{oooo*} \times \dynkin[edge length=6mm] A{*oo}$$
indicates that $\fg_1=\Delta_+\otimes \CC^4$. In this case what we can deduce is that orbits of $Spin_{10}\times GL_4$ are not finitely many, but are well-behaved in the sense that they admit a classification in the style of Jordan theory.

\subsection{Jordan classes and sheets} 
Given a  $\mathbb{Z}_{m}$-graded semisimple Lie algebra $\mathfrak{g}$, we consider the action of $G_{0}$ on $V=\mathfrak{g}_{1}$. 
An important property is that each $x\in V$ admits a unique 
Jordan decomposition  $x=x_{s}+x_{n}$, where $x_s, x_n\in V$ and 
as elements of $\fg$, $x_s$ is semisimple, $x_n$ is nilpotent, 
and $[x_s,x_n]=0$. This allows to stratify $V$ as we do in ordinary Jordan theory \cite{JorVin1, JorVin2}. 
Two elements $x=x_{s}+x_{n}$ and $y=y_{s}+y_{n}$ in $V$ are said to be $G_{0}$-Jordan equivalent if there exists $g \in G_{0}$ such that 
$$
\mathfrak{g}^{y_{s}}=g \cdot \mathfrak{g}^{x_{s}}, \qquad y_{n}=g \cdot x_{n},
$$
where $\mathfrak{g}^{x}$ is  the centralizer of $x$ in $\mathfrak{g}$. Equivalently, $x$ and $y$ are $G_{0}$-Jordan equivalent if there exists $g \in G_{0}$ such that \cite[Prop. 3.9 and Prop. 3.17]{JorVin1}
$$\mathfrak{g}^{x}=g \cdot \mathfrak{g}^{y}.$$ This  implies that the identity components of the  stabilizer subgroups of $x$ and $y$ in $G_{0}$ are conjugate. Geometrically it yields a stratification in terms of automorphism groups, as in Section 6.3.\par 
The equivalence class $J_{G_{0}}(x)$ of $x \in V$ is called the $G_{0}$-Jordan class of $x$ in $V$. The  $G_{0}$-Jordan classes of semisimple elements in $V$ are defined by the reflection hyperplanes in the Cartan subspace, they were called flats in Section 6. Given a Cartan subspace $\mathfrak{c} \subset V$, denote by $\Sigma \subset \mathfrak{c}^{\vee}$ the set of restricted roots of $\mathfrak{c}$. Each restricted root $\alpha$ is orthogonal  to a reflection hyperplane in $H_{\alpha}\subset\fc$. For $x \in \mathfrak{c}$, denote by $\Sigma(x)$ the set of roots that are nonzero on $x$. Then for any two elements $x, y \in \fc$ the following are equivalent:\par 
(i) $x,y \in \fc$ are $G_{0}$-Jordan equivalent;\par 
(ii) there exists $w \in W_{\fc}$ such that $w.\Sigma(x)=\Sigma(y)$; \par 
(iii) their exists $w \in W_{\fc}$ such that 
$Stab_{W_{\fc}}(x)=w \cdot Stab_{W_{\fc}}(y) \cdot w^{-1}$.\par 
More generally one can characterize the Jordan class of $x$ by the roots in $\Sigma(x_{s})$ (after conjugating $x_s$ to some element of $\fc$), and its nilpotent part $x_{n}$ as follows \cite[Prop. 14, Prop. 30 and Prop. 37]{JorVin2}.

\begin{prop}
 Let $x \in V$, and assume its semisimple part $x_{s}$ lies in $\fc$. Consider 
$$\mathfrak{z}(\mathfrak{g}^{x_{s}})_{1}^{reg}
=\{y \in \mathfrak{c}: \Sigma(x_{s})=\Sigma(y)\}.
$$
(i) The $G_{0}$-Jordan class $J_{G_{0}}(x)$ of $x$ is an irreducible smooth locally closed subset of $V$, namely
$$
J_{G_{0}}(x)=G_{0} \cdot (\mathfrak{z}(\mathfrak{g}^{x_{s}})_{1}^{reg}+x_{n}).
$$ 
(ii) Denote by $\Gamma=N_{W}(\mathfrak{z}(\mathfrak{g}^{x_{s}})_{1}^{reg})$ the stabilizer of $\mathfrak{z}(\mathfrak{g}^{x_{s}})_{1}^{reg}$ in $W$; it acts naturally on the set of nilpotent orbits of the action of $G_{0}^{x_{s}}$ on $\mathfrak{g}^{x_{s}}_{1}$. Denote by $\Gamma_{n}$ the stabilizer in $\Gamma$ of the nilpotent orbit of $x_{n}$:
Then the assignment $\phi:  y_{s} \mapsto \mathcal{O}^{G_{0}}_{y_{s}+x_{n}}$  from $\mathfrak{z}(\mathfrak{g}^{x_{s}})_{1}^{reg}$ to the orbit set $J_{G_{0}}(x)/G_{0}$ induces a homeomorphism:
$$
\overline{\phi}: \mathfrak{z}(\mathfrak{g}^{x_{s}})_{1}^{reg}/\Gamma_{n} \rightarrow J_{G_{0}}(x)/G_{0}.
$$
\end{prop}

Next we recall the relation between Jordan classes and sheets. For any $d \in \mathbb{N}$, we define  $$V_{(d)}=\{x \in V: dim(\mathcal{O}_{x}^{G_{0}})=d\}.$$ 
Each $V_{(d)}$ is locally closed; its irreducible components are called {\it sheets}. For any  $S \subset V$, we denote by $S^{\bullet}$ the open subset of points in $S$ whose orbits are of maximal dimension.
\begin{prop}\cite[Prop 3.19]{JorVin1}
For any sheet $S$ in $V$, there exists a unique $G_{0}$-Jordan class $J$ in $V$, such that $S=\overline{J}^{\bullet}$. Moreover we have $\overline{J}=\overline{S}$.
\end{prop}

We are also interested in the stratification of $V$ in terms of the stabilizers in $G_{0}$.
For any reductive subgroup $H \subset G_{0}$, denote
$$
V_{H}=G_{0}.\{s \in V: Stab_{G_{0}}(s)=H\},
$$ 
and denote by $V_{H}^{o}$ by the elements whose identity component of its stabilizer subgroup is conjugate to $H^{o}$.
It is not known in general whether this stratification on $V$ is locally closed, or smooth. However, restricted to smaller loci we have satisfactory answers as follows. Let $V^{ssp}=G_{0}.\mathfrak{c}$ be the subset of semisimple elements in $V$. Recall that  $\mathbb{P}(V^{ssp})$ contains the 
GIT-polystable locus of the action of $G_{0}$ on $\mathbb{P}(V)$, and that $Stab_{G_{0}}^{o}([v])=Stab_{G_{0}}^{o}(v)$ for any $v \in V^{ssp}$. 
We denote $V_{H}^{ssp}=V_{H} \cap V^{ssp}$ to be the intersections with $V_{H}$.

\begin{prop}Fix a connected reductive subgroup $H$ of $G$, then each $\PP(V_{H}^{ssp})$ is a locally closed subset of $\PP V^{ssp}$, and  
$$
\mathbb{P}(V^{ssp}) \cap \overline{\mathbb{P}(V_{H}^{ssp}})=\underset{H \subset H' }{\bigcup} \mathbb{P}(V_{H'}^{ssp}).
$$
\end{prop}
We conclude that on $V^{ssp}$, the following relations hold between the stratifications:
$$
\{sheets\} \rightarrow \{G_{0}-\text{Jordan classes}\} \rightarrow \{V_{H}: H \ \text{reductive} \} \rightarrow \{(V_{H})^{o}: H \, \text{reductive}\} \rightarrow \{V_{(d)}: d\in \mathbb{N}\},
$$
meaning that the stratifications become coarser from left to right. They  are all locally closed, and the second one is smooth.
%

\section{Appendix C. Combinatorics}\label{combinatorics}

In this Appendix we gather a few data related to the combinatorics of the little Weyl group $W_\fc$. First, the $60$ reflection hyperplanes of this complex reflection group are given by the following equations:

\small
\medskip
$$\begin{array}{llllll}
(1) & a_1, & (21) & a_1+ia_2+(i+1)a_4, & (41) & a_1-a_2-ia_3+ia_4, \\
(2) & a_2, & (22) & a_1+ia_2+(i-1)a_4, &(42) & a_1-a_2+ia_3-ia_4, \\
(3) & a_3, & (23) & a_1+ia_2+(-i+1)a_4,\qquad  & (43) & a_1+a_2-ia_3-ia_4, \\
(4) & a_4, &  (24) & a_1+ia_2+(-i-1)a_4, &  (44) & a_1-a_2-ia_3-ia_4,\\
(5) & a_1+a_2, & (25) & a_1-ia_2+(i+1)a_4, & (45) & (i+1)a_1+ia_3-a_4, \\
(6) & a_1-a_2, &(26) & a_1-ia_2+(i-1)a_4, &(46) & (i+1)a_1+ia_3+a_4, \\
(7) & a_3+a_4, & (27) & a_1-ia_2+(-i+1)a_4, &(47) & (i-1)a_1-ia_3-a_4, \\
(8) & a_3-a_4, & (28) & a_1-ia_2+(-i-1)a_4, & (48) & (i-1)a_1-ia_3+a_4,\\
(9) & a_1+ia_2, & (29) & a_1+a_2+a_3+a_4, & (49) & (i-1)a_1+ia_3-a_4,\\
(10) & a_1-ia_2, &(30) & a_1-a_2+a_3+a_4, & (50) & (i-1)a_1+ia_3+a_4,\\
(11) & a_3+ia_4, & (31) & a_1+a_2-a_3+a_4,  & (51) & (i+1)a_1-ia_3-a_4, \\
(12) & a_3-a_4, &  (32) & a_1+a_2+a_3-a_4, &  (52) & (i+1)a_1-ia_3+a_4,\\
(13) & a_1+ia_2+(i+1)a_3, & (33) & a_1-a_2-a_3+a_4, & (53) & (i+1)a_2+ia_3-a_4,  \\
(14) & a_1+ia_2+(i-1)a_3,\qquad  & (34) & a_1-a_2+a_3-a_4,\qquad & (54) & (i+1)a_2+ia_3+a_4,\\
(15) & a_1+ia_2+(-i+1)a_3, & (35) & a_1+a_2-a_3-a_4, &(55) & (i-1)a_2-ia_3-a_4,\\
(16) & a_1+ia_2+(-i-1)a_3, &  (36) & a_1-a_2-a_3-a_4, & (56) & (i-1)a_2-ia_3+a_4,\\
(17) & a_1-ia_2+(i+1)a_3, &  (37) & a_1+a_2+ia_3+ia_4, &(57) & (i-1)a_2+ia_3-a_4,\\
(18) & a_1-ia_2+(i-1)a_3, & (38) & a_1-a_2+ia_3+ia_4, &(58) & (i-1)a_2+ia_3+a_4, \\
(19) & a_1-ia_2+(-i+1)a_3, &  (39) & a_1+a_2-ia_3+ia_4, &(59) & (i+1)a_2-ia_3-a_4,\\
(20) & a_1-ia_2+(-i-1)a_3, &  (40) & a_1+a_2+ia_3-ia_4,   & (60) & (i+1)a_2-ia_3+a_4.
\end{array}$$
\normalsize

\medskip
The five polynomials from Proposition \ref{Kummer-equation} are given by the following formulas:

$$\begin{array}{rcl}
A(a) &=& 4a_1a_2a_3a_4(a_1^4-a_2^4)(a_3^4-a_4^4), \\
 & & \\
B(a) & = & -1/4a_1^{10}a_3^2 + 5/4a_1^8a_2^2a_3^2 - 5/2a_1^6a_2^4a_3^2 + 5/2a_1^4a_2^6a_3^2 - 5/4a_1^2a_2^8a_3^2 + 1/4a_2^{10}a_3^2 \\ & & + 1/2a_1^6a_3^6 + 5/2a_1^4a_2^2a_3^6 - 5/2a_1^2a_2^4a_3^6 - 1/2a_2^6a_3^6 - 1/4a_1^2a_3^{10} + 1/4a_2^2a_3^{10} + 1/4a_1^{10}a_4^2 \\ 
& & - 5/4a_1^8a_2^2a_4^2 + 5/2a_1^6a_2^4a_4^2 - 5/2a_1^4a_2^6a_4^2 + 5/4a_1^2a_2^8a_4^2 - 1/4a_2^{10}a_4^2 + 5/2a_1^6a_3^4a_4^2 \\ & & + 25/2a_1^4a_2^2a_3^4a_4^2 - 25/2a_1^2a_2^4a_3^4a_4^2 - 5/2a_2^6a_3^4a_4^2 + 5/4a_1^2a_3^8a_4^2 - 5/4a_2^2a_3^8a_4^2 - 5/2a_1^6a_3^2a_4^4 \\ & & - 25/2a_1^4a_2^2a_3^2a_4^4 + 25/2a_1^2a_2^4a_3^2a_4^4 + 5/2a_2^6a_3^2a_4^4 - 5/2a_1^2a_3^6a_4^4 + 5/2a_2^2a_3^6a_4^4 - 1/2a_1^6a_4^6 \\ & & - 5/2a_1^4a_2^2a_4^6 + 5/2a_1^2a_2^4a_4^6 + 1/2a_2^6a_4^6 + 5/2a_1^2a_3^4a_4^6 - 5/2a_2^2a_3^4a_4^6 - 5/4a_1^2a_3^2a_4^8 \\ & & + 5/4a_2^2a_3^2a_4^8   + 1/4a_1^2a_4^{10} - 1/4a_2^2a_4^{10}, \\ 
 & & \\
C(a) &=& 1/2a_1^{10}a_3a_4 + 5/2a_1^8a_2^2a_3a_4 + 5a_1^6a_2^4a_3a_4 + 5a_1^4a_2^6a_3a_4 + 5/2a_1^2a_2^8a_3a_4 + 1/2a_2^{10}a_3a_4\\ & &  + 2a_1^6a_3^5a_4 - 10a_1^4a_2^2a_3^5a_4 - 10a_1^2a_2^4a_3^5a_4 + 2a_2^6a_3^5a_4 + 2a_1^6a_3a_4^5 - 10a_1^4a_2^2a_3a_4^5 \\ & & - 10a_1^2a_2^4a_3a_4^5 + 2a_2^6a_3a_4^5 + 8a_1^2a_3^5a_4^5 + 8a_2^2a_3^5a_4^5, \\ 
 & & \\
D(a) &=& 
-8a_1^5a_2^5a_3^2 - 2a_1^5a_2a_3^6 - 2a_1a_2^5a_3^6 - 1/2a_1a_2a_3^{10} - 8a_1^5a_2^5a_4^2 + 10a_1^5a_2a_3^4a_4^2 + 10a_1a_2^5a_3^4a_4^2\\ & &  - 5/2a_1a_2a_3^8a_4^2 + 10a_1^5a_2a_3^2a_4^4 + 10a_1a_2^5a_3^2a_4^4 - 5a_1a_2a_3^6a_4^4 - 2a_1^5a_2a_4^6 - 2a_1a_2^5a_4^6\\ & &  - 5a_1a_2a_3^4a_4^6 - 5/2a_1a_2a_3^2a_4^8 - 1/2a_1a_2a_4^{10}, \\ 
 & & \\
E(a) &=& -1/2a_1^{10}a_2^2 + a_1^6a_2^6 - 1/2a_1^2a_2^{10} - 5/8a_1^8a_3^4 + 5/2a_1^6a_2^2a_3^4 + 25/4a_1^4a_2^4a_3^4 + 5/2a_1^2a_2^6a_3^4\\ & &  - 5/8a_2^8a_3^4 + 5/8a_1^4a_3^8 + 5/4a_1^2a_2^2a_3^8 + 5/8a_2^4a_3^8 - 5/4a_1^8a_3^2a_4^2 - 15a_1^6a_2^2a_3^2a_4^2 \\ & & + 25/2a_1^4a_2^4a_3^2a_4^2 - 15a_1^2a_2^6a_3^2a_4^2 - 5/4a_2^8a_3^2a_4^2 - 5/2a_1^4a_3^6a_4^2 + 15a_1^2a_2^2a_3^6a_4^2 - 5/2a_2^4a_3^6a_4^2\\ & & + 1/2a_3^{10}a_4^2 - 5/8a_1^8a_4^4 + 5/2a_1^6a_2^2a_4^4 + 25/4a_1^4a_2^4a_4^4 + 5/2a_1^2a_2^6a_4^4 - 5/8a_2^8a_4^4\\ & & - 25/4a_1^4a_3^4a_4^4 - 25/2a_1^2a_2^2a_3^4a_4^4 - 25/4a_2^4a_3^4a_4^4 - 5/2a_1^4a_3^2a_4^6 + 15a_1^2a_2^2a_3^2a_4^6 - 5/2a_2^4a_3^2a_4^6\\  & & - a_3^6a_4^6 + 5/8a_1^4a_4^8 + 5/4a_1^2a_2^2a_4^8 + 5/8a_2^4a_4^8 + 1/2a_3^2a_4^{10}.
\end{array}$$

\medskip
The fifteen {\it blocks} mentionned in Definition \ref{def-blocks-pentads} and Theorem \ref{blocks} are the following:
$$\begin{array}{rcl}
D_1 & = & \ell_1\ell_2\ell_3\ell_4\ell_5\ell_6\ell_7\ell_8\ell_9\ell_{10}\ell_{11}\ell_{12}, \\
D_2 & = & \ell_{1}\ell_2\ell_{11}\ell_{12}\ell_{45}\ell_{48}\ell_{49}\ell_{52}\ell_{54}\ell_{55}
\ell_{58}\ell_{59}, \\
D_3 & = & \ell_{1}\ell_{2}\ell_{11}\ell_{12}\ell_{46}\ell_{47}\ell_{50}\ell_{51}\ell_{53}\ell_{56}\ell_{57}\ell_{60}, \\
D_4 & = & \ell_{3}\ell_{4}\ell_{9}\ell_{10}\ell_{13}\ell_{14}\ell_{15}\ell_{16}\ell_{25}\ell_{26}\ell_{27}\ell_{28}, \\
D_5 & = & \ell_{3}\ell_{4}\ell_{9}\ell_{10}\ell_{17}\ell_{18}\ell_{19}\ell_{20}\ell_{21}\ell_{22}\ell_{23}\ell_{24}, \\
D_6 & = & \ell_{5}\ell_{6}\ell_{7}\ell_{8}\ell_{29}\ell_{33}\ell_{34}\ell_{35}\ell_{37}\ell_{41}\ell_{42}\ell_{43}, \\
D_7 & = & \ell_{5}\ell_{6}\ell_{7}\ell_{8}\ell_{30}\ell_{31}\ell_{32}\ell_{36}\ell_{38}\ell_{39}\ell_{40}\ell_{44}, \\
D_8 & = & \ell_{18}\ell_{19}\ell_{21}\ell_{24}\ell_{31}\ell_{32}\ell_{38}\ell_{44}\ell_{47}\ell_{50}\ell_{53}\ell_{60}, \\
D_9 & = & \ell_{14}\ell_{15}\ell_{25}\ell_{28}\ell_{33}\ell_{34}\ell_{37}\ell_{43}\ell_{47}\ell_{50}\ell_{53}\ell_{60}, \\
D_{10} & = & \ell_{18}\ell_{19}\ell_{21}\ell_{24}\ell_{29}\ell_{35}\ell_{41}\ell_{42}\ell_{48}\ell_{49}\ell_{54}\ell_{59}, \\
D_{11} & = & \ell_{14}\ell_{15}\ell_{25}\ell_{28}\ell_{30}\ell_{36}\ell_{39}\ell_{40}\ell_{48}\ell_{49}\ell_{54}\ell_{59}, \\
D_{12} & = & \ell_{13}\ell_{16}\ell_{26}\ell_{27}\ell_{31}\ell_{32}\ell_{38}\ell_{44}\ell_{45}\ell_{52}\ell_{55}\ell_{58}, \\
D_{13} & = & \ell_{17}\ell_{20}\ell_{22}\ell_{23}\ell_{33}\ell_{34}\ell_{37}\ell_{43}\ell_{45}\ell_{52}\ell_{55}\ell_{58}, \\
D_{14} & = & \ell_{13}\ell_{16}\ell_{26}\ell_{27}\ell_{29}\ell_{35}\ell_{41}\ell_{42}\ell_{46}\ell_{51}\ell_{56}\ell_{57}, \\
D_{15} & = & \ell_{17}\ell_{20}\ell_{22}\ell_{23}\ell_{30}\ell_{36}\ell_{39}\ell_{40}\ell_{46}\ell_{51}\ell_{56}\ell_{57}.
\end{array}$$

\medskip
The expressions of  $D_1,\ldots, D_{15}$ in terms of the polynomials $A,\ldots, E$ are given by the following formulas:
$$\begin{array}{rclrcl}
D_1 &=& A/4, & D_{8} & =&  4iA+2iB+2C-2D-2E, \\ 
D_{2} & =& -iA/4-D/8, \qquad & 
D_{9} & =&  -4iA+2iB+2C+2D-2E, \\
D_3 & =&  iA/4-D/8, & 
D_{10} & =&  -4iA+2iB-2C-2D-2E, \\
D_{4} & =& 2C-4iA, & 
D_{11} & =&  4iA+2iB-2C+2D-2E, \\ 
D_{5} & =&  2C+4iA, &   
D_{12} & =&  -4iA-2iB+2C-2D-2E, \\ 
D_{6} & =&  8A-4B, & 
D_{13} & =&  4iA-2iB+2C+2D-2E, \\
D_{7} & =&  -8A-4B, & 
D_{14} & =&  4iA-2iB-2C-2D-2E, \\ 
 & & & 
D_{15} & =&   -4iA-2iB-2C+2D-2E.
\end{array}$$

The six {\it pentads} of blocks  mentionned in Definition \ref{def-blocks-pentads} and Theorem \ref{blocks} are:
$$\begin{array}{rcl}
P_1 &=& D_{1} - D_{8} - D_{11} - D_{13} - D_{14}, \\ 
P_2 &=& D_{1} - D_{9} - D_{10} - D_{12} - D_{15}, \\ 
P_3 &=& D_{2} - D_{4} - D_{6} - D_{8} - D_{15}, \\ 
P_4 &=& D_{2} - D_{5} - D_{7} - D_{9} - D_{14}, \\ 
P_5 &=& D_{3} - D_{4} - D_{7} - D_{10} - D_{13},  \\ 
P_6 &=& D_{3} - D_{5} - D_{6} - D_{11} - D_{12}.
\end{array}$$

A direct examination shows that two blocks that are not disjoint share four hyperplanes, forming  a {\it tetrahedron}. There are fifteen such tetrahedra, each contained in three blocks. We list them below as $Q_1,\ldots , Q_{15}$, each with the blocks that contain them. 

$$\begin{array}{lll}
Q_{1} & \ell_{1}\ell_{2}\ell_{11}\ell_{12} & D_{1}D_{2}D_{3} \\
Q_{2} & \ell_{3}\ell_{4}\ell_{9}\ell_{10} & D_{1}D_{4}D_{5} \\
Q_{3} & \ell_{5}\ell_{6}\ell_{7}\ell_{8} & D_{1}D_{6}D_{7} \\
Q_{4} & \ell_{48}\ell_{49}\ell_{54}\ell_{59} & D_{2}D_{10}D_{11} \\
Q_{5} & \ell_{45}\ell_{52}\ell_{55}\ell_{58} & D_{2}D_{12}D_{13} \\
Q_{6} & \ell_{47}\ell_{50}\ell_{53}\ell_{60} & D_{3}D_{8}D_{9} \\
Q_{7} & \ell_{46}\ell_{51}\ell_{56}\ell_{57} & D_{3}D_{14}D_{15} \\
Q_{8} & \ell_{14}\ell_{15}\ell_{25}\ell_{28} & D_{4}D_{9}D_{11} \\
Q_{9} & \ell_{13}\ell_{16}\ell_{26}\ell_{27} & D_{4}D_{12}D_{14} \\
Q_{10} & \ell_{18}\ell_{19}\ell_{21}\ell_{24} & D_{5}D_{8}D_{10} \\
Q_{11} & \ell_{17}\ell_{20}\ell_{22}\ell_{23} & D_{5}D_{13}D_{15} \\
Q_{12} & \ell_{33}\ell_{34}\ell_{37}\ell_{43} & D_{6}D_{9}D_{13} \\
Q_{13} & \ell_{29}\ell_{35}\ell_{41}\ell_{42} & D_{6}D_{10}D_{14} \\
Q_{14} & \ell_{31}\ell_{32}\ell_{38}\ell_{44} & D_{7}D_{8}D_{12} \\
Q_{15}\qquad & \ell_{30}\ell_{36}\ell_{39}\ell_{40}\qquad & D_{7}D_{11}D_{15} 
\end{array}$$

\medskip
The edges and vertices of these tetrahedra consist in $30$ lines and $60$ points: these are the flats of type 2 and 5 in Table \ref{flats}. With the $60$ reflection hyperplanes, these $60$ points recover Klein's $60_{15}$ configuration \cite{pokora}.  More details on the $30$ lines and their relations with the so-called fundamental quadrics can be found in \cite{cheltsov-shramov}. 

\smallskip
We note that the three tetrahedra coming from a same block have their vertices spread on two lines only, two for each line, whose union are the six special points on the line. These two lines belong to the same fundamental quadric. In particular the fifteen blocks split the $30$ lines into $15$ pairs. Moreover each pentad distinguishes ten lines. 
We  get a basis of the space of quartics $U'_5$ by chosing one tetrahedra in each block of a given pentad. 

\smallskip
Finally, the six quintets mentionned in Lemma \ref{totals} are, numbered  from left to right:

$$\begin{matrix} 1&2&3\\ 4&12&14\\ 5&8&10\\6&9&13\\7&11&15\end{matrix} \qquad 
\begin{matrix} 1&2&3\\ 4&9&11\\ 5&13&15\\6&10&14\\7&8&12\end{matrix} \qquad
\begin{matrix} 1&4&5\\ 2&12&13\\ 3&8&9\\6&10&14\\7&11&15\end{matrix} \qquad
\begin{matrix} 1&4&5\\ 2&10&11\\ 3&14&15\\6&9&13\\7&8&12\end{matrix} \qquad
\begin{matrix} 1&6&7\\ 2&12&13\\ 3&14&15\\4&9&11\\5&8&10\end{matrix}\qquad  
\begin{matrix} 1&6&7\\ 2&10&11\\ 3&8&9\\4&12&14\\5&13&15\end{matrix}$$

\medskip
We summarize the combinatorics in the following table.  The exterior automorphism of $S_6$ swaps the rows from top to bottom.
$$\begin{array}{cccc}
 \quad 6\quad & \mathrm{points} & &\mathrm{pentads} \\
  15 & \mathrm{pairs} & &\mathrm{blocks} \\
  15 & \quad \mathrm{synthemes}\quad  & &\mathrm{tetrahedra} \\
  6 & \mathrm{totals} & &\mathrm{quintets} 
  \end{array} $$
  
  \medskip

\bibliography{SpinorSections}

\providecommand{\bysame}{\leavevmode\hbox to3em{\hrulefill}\thinspace}
\providecommand{\MR}{\relax\ifhmode\unskip\space\fi MR }
\providecommand{\MRhref}[2]{%
  \href{http://www.ams.org/mathscinet-getitem?mr=#1}{#2}
}
\providecommand{\href}[2]{#2}
\begin{thebibliography}{BBFM25}

\bibitem[BBFM24]{coble-quadric}
V.~Benedetti, M.~Bolognesi, D.~Faenzi, and L.~Manivel, \emph{The {C}oble
  quadric}, Forum Math. Sigma \textbf{12} (2024), Paper No. e63, 25.

\bibitem[BBFM25]{gopel}
\bysame, \emph{G\"opel varieties}, in preparation, 2025.

\bibitem[BFM20]{baifuman}
Ch. Bai, B.~Fu, and L.~Manivel, \emph{On {F}ano complete intersections in
  rational homogeneous varieties}, Math. Z. \textbf{295} (2020), no.~1-2,
  289--308.

\bibitem[Cat23]{catanese}
F.~Catanese, \emph{Kummer quartic surfaces, strict self-duality, and more}, The
  art of doing algebraic geometry, Trends Math., Birkh\"{a}user, 2023,
  pp.~55--92.

\bibitem[CES23]{JorVin2}
G.~Carnovale, F.~Esposito, and A.~Santi, \emph{On {Jordan} classes for
  {Vinberg}'s {{\(\theta\)}}-group}, Transform. Groups \textbf{28} (2023),
  no.~1, 151--183.

\bibitem[Cha56]{charmichael}
R.D. Charmichael, \emph{Introduction to the theory of groups of finite order},
  Dover, 1956.

\bibitem[Che97]{chevalley}
C.~Chevalley, \emph{The algebraic theory of spinors and {C}lifford algebras},
  Springer, 1997, Collected works. Vol. 2.

\bibitem[Ciu12]{ciubotaru}
D.~Ciubotaru, \emph{Spin representations of {W}eyl groups and the {S}pringer
  correspondence}, J. Reine Angew. Math. \textbf{671} (2012), 199--222.

\bibitem[CS19]{cheltsov-shramov}
I.~Cheltsov and C.~Shramov, \emph{Finite collineation groups and birational
  rigidity}, Selecta Math. (N.S.) \textbf{25} (2019), no.~5, Paper No. 71, 68.

\bibitem[dG25]{degraaf}
Willem~A. de~Graaf, \emph{A classification of four-tuples of spinors of a ten
  dimensional space}, preprint, 2025.

\bibitem[DM22]{dedieu-manivel}
Th. Dedieu and L.~Manivel, \emph{On the automorphisms of {M}ukai varieties},
  Math. Z. \textbf{300} (2022), no.~4, 3577--3621.

\bibitem[Dol04]{dolg-config}
I.V. Dolgachev, \emph{Abstract configurations in algebraic geometry}, The
  {F}ano {C}onference, Univ. Torino, 2004, pp.~423--462.

\bibitem[Dol12]{Dolg_classical}
\bysame, \emph{Classical algebraic geometry. {A} modern view}, Cambridge, 2012.

\bibitem[FH20]{Euler_sym}
B.~Fu and J.-M. Hwang, \emph{Euler-symmetric projective varieties}, Algebr.
  Geom. \textbf{7} (2020), no.~3, 377--389.

\bibitem[FN89]{fn}
M.~Furushima and N.~Nakayama, \emph{The family of lines on the {F}ano threefold
  {$V_5$}}, Nagoya Math. J. \textbf{116} (1989), 111--122.

\bibitem[GD94]{gd}
M.R. Gonzalez-Dorrego, \emph{{$(16,6)$} configurations and geometry of {K}ummer
  surfaces in {${\bf P}^3$}}, Mem. Amer. Math. Soc. \textbf{107} (1994),
  no.~512, vi+101.

\bibitem[GGI16]{ggi}
S.~Galkin, V.~Golyshev, and H.~Iritani, \emph{Gamma classes and quantum
  cohomology of {F}ano manifolds: gamma conjectures}, Duke Math. J.
  \textbf{165} (2016), no.~11, 2005--2077.

\bibitem[GH94]{GH}
Ph. Griffiths and J.~Harris, \emph{Principles of algebraic geometry}, Wiley
  Classics Library, Wiley, 1994.

\bibitem[GKM98]{gkm}
M.~Goresky, R.~Kottwitz, and R.~MacPherson, \emph{Equivariant cohomology,
  {K}oszul duality, and the localization theorem}, Invent. Math. \textbf{131}
  (1998), no.~1, 25--83.

\bibitem[GSW13]{GSW}
Laurent Gruson, Steven~V. Sam, and Jerzy Weyman, \emph{Moduli of abelian
  varieties, {V}inberg {$\theta$}-groups, and free resolutions}, Commutative
  algebra, Springer, New York, 2013, pp.~419--469.

\bibitem[HMSV08]{syntheme}
B.~Howard, J.~Millson, A.~Snowden, and R.~Vakil, \emph{A description of the
  outer automorphism of {$S_6$}, and the invariants of six points in projective
  space}, J. Combin. Theory Ser. A \textbf{115} (2008), no.~7, 1296--1303.

\bibitem[HNS00]{hulek-heisenberg}
K.~Hulek, I.~Nieto, and G.~K. Sankaran, \emph{Heisenberg-invariant {K}ummer
  surfaces}, Proc. Edinburgh Math. Soc. (2) \textbf{43} (2000), no.~2,
  425--439.

\bibitem[HS02]{hulek-sankaran}
K.~Hulek and G.K. Sankaran, \emph{The geometry of {S}iegel modular varieties},
  Higher dimensional birational geometry ({K}yoto, 1997), Adv. Stud. Pure
  Math., vol.~35, Math. Soc. Japan, Tokyo, 2002, pp.~89--156.

\bibitem[Hud90]{hudson}
R.W.H.T. Hudson, \emph{Kummer's quartic surface}, Cambridge, 1990, Reprint of
  the 1905 original.

\bibitem[Hun96]{hunt}
B.~Hunt, \emph{The geometry of some special arithmetic quotients}, Lecture
  Notes in Mathematics, vol. 1637, Springer, 1996.

\bibitem[KO82]{orbites_SL3Spin10}
T.~Kimura and I.~Ozeki, \emph{On the microlocal structure of a regular
  prehomogeneous vector space associated with {${\rm Spin}(10)\times {\rm
  GL}(3)$}}, Proc. Japan Acad. Ser. A Math. Sci. \textbf{58} (1982), no.~6,
  239--242.

\bibitem[Kon16]{Igusa16}
S.~Kond{\=o}, \emph{The {Igusa} quartic and {Borcherds} products}, K3 surfaces
  and their moduli, Springer, 2016, pp.~147--170.

\bibitem[Kuz18]{kuz_spin}
A.~G. Kuznetsov, \emph{On linear sections of the spinor tenfold. {I}}, Izv.
  Ross. Akad. Nauk Ser. Mat. \textbf{82} (2018), no.~4, 53--114.

\bibitem[KW11]{KWE8}
W.~Kraskiewicz and J.~Weyman, \emph{Geometry of orbit closures for the
  representations associated to gradings of {L}ie algebras of types $e_8$},
  preprint, 2011.

\bibitem[KW12]{KWE6}
\bysame, \emph{Geometry of orbit closures for the representations associated to
  gradings of lie algebras of types $e_6$, $f_4$ and $g_2$}, arXiv e-print
  math.AG/1201.1102, 2012.

\bibitem[KW13]{KWE7}
\bysame, \emph{Geometry of orbit closures for the representations associated to
  gradings of lie algebras of types $e_7$}, arXiv e-print math.AG/1301.0720,
  2013.

\bibitem[LM03]{LM-linear}
J.M. Landsberg and L.~Manivel, \emph{On the projective geometry of rational
  homogeneous varieties}, Comment. Math. Helv. \textbf{78} (2003), no.~1,
  65--100.

\bibitem[Mac72]{macdonald}
I.G. Macdonald, \emph{Some irreducible representations of {W}eyl groups}, Bull.
  London Math. Soc. \textbf{4} (1972), 148--150.

\bibitem[Man19]{doublespinor}
L.~Manivel, \emph{Double spinor {C}alabi-{Y}au varieties}, \'{E}pijournal
  G\'{e}om. Alg\'{e}brique \textbf{3} (2019), Art. 2, 14.

\bibitem[MM07]{morris-mwamba}
A.O. Morris and P.~Mwamba, \emph{Macdonald representations of complex
  reflection groups}, S\'{e}m. Lothar. Combin. \textbf{52} (2004/07), Art.
  B52g, 17.

\bibitem[Muk88]{mukai-curves}
S.~Mukai, \emph{Curves, {$K3$} surfaces and {F}ano {$3$}-folds of genus {$\leq
  10$}}, Algebraic geometry and commutative algebra, {V}ol. {I}, Kinokuniya,
  Tokyo, 1988, pp.~357--377.

\bibitem[Muk92]{mukai-fano3folds}
\bysame, \emph{Fano {$3$}-folds}, Complex projective geometry ({T}rieste,
  1989/{B}ergen, 1989), London Math. Soc. Lecture Note Ser., vol. 179,
  Cambridge, 1992, pp.~255--263.

\bibitem[Muk12]{mukai-igusa}
\bysame, \emph{Igusa quartic and {S}teiner surfaces}, Compact moduli spaces and
  vector bundles, Contemp. Math., vol. 564, Amer. Math. Soc., 2012,
  pp.~205--210.

\bibitem[New81]{newstead}
P.~E. Newstead, \emph{Invariants of pencils of binary cubics}, Math. Proc.
  Cambridge Philos. Soc. \textbf{89} (1981), no.~2, 201--209.

\bibitem[Ott88]{ott-spinor}
G.~Ottaviani, \emph{Spinor bundles on quadrics}, Trans. Amer. Math. Soc.
  \textbf{307} (1988), no.~1, 301--316.

\bibitem[Pas09]{pasquier}
Boris Pasquier, \emph{On some smooth projective two-orbit varieties with
  {P}icard number 1}, Math. Ann. \textbf{344} (2009), no.~4, 963--987.

\bibitem[Pop18]{JorVin1}
V.L. Popov, \emph{Modality of representations, and packets for
  {{\(\theta\)}}-groups}, Lie groups, geometry, and representation theory.,
  Birkh{\"a}user, 2018, pp.~459--479.

\bibitem[PSS24]{pokora}
P.~Pokora, T.~Szemberg, and J.~Szpond, \emph{Unexpected properties of the
  {K}lein configuration of 60 {P}oints in {$\Bbb P^3$}}, Michigan Math. J.
  \textbf{74} (2024), no.~3, 599--615.

\bibitem[Rei72]{reid}
M.~Reid, \emph{The complete intersections of two or more quadrics},
  unpublished, 1972.

\bibitem[SK77]{sk}
M.~Sato and T.~Kimura, \emph{A classification of irreducible prehomogeneous
  vector spaces and their relative invariants}, Nagoya Math. J. \textbf{65}
  (1977), 1--155.

\bibitem[ST54]{st}
G.C. Shephard and J.A. Todd, \emph{Finite unitary reflection groups}, Canad. J.
  Math. \textbf{6} (1954), 274--304.

\bibitem[Vin76]{vinberg}
E.B. Vinberg, \emph{The {W}eyl group of a graded {L}ie algebra.}, Izv. Akad.
  Nauk SSSR Ser. Mat. (1976), no.~no. 3,, 488--526, 709.

\end{thebibliography}

\bibliographystyle{amsalpha}

\end{document}